\documentclass[11pt, reqno]{amsart}

\usepackage{amsmath,amsfonts,amsthm,amssymb,amscd,mathtools,stmaryrd}
\usepackage{epsfig}
\usepackage{inputenc}
\usepackage{upref}
\usepackage{bm}
\usepackage{calc}
\usepackage[justification=Centering]{caption}
\usepackage{comment}
\usepackage{subcaption}
\usepackage[pagebackref]{hyperref}
\usepackage[nameinlink]{cleveref}
\usepackage[usenames,dvipsnames]{xcolor}
\usepackage[normalem]{ulem}%
\usepackage{wasysym}
\usepackage{mathrsfs}
\usepackage{enumitem}
\usepackage{setspace}
\usepackage{array}
\usepackage{cancel}
\usepackage{tikz,tikz-cd}%
\usepackage{colonequals}
\usepackage[titletoc]{appendix}
\usepackage{overpic}
\usepackage{float} 
\usepackage{xcolor}
\usepackage{relsize}

\hypersetup{
	colorlinks=true,
	linktoc=all,%
	linkcolor=RoyalBlue,
	filecolor=blue,
	citecolor=ForestGreen,%
	urlcolor=blue,
}

\usepackage{stackengine,xcolor,scalerel}
\def\cpm{\mathbin{\ThisStyle{\ensurestackMath{\abovebaseline[-\dimexpr1pt+2.4\LMpt]{%
					\stackunder[-\dimexpr1pt+2.5\LMpt]{\color{blue}\SavedStyle+}{%
						\color{red}\SavedStyle-}}}}}}
\def\cmp{\mathbin{\ThisStyle{\ensurestackMath{\abovebaseline[-\dimexpr1.1pt+0.55\LMpt]{%
					\stackunder[-\dimexpr1pt+2.5\LMpt]{\color{blue}\SavedStyle-}{%
						\color{red}\SavedStyle+}}}}}}

\newcommand{\lowerromannumeral}[1]{(\romannumeral#1\relax)}

\newtheorem{theorem}{Theorem}[section]
\newtheorem*{theorem*}{Theorem}	

\newtheorem{corollary}[theorem]{Corollary}
\newtheorem{lemma}[theorem]{Lemma}
\newtheorem{proposition}[theorem]{Proposition}

\newtheorem{question}[theorem]{Question}
\newtheorem{claim}[theorem]{Claim}

\theoremstyle{definition}

\newtheorem{definition}[theorem]{Definition}
\newenvironment{defx}
{\pushQED{\qed}\definition}
{\popQED\enddefinition}

\newtheorem{examplex}[theorem]{Example}
\newenvironment{example}
{\pushQED{\qed}\examplex}
{\popQED\endexamplex}

\newtheorem{remark}[theorem]{Remark}
\newenvironment{rmk}
{\pushQED{\qed}\remark}
{\popQED\endremark}

\numberwithin{equation}{section}

\DeclareMathOperator{\im}{Im}

\DeclareMathOperator{\Id}{Id}

\DeclareMathOperator{\ad}{ad}
\DeclareMathOperator{\Ad}{Ad}
\DeclareMathOperator{\hol}{hol}
\DeclareMathOperator{\Hom}{Hom}

\DeclareMathOperator{\tr}{tr}
\DeclareMathOperator{\Stab}{Stab}

\DeclareMathOperator{\so}{SO(3)}

\DeclareMathOperator{\coker}{coker}

\DeclareMathOperator{\grad}{grad}

\newcommand{\R}{\mathbb{R}}

\newcommand{\Q}{\mathbb{Q}}

\newcommand{\Z}{\mathbb{Z}}

\newcommand{\CP}{\mathbb{CP}}
\newcommand{\RP}{\mathbb{RP}}

\newcommand{\BF}{\mathbb{F}}

\newcommand{\supp}{\text{supp}}

\newcommand{\CB}{\mathcal{B}}
\newcommand{\SC}{\mathcal{C}}

\newcommand{\D}{\mathcal{D}}

\newcommand{\CG}{\mathcal{G}}

\newcommand{\M}{\mathcal{M}}

\newcommand{\Pa}{\mathcal{P}}

\newcommand{\U}{\mathcal{U}}

\newcommand{\G}{\mathfrak{G}}
\newcommand{\g}{\mathfrak{g}}

\newcommand{\fo}{\mathfrak{o}}

\newcommand{\fri}{\mathfrak{I}}
\newcommand{\fv}{\mathfrak{v}}

\newcommand{\rk}{\mathrm{rk}}

\newcommand{\ind}{\ensuremath{\mathrm{Ind}}}

\let\emptyset\varnothing

\newcommand{\yk}{Y \setminus \nu (K)}
\newcommand{\pp}{\boldsymbol{P}}
\def\<#1>{\mathinner{\langle#1\rangle}}
\def\|#1|{\mathinner{\lVert #1 \rVert}}

\newcommand{\fc}{\mathfrak{C}}
\newcommand{\pd}[1][i]{\partial_{#1}}
\newcommand{\conk}[1][k]{\mathcal{C}_{#1}(Y, K, \pp)}
\newcommand{\lp}[3][k]{\check{L}^2_{#1} (\check{#3}, \mathfrak{g}_{\check{P}}\otimes {#2})}

\newcommand{\conw}[3][k]{\mathcal{C}_{#1, \alpha}(#2, #3, \pp)}
\newcommand{\pR}{[0, +\infty)}
\def\poinc{\text{Poincar\'e }}
\newcommand{\GK}{\Gamma_{L \setminus K}}

\let\emptyset\varnothing

\title{Surgery Exact Triangles in Instanton Theory}
\author{Deeparaj Bhat}

\address{Department of Mathematics, Columbia University, New York, NY 10027, USA}
\email{d.bhat@columbia.edu}

\begin{document}
	
	\begin{abstract}
		We prove an exact triangle relating knot instanton Floer homology to the instanton homology of surgeries along the knot. To the author's knowledge this is the first such result in instanton homology with integer coefficients and has no simple analogue in Heegaard Floer homology. To illustrate the latter claim, we derive as a consequence of this triangle, building on previous computations of \cite{sca}, \cite{bs-l-space} and \cite{kmknot}, that the \poinc Homology Sphere is not an instanton L-space with $\Z / 2$-coefficients.
	\end{abstract}
	
	\maketitle

	\section{Introduction}
	
	Instanton homology is a Floer homology theory for a class of closed 3-manifolds due to Andreas Floer in \cite{floer}. Floer also introduced a version for knots in 3-manifolds. Since then, many flavours of the same been constructed. We will be using the constructions of Kronheimer and Mrowka \cite{kmknot} in this paper. While being hard to compute from its definition, Floer introduced a surgery exact triangle (see \Cref{thm:floer-exact-triangle}) which aids in its computation. This exact triangle relates the instanton homology of 3-manifolds obtained by surgeries along a knot. There is an analogous exact triangle (the unoriented skein exact triangle) due to Kronheimer and Mrowka \cite{kmknot} for case of knot instanton homology.
	
	We prove a new surgery exact triangle which relates the instanton homology of surgered 3-manifolds and knot instanton homology. To state the main result, fix a closed 3-manifold $Y$ and a \emph{framed} knot $K \subset Y$. From the constructions in \cite{kmknot}, we have the invariant $I^\#$ for 3-manifolds and knots in 3-manifolds used in the statement below.
	
	\begin{theorem}[see \Cref{thm:main}] \label{thm:intro}
		The following is an exact triangle of $\Z$ modules for any $n \in \Z$:
		\begin{equation*}
			\begin{tikzcd}[column sep=small]
				I^\#(Y_n(K); \Z) \arrow[rr, "\psi"] & & I^\#(Y_{n+2}(K); \Z) \arrow[dl, "f_1"] \\
				& I^\#(Y, K; \Z) \arrow[ul, "f_2"]&
			\end{tikzcd}
		\end{equation*}
		The maps $f_1$ and $f_2$ are defined by cobordism maps and $\psi$ is an integer linear combination of cobordism maps.
	\end{theorem}
	
	The proof of the above theorem uses ideas from the modern recasting of the proofs as in \cite{int-monopole, kmknot}; see also \cite{sca} for a proof of \Cref{thm:floer-exact-triangle} in this framework. The more refined version \Cref{thm:main} also uses the idea of energy filtration as in the original paper by Floer \cite{flo}. While Floer approaches this issue by puncturing the 4-manifolds defining the maps in his triangle, we use a softer approach by instead combining the metric perturbations, developed by Freed and Uhlenbeck \cite{fu}, and holonomy perturbations, developed by Kronheimer and Mrowka \cite{yaft} in the context of singular instantons. This framework and its consequences are developed in \Cref{sec:energy} and \Cref{sec:energy-order} respectively.
	
	We say a little about the analogous situation in Heegaard Floer theory. As far as the author is aware, there is no analogous knot Floer group to $I^\#(Y, K; \Z)$ in Heegaard Floer homology. Further, to the author's knowledge, there is no simple analogue in Heegaard Floer theory that relates knot homology to homology of surgeries. To put this into perspective, we state a corollary which follows from an iterated special case of \Cref{thm:main}.
	
	\begin{corollary}
		Suppose $L \subset Y$ is a framed $k$ component link in a closed 3-manifold $Y$. Then, there exists a spectral sequence whose first page is
		\[
		\bigoplus_{\Lambda \in \{\pm 1\}^k} I^\# (Y_{\Lambda} (L); \Z)
		\]
		that converges to $I^\# (Y, L; \Z)$. \qed
	\end{corollary}
	
	In contrast with the link surgery formulae in Heegaard Floer theory as in \cite{linksurgery}, which computes the Heegaard Floer homology of surgery using the link Floer complex, the above instead computes the link instanton homology from the instanton homology of surgeries.
	
	Further, as an application of the formal algebraic properties of \Cref{thm:intro}, we show that the \poinc Homology Sphere is \emph{not} an instanton $L$-space with $\BF_2$ coefficients (see below for definition) in sharp contrast with the Heegaard Floer homology computation.
	
	\begin{defx}
		Let $\mathbb{K}$ be a field. We say $Y$ is an $L$-space with $\mathbb{K}$ coefficients if $\dim_{\mathbb{K}} I^\# (Y; \mathbb{K}) = \chi (I^\#(Y; \Z))$.
	\end{defx}
	
	\begin{rmk}
		If $p$ is a prime then, $Y$ is an instanton $L$-space with $\mathbb{F}_p$ coefficients if and only if $Y$ is an $L$-space with $\mathbb{Q}$ coefficients and $I^\#(Y; \Z)$ contains no $p$-torsion.
	\end{rmk}
	
	\begin{theorem}\label{thm:l-space}
		Let $P$ denote the \poinc homology sphere. Then, the instanton homology group of $P$ is given by
		\[
		I^\# (P; \Z) = \Z_{(0)} \oplus G_{(1)}
		\]
		where $G \neq 0$ is a 2-group and the subscripts indicate absolute $\Z / 2$ grading.
		
		In particular, $P$ is not an instanton L-space with $\BF_2$ coefficients.
	\end{theorem}
	
	Despite this, we expect the triangle to be compatible with sutured instanton theory where torsion information is lost. Indeed, in ongoing work in progress with Zhenkun Li and Fan Ye, we show a variant of the triangle closely related to \Cref{thm:intro} is compatible with sutured instanton Floer theory.
	
	As one expects 2-torsion in $I^\#(Y, K; \Z)$ (see \cite{torsion}), \Cref{thm:intro} provides a mechanism to relate the 2-torsion of the 3-manifold groups. Further, while we have not computed interesting examples, the more general result \Cref{thm:main} also provides a relation between the groups with local coefficients as developed in \cite{yaft}.
	
	The existence of torsion (and subsequently \Cref{thm:l-space}) is of topological interest due to the following result.
	
	\begin{theorem}[{\cite[Theorem 4.6]{bs-stein}}] \label{thm:irred-torsion}
		Suppose $Y$ is a rational homology sphere with $\pi_1(Y)$ cyclically finite and $I^\#(Y; \Z) \neq \Z^{|H_1(Y; \Z)|}$, then there is an irreducible representation $\pi_1(Y) \to SU(2)$.
	\end{theorem}
	
	\begin{rmk}
		Strictly speaking, \cite[Theorem 4.6]{bs-stein} uses $\Q$ coefficients for the homology groups. However, the proof readily implies the version with $\Z$ coefficients stated above.
	\end{rmk}
	
	In light of the above, we pose the following question.
	
	\begin{question}
		Suppose $Y$ is an integer homology sphere and $Y \ncong S^3$. Does $I^\#(Y; \Z)$ contain 2-torsion?
	\end{question}
	
	Very recent work in \cite{torsion} provides partial evidence towards the positive answer to this question and also partial results regarding 2-torsion in $I^\#(S^3, K; \Z)$ obtained by combining sutured techniques with \Cref{thm:intro}.
	
	\subsection{Motivation}
	\label{sec:motivation}
	
	We briefly use the Atiyah-Floer perspective to explain how one may have arrived at the main result of the paper.
	
	Recall that in the Atiyah-Floer perspective, we can write $I(Y) = LF(\M (H_0), \M (H_1))$ where $H_0 \cup_\Sigma H_1 = Y$ is a handle-body decomposition, $LF$ denotes Lagrangian intersection Floer homology\footnote{\label{ft:lagfloer}This may not be well defined in the context here, but this will not be important for our purposes.} for the ambient symplectic manifold\footnote{\label{ft:manifold} We will ignore the precise definitions of these when there are singularities.} $\M (\Sigma)$, the moduli of flat connections upto gauge transformation and $\M(H_i)$ denote the moduli of flat connections (up to gauge transformation) on $H_i$ are Lagrangian submanifolds\footref{ft:manifold}.
	
	Before we explain the heuristic in the context of \Cref{thm:intro}, we first state and discuss the case of the Floer's exact triangle as stated below.
	
	\begin{theorem}[{Floer \cite{flo}; see also \cite{sca}}] \label{thm:floer-exact-triangle}
		The following is an exact triangle of $\Z$ modules where $n \in \Z$:
		\begin{equation*}
			\begin{tikzcd}[column sep=small]
				I^\#(Y_n(K); \Z)_{\omega_n} \arrow[rr] & & I^\#(Y_{n+1}(K); \Z) \arrow[dl] \\
				& I^\#(Y; \Z) \arrow[ul]&
			\end{tikzcd}
		\end{equation*}\qed
	\end{theorem}
	
	We further relax the condition that $H_i$ are handle-bodies in the Atiyah-Floer picture above: We write $Y = (\yk) \cup_{T^2} \nu (K)$ when $K$ is framed. Note that $\M (T^2) \cong T^2 / (\Z/ 2)$ which is homeomorphic to $S^2$ as a topological space with four singular points, called the pillowcase. Any surgery can be obtained by modifying $\M(\nu (K))$ by post composing by the relevant diffeomorphism of $T^2$ which extends to $D^2 \times S^1 \cong \nu (K)$. Indeed, we can write $I(Y_{p/q}(K)) = LF(\M(\yk), L_{p/q})$ for appropriate Lagrangian arcs $L_{p/q}$. Under this interpretation, \Cref{thm:floer-exact-triangle} follows from the \cite{fooo}. See also \cite{seidel} for the case when one of the lagrangian is a circle. \Cref{fig:floer} illustrates this where the ambient manifold is $\M(T^2) \cong S^2$ (with 4 singular points) drawn as a pillowcase and the relevant Lagrangians are in blue.
	
	\begin{figure}[!h]
		\tikzset{every picture/.style={line width=0.75pt}} %
		\begin{tikzpicture}[x=0.5pt,y=0.5pt,yscale=-1,xscale=1]
			\clip (100,10) rectangle (400, 280);
			\draw   (131,22) -- (350.45,22) -- (350.45,241.45) -- (131,241.45) -- cycle ;
			\draw    (131.45,131.05) .. controls (186.45,150.05) and (295.45,149.05) .. (349.45,131.05) ;
			\draw  [dash pattern={on 4.5pt off 4.5pt}]  (131.45,131.05) .. controls (189.45,111.05) and (290.45,111.05) .. (349.45,131.05) ;
			\draw [color={rgb, 255:red, 0; green, 114; blue, 255 }  ,draw opacity=1 ][line width=1.5]    (131,241.45) -- (350.45,241.45) ;
			\draw [shift={(250.72,241.45)}, rotate = 180] [fill={rgb, 255:red, 0; green, 114; blue, 255 }  ,fill opacity=1 ][line width=0.08]  [draw opacity=0] (15.6,-3.9) -- (0,0) -- (15.6,3.9) -- cycle    ;
			\draw [color={rgb, 255:red, 2; green, 117; blue, 253 }  ,draw opacity=1 ][line width=1.5]    (350.45,22) -- (350.45,241.45) ;
			\draw [shift={(350.45,131.72)}, rotate = 90] [fill={rgb, 255:red, 2; green, 117; blue, 253 }  ,fill opacity=1 ][line width=0.08]  [draw opacity=0] (15.6,-3.9) -- (0,0) -- (15.6,3.9) -- cycle    ;
			\draw [color={rgb, 255:red, 0; green, 120; blue, 255 }  ,draw opacity=1 ][line width=1.5]    (131,241.45) -- (350.45,22) ;
			\draw [shift={(247.8,124.65)}, rotate = 135] [fill={rgb, 255:red, 0; green, 120; blue, 255 }  ,fill opacity=1 ][line width=0.08]  [draw opacity=0] (15.6,-3.9) -- (0,0) -- (15.6,3.9) -- cycle    ;
			
			\draw (111.42,-224.12) node [anchor=north west][inner sep=0.75pt]  [rotate=-46.88] [align=left] {};
			\draw (235,251.4) node [anchor=north west][inner sep=0.75pt]  [color={rgb, 255:red, 0; green, 119; blue, 255 }  ,opacity=1 ]  {$L_0^{\omega_0}$};
			\draw (250,120) node [anchor=north west][inner sep=0.75pt]   [align=left] {$\textcolor[rgb]{0,0.46,1}{L_{1}}$};
			\draw (363,131.4) node [anchor=north west][inner sep=0.75pt]    {$\textcolor[rgb]{0,0.47,1}{L_\infty }$};
		\end{tikzpicture}
		\caption{Atiyah-Floer perspective on \Cref{thm:floer-exact-triangle} when $n = 0$.}
		\label{fig:floer}
	\end{figure}
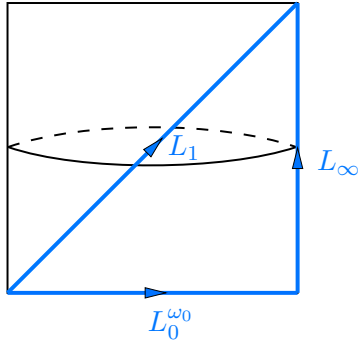
	
	Now, a reasonable guess for $I(Y, K)$, guided by this viewpoint, is $LF(\M(\yk), L_{\tr})$ where $L_{\tr}$ denotes the Lagrangian loop corresponding to the meridional holonomy being traceless. Note that the Lagrangian corresponding to 0-surgery, $L_0$, and $L_{\tr}$ intersect in a single point, denote it by $p$. Then, one notes that $L_{2} \simeq L_{0} \# L_{\tr} \simeq \tau_{L_{\tr}}(L_0)$ whence from \cite{seidel} we know that the following is a distinguished triangle:
	
	\begin{equation*}
		\begin{tikzcd}[column sep=small]
			L_0 \arrow[rr, "\cdotp p"] & & L_{\tr} \arrow[dl, "\cdot q"] \\
			& L_2 \arrow[ul]&
		\end{tikzcd}
	\end{equation*}
	where $q$ denotes the intersection point of $L_2$ and $L_{\tr}$ (see \Cref{fig:atiyah-floer}).
	
	\begin{figure}[!h]
		\tikzset{every picture/.style={line width=0.75pt}} %
		\begin{tikzpicture}[x=0.60pt,y=0.60pt,yscale=-1,xscale=1]
			\clip (100,0) rectangle (400, 250);
			\draw   (131,22) -- (350.45,22) -- (350.45,241.45) -- (131,241.45) -- cycle ;
			\draw    (131.45,131.05) .. controls (186.45,150.05) and (295.45,149.05) .. (349.45,131.05) ;
			\draw  [dash pattern={on 4.5pt off 4.5pt}]  (131.45,131.05) .. controls (189.45,111.05) and (290.45,111.05) .. (349.45,131.05) ;
			\draw [color={rgb, 255:red, 0; green, 114; blue, 255 }  ,draw opacity=1 ][line width=1.5]    (131,22) -- (350.45,22) ;
			\draw [shift={(240.72,22)}, rotate = 180] [fill={rgb, 255:red, 0; green, 114; blue, 255 }  ,fill opacity=1 ][line width=0.08]  [draw opacity=0] (15.6,-3.9) -- (0,0) -- (15.6,3.9) -- cycle    ;
			\draw [color={rgb, 255:red, 0; green, 116; blue, 255 }  ,draw opacity=1 ][line width=1.5]    (131,22) -- (239.66,240.75) ;
			\draw [shift={(189.78,140.33)}, rotate = 243.58] [fill={rgb, 255:red, 0; green, 116; blue, 255 }  ,fill opacity=1 ][line width=0.08]  [draw opacity=0] (15.6,-3.9) -- (0,0) -- (15.6,3.9) -- cycle    ;
			\draw [color={rgb, 255:red, 0; green, 117; blue, 255 }  ,draw opacity=1 ][line width=1.5]  [dash pattern={on 5.63pt off 4.5pt}]  (239.66,240.75) -- (350.45,22) ;
			\draw [shift={(299.57,122.45)}, rotate = 116.86] [fill={rgb, 255:red, 0; green, 117; blue, 255 }  ,fill opacity=1 ][line width=0.08]  [draw opacity=0] (15.6,-3.9) -- (0,0) -- (15.6,3.9) -- cycle    ;
			\draw [color={rgb, 255:red, 208; green, 2; blue, 27 }  ,draw opacity=1 ][line width=1.5]    (240.72,22) .. controls (250.66,59.75) and (249.66,199.75) .. (239.66,240.75) ;
			\draw [shift={(247.69,119.42)}, rotate = 89.98] [fill={rgb, 255:red, 208; green, 2; blue, 27 }  ,fill opacity=1 ][line width=0.08]  [draw opacity=0] (15.6,-3.9) -- (0,0) -- (15.6,3.9) -- cycle    ;
			\draw [color={rgb, 255:red, 208; green, 2; blue, 27 }  ,draw opacity=1 ] [dash pattern={on 4.5pt off 4.5pt}]  (240.72,22) .. controls (230.66,59.75) and (231.66,200.75) .. (239.66,240.75) ;
			\draw [shift={(233.45,138.62)}, rotate = 269.78] [fill={rgb, 255:red, 208; green, 2; blue, 27 }  ,fill opacity=1 ][line width=0.08]  [draw opacity=0] (12,-3) -- (0,0) -- (12,3) -- cycle    ;
			
			\draw (111.42,-224.12) node [anchor=north west][inner sep=0.75pt]  [rotate=-46.88] [align=left] {};
			\draw (234,1.4) node [anchor=north west][inner sep=0.75pt]  [color={rgb, 255:red, 0; green, 119; blue, 255 }  ,opacity=1 ]  {$L_{0}$};
			\draw (178,156.4) node [anchor=north west][inner sep=0.75pt]  [color={rgb, 255:red, 0; green, 119; blue, 255 }  ,opacity=1 ]  {$L_{2}$};
			\draw (252,78.4) node [anchor=north west][inner sep=0.75pt]  [color={rgb, 255:red, 208; green, 2; blue, 27 }  ,opacity=1 ]  {$L_{\tr}$};
		\end{tikzpicture}
		\caption{Atiyah-Floer perspective on \Cref{thm:intro} when $n = 0$.}
		\label{fig:atiyah-floer}
	\end{figure}
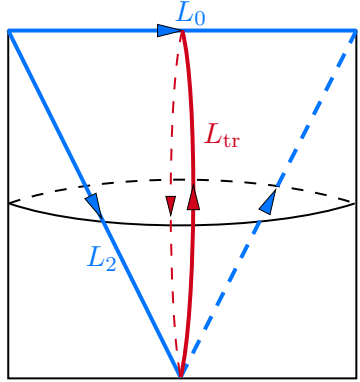
	
	One can verify that the maps $\cdot p$ and $\cdot q$ under the functor $LF(\M(\yk), \cdot)$ correspond to the cobordisms $(Y, K) \to Y_2$  and $Y_0 \to (Y, K)$ for the trace of the surgery (or its reverse). The connections are singular along the co-core and core of the 2-handle respectively. The map $L_2 \to L_1$ is less transparent as the usual representatives of the Lagrangians intersect in two (singular) points. However, the author believes that in the above analogy it corresponds to the trace of a rational cobordism as defined in \Cref{sec:rational-surgery} under the functor $LF(\M(\yk), \cdot)$ perhaps with signs that depend on the connection near the isolated orbifold point. In any case, we will construct the map on the instanton chain complex with this understanding and see that it is the one that appears in \Cref{thm:intro}.
	
	\subsection{Organisation}
	In \Cref{sec:applns}, we give the proof of \Cref{thm:l-space} assuming \Cref{thm:intro}. In \Cref{sec:instantons}, we primarily review the construction of instanton Floer theory and its properties as described in \cite{kmknot}. Additionally, we state and prove some of the technical ingredients needed for the proof of \Cref{thm:main}.
	
	In \Cref{sec:energy-order}, we describe a way to choose a basis with good properties which rests on the new perturbation scheme described in \Cref{sec:energy}. In \Cref{sec:preliminaries}, we state \Cref{thm:main} and setup the needed topological and metric constructions for the proof. In \Cref{sec:moduli-computation}, we state and prove a result about the moduli space of anti-self dual connections that will be the central to the proof of \Cref{thm:main}. Finally, in \Cref{sec:tech}, we reprove a homological algebra result of \cite{lin} suited to our needs and prove \Cref{thm:main}.
	
	In \Cref{appendix:perturbations}, we elaborate on the proof details of \Cref{prop:mtransverse}.
	
	\subsection{Acknowledgements}
	I would like to thank my advisor, Tomasz Mrowka, for introducing me to this circle of ideas and his constant support and encouragement throughout the project. Further, I would like to thank Tye Lidman for raising the question whose answer went on to become the main result of this paper. I would also like to thank Sherry Gong for sharing a draft of her work which influenced some of the proofs presented here and Jiakai Li for helpful discussions at various stages.
	
	The author was partially supported by NSF grants DMS-2105512 and 1440140. The author was also supported by Simons Foundation Award \#994330 (Simons Collaboration on New Structures in Low-Dimensional Topology). Part of the work was carried out while the author was in residence at the Simons Laufer Mathematical Sciences Institute (SLMath), formerly Mathematical Sciences Research Institute (MSRI), in Berkeley, California, during the Fall semester of 2022.

	\section{Proof of \Cref{thm:l-space}} \label{sec:applns}
	
	We first collect the computational results that the proof relies on.
	
	\begin{proposition} \label{prop:facts}
		The following results hold where the grading is the absolute $\Z / 2$ grading:
		\begin{description}[before={\renewcommand\makelabel[1]{$\bullet$ \bfseries ##1}}]
			\item[\cite{kmknot}] $I^\# (S^3, T_{2, 3}; \Z) = \Z_{(0)}^3 \oplus \Z / 2_{(0)} \oplus \Z_{(1)}$.
			\item[\cite{bs-l-space}] $I^\# (S^3_{n} (T_{2, 3}); \Q) = \Q^n_{(0)}$ when $n = 1, 3$.
			\item[\cite{sca}] $I^\# (L(p, q); \Z) = \Z^p_{(0)}$. \qed
		\end{description}
	\end{proposition}
	
	We briefly explain the strategy of the proof now. An application of \Cref{thm:intro} with $n= 3$ is used to get partial information about the torsion in $I^\#(S^3_3(T_{2,3}); \Z)$. Then, an application of \Cref{thm:intro} with this partial information is enough to prove the existence of non-trivial torsion. Finally, we use \cite[Corollary 1.7]{sca} to conclude.
	
	\begin{proof}[Proof of \Cref{thm:l-space}]
		Let $A_n = I^\#(S^3_n(T_{2, 3}; \Z))$. Recall that $S^3_5(T_{2, 3}) \cong L(5, 4)$ and $S^3_1(T_{2, 3}) \cong \Sigma(2, 3, 5)$, the \poinc homology sphere. Then, by \Cref{prop:facts}, $A_5 = \Z_{(0)}^5$. Now, applying \Cref{thm:intro}, 
		\begin{equation*}
			\begin{tikzcd}[column sep=small]
				A_3 \arrow[rr, "\psi"] & & \Z_{(0)}^5 \arrow[dl, "f_1"] \\
				& \Z^3_{(0)} \oplus \Z / 2_{(0)} \oplus \Z_{(1)} \arrow[ul, "f_2"]&
			\end{tikzcd}
		\end{equation*}
		where $\mathrm{deg}(\psi) = \mathrm{deg}(f_1) = 0$ and $\mathrm{deg}(f_2) = 1$. This implies that the following is exact
		\[
		0 \to \coker(f_1)[1] \to A_3 \to \ker(f_1)\to 0.
		\]
		As $\ker(f_1)$ is free, the sequence splits. Further, $\mathrm{Tor}(A_3) = \mathrm{Tor}(\coker(f_1)[1])$ shows that the torsion of $A_3$ is contained in the odd grading. In particular, we have, by \Cref{prop:facts}, $A_3 = \Z_{(0)}^3 \oplus K_{(1)}$ where $K$ is a finite abelian group, possibly trivial. Next, we have
		\begin{equation*}
			\begin{tikzcd}[column sep=small]
				A_1 \arrow[rr, "\psi'"] & & \Z_{(0)}^3 \oplus K_{(1)} \arrow[dl, "f_1'"] \\
				& \Z^3_{(0)} \oplus \Z / 2_{(0)} \oplus \Z_{(1)} \arrow[ul, "f_2'"]&
			\end{tikzcd}
		\end{equation*}
		This, along with $\mathrm{deg} (f_1') = \mathrm{deg}(\psi') = 0$ and $\mathrm{deg}(f_2') = 1$, implies the following short exact sequence:
		\[
		0 \to \coker(f_1')[1] \to A_1 \to \ker(f_1') \to 0.
		\]
		We claim that $A_1$ has non-trivial torsion and that it is contained in odd grading.
		
		First, we observe that $1 \leq \rk(\coker(f_1')) \leq \rk(A_1)$ where the first inequality is for grading reasons. As $\rk(A_1) = 1$ by \Cref{prop:facts}, we see that $\rk(\coker(f_1')) = 1$. Thus, $\ker(f_1') = K_{(1)}$. Hence, the following is a short exact sequence:
		\[
		0 \to \Z_{(0)}^3 \xrightarrow{f_1'} \Z^3_{(0)} \oplus \Z / 2_{(0)} \oplus \Z_{(1)} \to \coker(f_1') \to 0.
		\]
		If $\coker(f_1')$ were free, this sequence would split yielding a contradiction. Let $\coker(f_1') = \Z_{(1)} \oplus H_{(0)}$ for some finite abelian group $H \neq 0$. In conclusion, we have that the following sequence of graded abelian groups is exact:
		\[
		0 \to \Z_{(0)} \oplus H_{(1)} \to A_1 \to K_{(1)} \to 0.
		\]
		
		Thus, $A_1 = \Z_{(0)} \oplus G_{(1)}$ where
		\[
		0 \to H \to G \to K \to 0
		\]
		is an exact sequence of finite abelian groups with $H \neq 0$. Thus, $G \neq 0$. Finally, we conclude that any torsion of $A_1$ must be 2-torsion by \cite[Corollary 1.7]{sca}.
	\end{proof}
	
	We record an interesting corollary of the above proof.
	
	\begin{corollary}
		If $I^\# (S^3_3(T_{2, 3}); \Z)$ is not free, then the torsion is 2-torsion and is contained in odd grading.
	\end{corollary}

	\section{Instanton Homology}
	\label{sec:instantons}

	We review and set notation for singular instantons and their associated Floer homology groups. The main technical references are \cite{kmknot} and \cite{kmsurf1}. Other useful references are \cite{dym} and \cite{fsorbi}.
	
	\subsection{Singular Connections, Metrics and Chern-Simons Functional}
	\label{sec:con}
	
	Let $Y$ be a closed connected oriented 3-manifold with a knot (link) $K$. Fix the data of a singular bundle on $(Y, K)$ as in \cite[\S 3.1]{kmknot} with principal $\so$-bundle $P$ over $\yk$ (arising from the bundle $P_\Delta$ over $Y_\Delta$) and local system $\Delta$ and an orbifold metric $\check{g}$ with cone angle $\pi$ near $K$. We denote by $\check{Y}$, the manifold $Y$ with this orbifold structure near $K$ and by $\boldsymbol{P}$ the singular bundle data.
	
	In this situation, we have a model singular connection $A_1$ which has holonomy of order $2$ along the meridians of $K$ as described in \cite[\S 2.2]{kmknot} which is a smooth orbifold connection on $\check{Y}$. Thus, we can now define Sobolev spaces: 
	\[ \check{L}^2_k (\check{Y}; \mathfrak{g}_{\check{P}}\otimes \Lambda^q T^* \check{Y}) \] which in turn allows us to define the relevant space of connections $\mathcal{C}_k(Y, K, \pp)$ as an affine space modelled on $\check{L}^2_k (\check{Y}; \mathfrak{g}_{\check{P}}\otimes T^* \check{Y})$ and containing $A_1$.
	
	We equip $\conk$ with an inner product on the tangent space $T_B \conk = \check{L}^2_k (\check{Y}; \mathfrak{g}_{\check{P}}\otimes T^* \check{Y})$ given by 
	\[\<b, b'>_{L^2} = \int_{Y} -\tr(*b \wedge b') \]
	where the Hodge star is the one defined by $\check{g}$.
	
	Analogous definitions apply to the 4-dimensional case; see \cite[\S 2.1-2]{kmknot}.
	
	We have the \textit{Chern-Simons} functional defined with the property that 
	\[ (\grad CS)_B = * F_B\] 
	and so the critical points are flat connections in $\conk$. We denote the set of gauge equivalence class of critical points by $\fc = \fc(Y, K, \pp)$.
	
	\subsection{Gauge Group}
	
	To define the gauge group in this context, we need to ensure that our gauge transformations respect the singular data. Let $Q_\Delta$ denote the reduction of the $\so$ bundle $P_\Delta$ to a $O(2)$ bundle over $K_\Delta$ (note that this is part of the data of the singular bundle) and $Q$ for the reduction over $\nu \setminus K$ where $\nu$ is the tubular neighbourhood of $K$ in $Y$. Let $G(P_\Delta) = P_\Delta \times_{\Ad} SU(2)$ denote a bundle of groups over $Y_\Delta$ with fibres $SU(2)$. This bundle comes with a distinguished sub bundle $H_\Delta$ over $K_\Delta$ with fibres $S^1 \subset SU(2)$ which preserves $Q_\Delta$ with its orientation. Note that the structure group of this bundle is $\pm 1$ and it is in fact the pull back of a bundle of groups $H$ over $K$ with fibres $S^1$ (and structure group $\pm 1$).
	
	Define a topological space $\boldsymbol{G}$ with a map to $Y$ by the following pushout diagram:
	\begin{equation*}
		\begin{tikzcd}
			\pi^* H|_{\nu (K) \setminus K} \arrow[r] \arrow[d] & G(P)|_{Y \setminus K} \arrow[d] \\ 	
			\pi^* H \arrow[r] & \boldsymbol{G} \arrow[ul, phantom, "\ulcorner", very near start]
		\end{tikzcd}
	\end{equation*}
	where $\pi : \nu (K) \to K$ is the projection map.
	
	The fibres of $\boldsymbol{G}$ are $SU(2)$ over $\yk$ and $S^1$ over $K$ and the set of continuous gauge transformations $\CG^{top}$ is the space of continuous sections of $\boldsymbol{G} \to X$. As $\pi^* H|_{\nu(K) \setminus K}$ is a smooth subbundle of $G(P)|_{\nu (K) \setminus K}$, smooth sections of $\boldsymbol{G}$ are well defined. With this, we can define $\CG_{k+1} (Y, K, \pp)$ as the completion of $\CG^{smooth}$ under the usual Sobolev norms.

	\subsection{Weighted Function Spaces and Weighted Operators} 
	\label{sec:weighted-spaces}
	
	We introduce the notion of weighted spaces of connections and gauge group relevant to us. As usual, these do not affect the usual theory of \cite{kmknot} when the connections involved are irreducible. We follow \cite{dym} with the appropriate modifications needed for the singular locus.
	
	Consider $(Z, S) = (Y, K) \times \pR$ as a product orbifold with singularity along $K \times \pR$; we shall denote this, equipped with the product orbifold metric by $\check{Z}^+$. We extend the singular data $\pp$ to $(Z, S)$ in an analogous fashion and define $\conw{Z}{S}$ as the affine space of connections containing the product connection $A_1 + dt$ with tangent space $\lp[k, \alpha]{T^* \check{Z}}{Z}$ where $\alpha \in \R$ denotes the weight; the weight function is the usual one: $e^{\alpha t}$. More explicitly, 
	\[
	\| f |_{\check{L}^2_{k, \alpha}} = \| e^{\alpha t} f|_{\check{L}^2_k}
	\]
	
	We shall always require the weight to be small enough so that it is less than the smallest (in absolute value) eigenvalue of the relevant differential operator. We shall use this notation even when $Z$ has multiple ends (in which case $\alpha$ must be interpreted as a tuple) and ask the weight function to be standard on the ends while smoothly varying and bounded below by a positive constant elsewhere.
	
	The gauge group $ \CG_{k+1, \alpha} (Z, S, \pp)$ is similarly defined as the completion of $\CG^{smooth}$ with respect to the weighted Sobolev space topology. In particular, this means that $\nabla_{A_1} g \in \lp[k, \alpha]{T^* \check{Z}}{Z}$ when $g \in \CG_{k+1, \alpha} (Z, S, \pp)$ and so the gauge group acts on $\conw{Z}{S}$. The action is smooth is by the usual multiplication theorems (cf. \cite{dym}).
	
	The `Coloumb gauge' condition takes a slightly different form in this context:
	\[ \lp[k, \alpha]{T^* \check{Z}}{Z} = \ker(d_{A_1}^{*, \alpha}) \oplus \im (d_{A_1}) \]
	where $d_{A_1}^{*, \alpha}$ denotes the formal adjoint of the covariant derivative $d_{A_1}$ with respect to the \emph{weighted} $L^2$ inner product instead of the usual $L^2$ inner product. See \cite[\S 4.3]{dym} for a discussion when the orbifold is a manifold; to deduce the orbifold version stated above, note that the result holds for the double branched cover and the relevant covering transformation is parallel with respect to $A_1$ by construction and the weight function can be chosen to be invariant with respect to the covering transformation. 
	
	This in turn affects the linearisation of the ASD operator. The relevant operator, incorporating this change, we denote by 
	\[ \D_{A_1, Z, \alpha} \coloneqq - d_{A_1}^{*, \alpha} \oplus d_{A_1}^+  : \lp[k, \alpha]{\Lambda^1}{Z} \to \lp[k-1, \alpha]{(\Lambda^0 \oplus \Lambda^+)}{Z} \] 
	which we shall use whenever any of the limiting connections is reducible. However the zero set of the non-linear ASD operator is not the moduli space but requires a further quotient by the stabiliser of the limiting connection: This merely reflects the fact that the `Coloumb gauge' condition only quotients out by gauge transformations that limit to $\pm 1$ at the ends. The dimension of this moduli space, for instance when $Z$ has a single end $Y$ with the limiting connection being $\rho$ and assuming the action of the stabiliser is free, is $\ind \D_{A, Z, \alpha} - \dim H_{\rho}^0(Y; \mathfrak{g}_{P_Y})$. This is precisely the quantity $\ind^+ P$ in \cite{dym}. 
	
	With this understood, we shall omit $\alpha$ if it is clear that a reducible connection is present when dealing with differential operators and function spaces with the understanding that $\alpha > 0$ and $\alpha \ll 1$.
	
	\subsection{Reducible Connections and Perturbations}
	The perturbations we discuss here were first introduced in \cite{taubes-casson} and were used in the context of Floer groups first in \cite{floer}. These were later adapted to the setting of singular instanton Floer theory in \cite{kmknot}. We state a proposition from \cite[\S 3.3, 3.4]{kmknot} regarding non-integrability in the context of reducible connections.
	
	\begin{proposition}[{\cite[Proposition 3.2]{kmknot}}]
		If the singular data $\pp$ on $(Y, K)$ satisfies the non-integral condition (as defined in \cite[Definition 3.1]{kmknot}), then the Chern-Simons functional on $\conk$ has no reducible critical points. Further, if $\Delta$ is non-trivial on any component of $K$, then the configuration space $\conk$ contains no reducible connections at all.
		\qed
	\end{proposition}
	
	We now describe a set of perturbations that ensure all relevant moduli spaces are transversely cut out. We follow \cite[\S 3.4]{kmknot} which in turn relies on \cite{yaft}.
	
	Fix a lift of $P$ over $\yk$ to a $U(2)$ bundle $\tilde{P}$ over $\yk$ and a connection $\theta$ on $\det \tilde{P}$. Then, any $B \in \conk$ defines a connection $\tilde{B}$ on $\tilde{P}$ with induced connection on $\det \tilde{P}$ being $\theta$. Given an immersion $q : S^1 \times D^2 \to \yk$, and a choice of base point $p \in S^1$, for every $x \in D^2$, we can define $\mathrm{Hol}_x(\tilde{B}) \in U(2)$ by taking the holonomy of $q^*\tilde{B}$ over $S^1 \times \{ x \}$ (starting at $(p, x)$). Choosing a class function $h : U(2) \to \R$, we get a function 
	
	\[ H_x : \conk \to \R \]
	defined as $H_x(B) = h(\mathrm{Hol}_x(\tilde{B}))$
	
	It is useful to \emph{mollify} this function as follows:
	
	\[ f_q(B) = \int_{D^2} H_x(B) \mu \]
	where $\mu$ is a volume form on $D^2$ supported in the interior with integral $1$.
	
	Analogously, taking a collection of immersions $\boldsymbol{q} = (q_1, \ldots, q_l)$ all agreeing on $p \times D^2$, the holonomy along the $l$ loops gives us $\mathrm{Hol}_x: \mathcal{C}_k \to U(2)^l$ and composing with conjugate invariant function $h: U(2)^l \to \R$ (to give $H_x (B) = h(\mathrm{Hol}_x(\tilde{B}))$) followed by mollifying gives:
	
	\[ f_q : \conk \to \R \] 
	\[ f_q(B) = \int_{D^2} H_x(B) \mu \]
	These smooth gauge invariant functions on $\mathcal{C}$ are \textit{cylinder functions}.
	
	Given this construction, it is explained in \cite{yaft} how one can choose a collection of immersions $\{ \boldsymbol{q^i} \}$ and construct a separable Banach space $\Pa$ of sequences $\pi = \{ \pi_i \}$ with norm
	
	\[ \| \pi |_{\Pa} = \sum_i C_i | \pi_i | \]
	such that for each $\pi \in \Pa$, the sum,
	
	\[ f_{\pi} = \sum_i \pi_i f_{\boldsymbol{q^i}}\]
	is convergent and defines a smooth and bounded function $f_\pi$ on $\SC_k$. Further, the $L^2$ gradient defines a smooth vector field $V_\pi$ on $\SC_k$. The analytic properties needed for these (which can be satisfied by constructing $\Pa$ appropriately) are stated in \cite[Proposition 3.7]{yaft}.
	
	We call such a function $f_\pi$ a holonomy perturbation and we consider the perturbed Chern-Simons function $CS + f_\pi$. The set of gauge equivalence classes of critical points for this is denoted by:
	\[ \fc_\pi = \fc_\pi (Y, K, \pp) \subset \CB_k (Y, K, \pp).\]
	These holonomy perturbations ensure the critical set is non-degenerate. We state the relevant result below.
	
	\begin{proposition}[{\cite[Proposition 3.12]{yaft}}] \label{prop:morse}
		There is a residual set of the Banach space $\Pa$ such for all $\pi$ in this subset, the irreducible critical points of the functional $CS + f_\pi$ in $\conk$ are non-degenerate in the directions transverse to the gauge orbit. \qed 
	\end{proposition}
	
	One would want to avoid having to work with reducible critical points by satisfying the non-integrality condition. The next result tells us that one can choose perturbations that do not introduce any reducibles in this situation.
	
	\begin{lemma}[{\cite[Lemma 3.11]{yaft}}] \label{prop:irred}
		Suppose $(Y, K, \pp)$ satisfies the non-integral condition. Then, there exists an $\varepsilon > 0$ such that for all $\pi \in \Pa$ with $\|\pi|_\Pa \leq \varepsilon$, the critical points of $CS + f_\pi$ have no reducible critical points. \qed
	\end{lemma}
	
	\subsection{Perturbed Trajectories}
	
	Given a holonomy perturbation $f_\pi$ as above, we get a perturbed gradient trajectory equations on the cylinder. By cylinder, we mean \[ (Z, S) = \R \times (Y, K) \]
	viewed as a four manifold with a (properly) embedded surface $S$. We can equip it with an product orbifold structure along $S$ arising from the orbifold structure in the 3-dimensional case. This data includes an orbifold metric which is just the product metric.
	
	Let $A = B + c dt$ be a connection on this 4-manifold where $B(t)$ is an orbifold connection on $(Y, K) \times \{ t \}$. The perturbation now takes the following form:
	
	\[ \hat{V}_\pi (A) = P_+ (dt \wedge V_\pi(B(t))) \]
	where $P_+$ is the projection onto the self dual 2-forms on the cylinder and $V_\pi(B)$ is a $\mathfrak{g}_{\check{P}}$ valued 1-form on $(Y, K) \times \{t\} $. The perturbed equations are:
	
	\[ F_A^+ + \hat{V}_\pi (A) = 0.\]
	
	For the rest of this section, unless stated otherwise, $\pi$ is chosen such that all the critical points $\fc_\pi$ are irreducible and non-degenerate. Choose connections $B_1$ and $B_0$ in $\conk$ for gauge equivalence classes $\beta_1, \beta_0 \in \fc_\pi$. With these choices, construct $A_o$, a connection on $\R \times Y$, such that $A_o$ agrees with the pull-back of $B_1$ and $B_0$ for large negative and large positive $t$ respectively. The connection $A_o$ determines a path $\gamma : \R \to \CB_k (Y, K, \pp)$, from $\beta_1$ to $\beta_0$ and thus a relative homotopy class  $z \in \pi_1( \CB_k; \beta_1, \beta_0 )$. We can define a Sobolev space of connections 
	
	\[ \SC_{k, \gamma} (Z, S, \pp; B_1, B_0) \]
	as the affine space
	
	\[ \{ A | A - A_o \in \lp[k, A_o]{T^* \check{Z} }{Z} \}, \]
	which depends on the choice of $A_o$.
	
	There is an associated gauge group, the space of sections of the bundle $G(P) \to Z \setminus S$ defined by 
	
	\[ \CG_{k+1} (Z, S, \pp) = \{ g | \nabla_{A_o}g, \ldots, \nabla_{A_o}^k g \in \check{L}^2 (Z \setminus S) \}. \]
	
	The quotient space will be denoted as:
	
	\[ \CB_{k, z} (Z, S, \pp; \beta_1, \beta_0) = \SC_{k, \gamma}(Z, S, \pp; B_1, B_0) / \CG_{k+1}(Z, S, P).\]
	Here $z$, as before, is the class $[\gamma] \in \pi_1 ( \CB_k (Y, K, \pp); \beta_1, \beta_0)$.
	
	The moduli space of solutions to the perturbed ASD equations (or $\pi$-ASD equations) is a subspace of connections up to gauge equivalence:
	
	\[ M_z (\beta_1, \beta_0) = \{ [A] \in \CB_{k, z} (Z, S, \pp; \beta_1, \beta_0) \; | \; F_A^+ + \hat{V}_\pi (A) = 0\\ \}. \]
	
	The total moduli space is defined as:
	
	\[ 
	M (\beta_1, \beta_0) = \bigcup_{ z \in \pi_1(\CB_k, \beta_1, \beta_0)} M_z (\beta_1, \beta_0). 
	\]
	The action of $\R$ by translations on $\R \times Y$ induces an action on $M(\beta_1, \beta_0)$ is free on the subset of non-constant trajectories and acts trivially on constant trajectories. We call the quotient of the space of non-constant trajectories by this $\R$ action to be $\breve{M} (\beta_1, \beta_0)$ and we will denote elements by $[\breve{A}]$.
	
	The linearisation of the $\pi$-ASD equation at a connection $A$ in $\SC_{k, \gamma} (B_1, B_0)$ is:
	\[
	d_A^+ + D \hat{V}_\pi : \lp[k, A]{\Lambda^1}{Z} \to \lp[k-1]{\Lambda^+}{Z}.
	\]
	When $B_1$ and $B_2$ are irreducible and non-degenerate, this operator is \emph{Fredholm} (with gauge fixing); we clarify the precise meaning below. 
	
	For $A$ as above, we write 
	\[
	\D_A = ( d_A^+ + D \hat{V}_\pi ) \oplus -d_A^*
	\]
	which defines an operator
	\[
	\lp[k, A]{\Lambda^1}{Z} \to \lp[k-1]{(\Lambda^+ \oplus \Lambda^0)}{Z}.
	\]
	The earlier statement is to be interpreted as $\D_A$ is a Fredholm operator. Note that the second component of $\D_A$ is the linearised (Coulomb) gauge fixing condition. Thus, if $\D_A$ is surjective, $M_z(\beta_1, \beta_0)$ is a smooth manifold near $[A]$, of dimension $\ind \, \D_A$; we denote the latter by
	\[
	\mathrm{gr}_z(\beta_1, \beta_0)
	\]
	and say that $[A]$ is regular in this case.
	We introduce one final piece of notation:
	\[
	M (\beta_1, \beta_0)_d = \bigcup_{ z \, | \, \mathrm{gr}_z = d} M_z (\beta_1, \beta_0).
	\]
	
	We now state the four dimensional transversality result achieved by the holonomy perturbations.
	
	\begin{proposition}[{\cite[Proposition 3.18]{yaft}}] \label{prop:morse-smale}
		Let $\pi_0$ be chosen such that all critical points $\fc_{\pi_0}$ are non-degenerate and irreducible. Then there exists $\pi \in \Pa$ such that:
		
		\begin{enumerate}[label=(\roman*)]
			\item $f_\pi = f_{\pi_0}$ in a neighbourhood of all the critical points of $CS + f_{\pi_0}$.
			\item $\fc_\pi = \fc_{\pi_0}$.
			\item all moduli spaces $M_z (\beta_1, \beta_0)$ for the perturbation $\pi$ are regular. \qed
		\end{enumerate}
	\end{proposition}
	
	From this point on, we shall assume such a perturbation has been chosen implicitly unless stated otherwise.
	
	\subsection{Orientations and Floer Homology} \label{sec:orientations}
	
	The Fredholm operators $\D_A$ form a family over the space $\CB_{k, z} (\beta_1, \beta_0)$. A key observation is that the determinant line bundle, $\det (\D_A)$, is orientable: One first proves it for $ (X, \Sigma) = (Y, K) \times S^1 $ and the general case follows from excision (cf. \cite[Proposition 2.12]{kmknot}). Thus, $M_z (\beta_1, \beta_0)$ is orientable. Moreover, fixing an orientation on $\CB_{k, z} (\beta_1, \beta_0)$ determines it for $\CB_{k, z'} (\beta_1, \beta_0)$, for any other $z'$: \cite[Proposition 2.14]{kmknot} deals with the closed case, and the proof extends to the current context straightforwardly. In light of this, denote by 
	\[
	\Lambda (\beta_1, \beta_0)
	\]
	the two element set of orientations of $\det(\D_A)$ over $\CB_{k, z} (\beta_1, \beta_0)$ with the understanding that this does not depend on $z$.
	
	One can make similar definitions even when $\beta_i$ are not critical points. One change is that now the Hessian is not guaranteed to be non-degenerate at such points which means that $\D_A$ may fail to be Fredholm on the usual Sobolev spaces. The usual fix to this described in \cite{yaft}: we use weighted Sobolev spaces
	\[
	e^{-\epsilon t} \check{L}^2_{k, A_o}
	\]
	for a small $\epsilon > 0$. When $\beta_i$ are not critical points, we define $\Lambda (\beta_1, \beta_0)$ with this convention.
	
	In this way, one can choose a convenient base-point connection $\theta \in \CB_k (Y, K, \pp)$ and define 
	\[
	\Lambda (\beta) = \Lambda (\theta, \beta)
	\]
	for any critical point $\beta$. If $K$ were oriented and $\Delta$ were trivial, then we can choose $\theta$ to be a reducible connection as in \cite{yaft}. If $\Delta$ is non-trivial, there is no preferred choice of $\theta$ as $\CB_k$ contains no reducibles. As we shall see, this leads to an overall ambiguity in signs when defining certain maps.
	
	With the above definition of $\Lambda (\beta)$, there are identifications
	\[
	\Lambda (\beta_1, \beta_0) = \Lambda(\beta_1) \Lambda(\beta_0),
	\]
	where the product on the right is the product of 2-element sets (see \cite[\S 3.6]{kmknot}). Hence, each connected subset $[\breve{A}] \subset M(\beta_1, \beta_0)_1$ determines an isomorphism $\Lambda(\beta_1) \to \Lambda(\beta_0)$. As in \cite{yaft, kmbook}, we denote by $\Z \Lambda(\beta)$ the group $\Z$ with generators identified $\Lambda (\beta)$, and we write
	\[
	\varepsilon[\breve{A}] : \Z \Lambda (\beta_1) \to \Z \Lambda (\beta_0)
	\]
	for the corresponding isomorphism of groups.
	
	We can now define the Floer homology groups. Given $(Y, K)$ where $K$ is an unoriented link in a closed, oriented, connected 3-manifold $Y$. Fix $\pp$, the data of a non-integral singular bundle. In this situation, we can choose a metric $\check{g}$ and perturbation $\pi$ such that there are no reducibles, every critical point is non-degenerate and all moduli spaces are regular; fix such a choice. Finally fix a base-point $\theta \in \CB_k(Y, K, \pp)$. The chain complex, denoted by $(C_*(Y, K, \pp), \partial)$, of free abelian groups is defined by
	\[
	C_*(Y, K, \pp) = \bigoplus_{\beta \in \fc_\pi} \Z \Lambda(\beta)
	\]
	equipped with the differential
	\[
	\partial = \mathlarger{\sum}_{\substack{(\beta_1, \beta_0, z) \\ \mathrm{gr}_z(\beta_1, \beta_0) = 1}} \mathlarger{\sum}_{[\breve{A}] \in M_z(\beta_1, \beta_0)_1} \varepsilon[\breve{A}].
	\]
	This sum is guaranteed to be finite by \cite[Corollary 3.25]{yaft}. 
	
	\begin{defx}
		In the above setup, we define the instanton Floer group 
		\[
		I(Y, K, \pp)
		\]
		to be the homology of $(C_*(Y, K, \pp), \partial)$.
	\end{defx}
	
	The above definition depends on some choices; the next section on cobordisms provides the material needed to use the standard cobordism style arguments to show that the group is well defined nonetheless. However, there is one difference that essentially stems from the lack of a canonical base-point which will be explained in \Cref{sec:cob-maps}.
	
	\subsection{Complex Orientations} \label{sec:complex orientations} 
	
	We discuss complex orientation here following the discussion in \cite[\S 5.1]{kmknot}.
	
	We briefly recall the case of ASD equations for an $\so$ bundle $P$ on a closed, oriented, Riemannian $4$-manifold $X$. In this case, \cite{Donaldson-orientations} describes a way to orient moduli spaces when $X$ admits an almost complex structure $J$: there is a standard homotopy from the operators $\D_{A}$ to the complex-linear operator
	\[
	\bar\partial_{A} + \bar\partial^{*}_{A} :
	\Omega^{0,1}_{X}\otimes_{\R}\g_{P} \to (\Omega^{0,0}_{X}\oplus
	\Omega^{0,2}_{X})\otimes_{\R}\g_{P}.
	\]
	For the latter operator, the complex orientation gives a preferred orientation; one can use the above homotopy to transfer this preferred choice to $\det(\D_{A})$. For our purposes, the crucial upper hand this construction provides over the other one in \cite{Donaldson-orientations} is that this method can be used to prove orientability of the determinant line bundle under not just the determinant-1 transformations but the \emph{full} gauge automorphism group.
	
	Following \cite[\S 5.1]{kmknot}, $\overline{\Lambda} (\beta)$ denotes the possible orientations of the determinant of the operator on $\R \times Y$ such that it restricts to $\D_\beta$ on the $+\infty$ end and $\bar\partial_{\beta} + \bar\partial^{*}_{\beta}$ on the $-\infty$ end. Then, just as in \Cref{sec:orientations}, we have
	\[
	\overline{\Lambda} (\alpha, \beta) = \overline{\Lambda}(\alpha) \overline{\Lambda} (\beta).
	\]
	The rest of the constructions in \Cref{sec:orientations} go through as before except when a critical point on a cylindrical point is not acyclic, in which case weights are needed. 
	
	\begin{rmk} \label{rmk:complex-orientations}
		The main advantage provided by this convention is that the situation with orientations is not overly complicated when stretching along a 3-manifold hyper-surface with flat connections that are non-degenerate, but not acyclic. In such a scenario, for the purposes of this paper, we shall deal exclusively with the complex orientation by using appropriate almost complex structures $J$ on the relevant cobordisms. See also \cite[Appendix 2]{kmsurf2} for a discussion in the case of 4-manifolds with an orientable singular locus. 
	\end{rmk}
	
	\subsection{Moduli Space on Orbifolds with Cylindrical Ends}
	
	The section introduces orbifold singularities in a manner that closely follows \cite{fsorbi}; however, for orientation reasons, we will deviate from them in the technical implementation. We will not be fixing a perturbation as in \Cref{prop:morse-smale} for this subsection.
	
	Given a cobordism $(W, S)$, we can attach cylindrical ends and define the map on floer homology by counting solutions to the ASD equations. However, for our applications we need a slightly more general situation that allows orbifold points in the interior of the cobordism. We make this precise below as orbifold cobordisms of pairs.
	
	\begin{defx}[Orbifold cobordism]
		Let $\tilde{W}$ by an oriented compact 4-manifold with $\partial \tilde{W} = Y_1 \bigsqcup -Y_0 \bigsqcup \left( \bigsqcup_{i = 1}^k L(p_i, q_i) \right)$ where $p_i$ are relatively prime. Suppose we have a properly embedded surface (possibly unoriented) $S \subset \tilde{W}$ with $\partial S \subset Y_1 \bigsqcup -Y_0$. Then, we cone off the lens spaces $L(p_i, q_i)$ (denote the cone points by $c_i$) to get a 4-dimensional orbifold $W$ with smooth boundary $Y_1 \bigsqcup -Y_0$ and an embedded surface $S$ disjoint from $c_i$. Now suppose $S \cap Y_i = K_i$, we call $(W,S)$ an orbifold cobordism from $(Y_0, K_0)$ to $(Y_1, K_1)$. If $S$ is oriented, we give $K_1$ the boundary orientation while $K_0$ gets the opposite of the boundary orientation as is usual for oriented cobordisms. 
	\end{defx}
	
	We first establish the relevant `removable singularities' theorem so that the maps defined by these cobordisms can be thought of as a sum of maps defined by manifolds with ends. This in particular allows us to import the technical package of \cite{kmknot} verbatim with only minor changes. We deal with orientation issues and define the maps from such cobordisms only in \Cref{sec:cob-maps}, but we lay the analytic groundwork in this section.
	
	To state this, we first introduce the appropriate set up closely following \cite{kmknot}. Given an orbifold cobordism $(W, S)$ we need to make some fix some topological data (as we made when defining the Floer homology groups) to define the appropriate moduli space. The notion of singular data from \Cref{sec:con} extends analogously to this setting in the absence of orbifold points. Near the orbifold points, we require the bundle to be an orbifold bundle; since this is a local condition and the points are disjoint from $S$, this poses no issues. We call this singular bundle data $\pp$.

	We attach cylindrical ends $(-\infty, 0] \times Y_0$ with metric $dt^2 + \check{g}_0$ and similarly $[0, +\infty) \times Y_1$ with $dt^2 + \check{g}_1$ while the metric in the interior of $W$ is an orbi metric $\check{g}$ with cone angle $\pi$ along $S$ and cone angle $\frac{2\pi}{p_i}$ near $c_i$. Further, the singular bundle data $\pp$ above furnishes the data of a singular bundle $\pp_1$ and $\pp_2$ by restriction. Fix $\beta_i \in \fc(Y_i, K_i, \pp_i)$. We now further need to choose $\beta^o_i \in \fc(L(p_i, q_i), \emptyset, \pp|_{\partial V_i})$ where $V_i$ denotes a conical neighbourhood of the orbifold point $c_i$. 
	
	Now, choose a singular bundle data $\pp'$ on $(W, S)$ that differs from $\pp$ by the addition of monopoles and instantons; we record this by a preferred map $g' : \pp' \to \pp$ defined outside a finite set of points. We choose gauge representatives $B_i'$ for $\beta_i$. On the neighbourhoods $V_i$, we pick gauge representatives $B^{o'}_i$ for $\beta^o_i$ and extend conically to $V_i$. With this, we can choose an orbifold connection, $A_o$, on $\pp'$ that agrees with the pull-back of $B_i'$ on the cylindrical ends and on $V_i$ agree with $B^{o'}_i$. 
	
	As before we can construct an affine space of connections $\mathcal{C}_k(\pp', A_o)$ consisting of all $A$ with 
	\[
	A - A_o \in \check{L}^2_{k, A_o} ( \check{W}, \Lambda^1 \otimes \mathfrak{g}_{\pp'})
	\]
	where the Sobolev space is defined using the metric $\check{g}$. We say the choices $(\pp', B_1', B_2', B^{o'}_1, \dots, B^{o'}_k)$ and $(\pp'', B_1'', B_2'', B^{o''}_1, \dots, B^{o''}_k)$ are equivalent if there is a determinant-1 gauge transformation $\pp' \to \pp''$ pulling back $B_i''$ to $B_i'$ and $B^{o''}_i$ to $B^{o'}_i$. We denote by $z$ the isomorphism class of such choices:
	\[
	z = [W, S, \pp', B_1', B_0', B^{o'}].
	\]
	where we have used $B^{o'}$ to denote all data of all the $B^{o'}_i$ for brevity. The quotient of $\mathcal{C}_k(\pp', A_o)$ by the determinant-1 gauge group $\mathcal{G}_{k+1}(\pp', A_o)$ depends only on $z$, and we write
	\[
	\mathcal{B}_z(W, S, \pp; \beta_1, \beta_0, \beta^o) = \mathcal{C}_k (\pp', A_o) / \mathcal{G}_{k+1} (\pp', A_o)
	\]
	
	We now describe the perturbations we use to ensure our moduli spaces are transversely cut out. Most of this is already covered by a result of \cite{kmknot}, see below, but there is one matter due to the orbifold points which will be settled by the removable singularities theorem to be stated. The perturbed equation on $\check{W}$ is:
	\begin{equation} \label{eqn:ASD}
	F_A^+ + \hat{V}(A) = 0
	\end{equation}
	where $\hat{V}$ is a holonomy perturbation supported on the cylindrical ends and in the union of punctured neighbourhoods of the orbifold points $c_i$. To define this on the cylindrical end $\pR \times Y_0$, we take the given $\pi_0$ from the auxiliary data $\boldsymbol{a}_0$ and an additional term $\pi_0'$. We refer to the former perturbation data as the \emph{primary perturbation} and the latter the \emph{secondary perturbation}. Given this, the perturbation is defined as follows:
	\[
	\hat{V} (A) = \phi(t) \hat{V}_{\pi_0} (A) + \psi(t) \hat{V}_{\pi'_0} (A),
	\]
	where $\phi(t)$ is a cut-off function equal to 1 on $[1, +\infty)$ and equal to 0 near $t = 0$, while $\psi(t)$ is a bump-function supported in $[0,1]$. On the other cylindrical end, the description is similar using the data of $\pi_1$ and an additional term $\pi'_1$. Near the orbifold points, we require terms $\pi^o_i$ where the perturbation takes the form:
	\[
	\hat{V} (A) = \psi(r) \hat{V}_{\pi^o_i} (A)
	\]
	where $\psi$ is a bump function supported on $[1/3, 2/3]$ and $r$ is the radial function from the cone point (in particular $r(c_i) = 0$). Note that if we conformally change the metric to make the punctured neighbourhood of $c_i$ cylindrical, then this is equivalent to the earlier form of perturbation except there is no primary perturbation defined on $L(p, q)$. 
	
	With this, we can talk about the moduli space of solutions of the equation \eqref{eqn:ASD}
	\[
	M_z (W, S, \pp; \beta_1, \beta_0, \beta^o) \subset \mathcal{B}_{k, z} (W, S, \pp; \beta_1, \beta_0, \beta^o).
	\]
	Note that these generalise in a straight-forward fashion to the case of multiple boundary components; in this case we shall denote all of the boundary data collectively as $\beta_0$. We record some of the above as definitions below for future reference.
	
	\begin{defx} \label{defn:metrics}
		Suppose $W$ is a 4-manifold with boundary and orbifold points $\tilde{c}_i$ with local isotropy groups $\Z / \tilde{p}_i$. Further, there is a distinguished collection of boundary components which are lens spaces $L(p_k, q_k)$ with $k = 1, 2, \cdots, r$. Next, assume we have a properly embedded surface $S$ such that $\partial S$ is disjoint from the distinguished boundary components and the orbifold points $c_i$. We call $S$ to be an \emph{admissible singular locus} when this is the case.
		
		Fix a metric $\check{g}$ such that it has orbifold singularities along $S$ with cone angle $\pi$, is conical with cone angle $2 \pi / \tilde{p}_i$ over orbifold points $\tilde{c}_i$, is cylindrical along boundary components and is the standard product metric on $\pR \times L(p_k, q_k)$ for the distinguished boundary components. We call such a orbifold metric $\check{g}$ compatible with the distinguished boundary data.
		
		Now, if $W'$ is obtained from $W$ by attaching a cone on $L(p_k, q_k)$, say $cL(p_k, q_k)$ with cone point $c_k$, we can equip $W' \setminus \{c_1, c_2, \cdots, c_r\}$ with a metric $\check{h}$ which agrees with $\check{g}$ away from $L(p_k, q_k)$ and $\check{h}$, restricted to a punctured neighbourhood of $c$, is conformally equivalent to $\check{g}$ on the cylindrical end $\pR \times L(p_k, q_k)$. By choosing $\check{h}$ to be conical metric with cone angle $2 \pi / p_k$ on the punctured neighbourhood of $c$, this extends to $W'$ to a metric with orbifold singularities which we shall continue to denote by $\check{h}$. 
		
		We call $W'$ the orbifold completion of $W$ with respect to the distinguished boundary components and a metric $\check{h}$ on $W'$ obtained this way a compatible orbifold metric. 
		
		The perturbation of the form discussed above with no primary terms in a neighbourhood of the orbifold points will be referred to as orbifold holonomy perturbations.
	\end{defx}
	
	We are now in a position to state and prove the appropriate removable singularity theorem needed.
	
	\begin{theorem}[Equivariant Removable Singularities] \label{thm:rm-sing}
		Suppose $W$ is a 4-manifold with orbifold points and boundary.  Suppose one of the boundary components is $L(p, q)$ and is the only  distinguished boundary component. Further assume we have a properly embedded admissible surface $S$. Fix an orbifold metric $\check{g}$ compatible with the distinguished boundary data. Let $W'$ be the orbifold completion and $\check{h}$ a compatible orbifold metric.
		
		Fix a singular bundle data over $W$, say $\pp$, and extend it to $W'$ and call it $\pp'$; this extension is merely a conical extension of the bundle over $L(p, q)$. Choose any orbifold holonomy perturbation $\hat{V}$ of the ASD equations on $W$ such that there is no primary perturbation terms for the distinguished boundary components.
		
		Now, fix a critical point $\beta$ on the bundle over $L(p, q)$ and $\beta_0$ on the rest of the boundary of $W$ with the connection data $\beta^o$ at the orbifold points. This defines an isomorphism class $z = [W, S, \pp, B_0 \cup B, B^o]$ and correspondingly we get a class $z(\beta) = [W', S, \pp', B_0, B^{o'}]$ where the data of the orbifold points, $\beta^{o'}$ is obtained by taking $\beta^o$ and adjoining the data of $\beta$ for the cone point $c$ and $B_0, B, B^o, B^{o'}$ denote a choice of connections in the gauge equivalence class of $\beta_0, \beta, \beta^o, \beta^{o'}$ respectively.
		
		Then,
		\[
		M_{z, \check{g}}(W, S, \pp; \beta \cup \beta_0, \beta^o) = M_{z(\beta), \check{h}} (W', S, \pp'; \beta_0, \beta^{o'}).
		\]
		
	\end{theorem}
	
	\begin{proof}
		Since $\check{h} = \check{g}$ away from a punctured neighbourhood of $c \in W'$ and the perturbations are also supported away from $c$, we will be done if we show that any solution to the $\check{g}$-ASD equation, $A$, on $\pR \times L(p, q)$ which limits to $\beta$ extends over $c$ as an orbifold connection to give a solution of $\check{h}$-ASD connection on $W'$. The converse, also claimed in the theorem, is clear as the $W \subset W'$ and $\check{h}|_W$ is conformally equivalent to $\check{g}$.
		
		We first note that every critical point $\beta \in \fc (L(p, q))$ is non-degenerate: Fix a representative connection $A_\beta$. Now, if $\pi : S^3 \to L(p, q)$ denotes the $\Z / p$ cover, we can pull back the connection $A_\beta$ to get a connection $A'$ on $S^3$. Further pull back induces an injective map, $\pi^* : H^1_{A_\beta} (L(p, q)) \hookrightarrow H^1_{A'} (S^3) = 0$. The last equality is due to the fact that the last group depends only on the gauge equivalence class of $A'$ and since $S^3$ is a simply connected this is merely the trivial connection for which the group is zero (cf. \cite{dym}). For the rest of the proof, fix a representative $A_\beta$ such that $\pi^* (A_\beta)$ is the trivial connection on $S^3$ (i.e., not merely gauge equivalent).
		
		Now, we also have a $\Z / p$ cover, $\pi : D^4 \setminus \{0\} \to cL(p, q) \setminus \{c\}$ which restricts to the boundary as the map $\pi$ above (we use the same notation for either map for this reason). Choosing the Coulomb gauge (with respect to $A_\beta$) for $A$ (with appropriate weights as $H^0_{A_\beta} (L(p, q))  \neq 0$), we have that $||A|_{\{t\}} - A_\beta ||_{L^2_1(L(p, q))} \leq Ce^{-\delta t}$ for $t \gg 0$ where $t \in \pR$ and $\delta > 0$ is the smallest eigenvalue of $- * d_A$ restricted to $\ker (d_A^*)$; in particular, for the case of $L(p, q)$ we can choose $\delta > 1$ (see \cite{dym}). Applying a conformal transformation to convert the metric $\check{g}$ to $\check{h}$, this translates to $||A|_{\{r\}} - A_\beta ||_{L^2_1(L(p, q))} \leq Cr^{\delta }$. Hence, $\pi^*(A)$ is $\Z /p$ equivariant connection on $D^4 \setminus \{0\}$ which differs from the trivial connection by a 1-form  $a \in L^2_1 (D^4 \setminus \{0\})$ with $||a|_{\{r\}}||_{L^2_1(S^3)} \leq r^\delta$ when $r \ll 1$. This shows that $a$ extends over origin and this defines a $\Z / p$ equivariant $L^2_1(D^4)$ connection (with respect to the trivial connection), say $\tilde{A}$. This hence defines an orbifold connection on $cL(p, q)$ that extends $A$. It is clear from this description that homotopy class of configurations is as claimed in the statement.
	\end{proof}
	
	\begin{rmk}\label{non-deg}
		For future reference, we note that in the proof we have shown that all critical points of $L(p, q)$ are non-degenerate (cf. the Weitzenbock formulae in \cite[pp. 94]{fu}). 
	\end{rmk}
	
	Just as in the cylindrical case, we can linearise the equations and use gauge-fixing to obtain a Fredholm operator $\D_A$. We say a solution $[A] \in M_z (W, S, \pp; \beta_1, \beta_0, \beta^o)$ is regular if this operator is surjective and we call the moduli space regular if each point is regular. Regular moduli spaces are smooth manifolds of dimension $\ind \, \D_A$ which we denote as 
	\[
	\mathrm{gr}_z (W, S, \pp; \beta_1, \beta_1, \beta^o).
	\]
	Further,
	\[
	M (W, S, \pp; \beta_1, \beta_0, \beta^o)_d = \bigcup_{ z \, | \, \mathrm{gr}_z = d} M_z (W, S, \pp; \beta_1, \beta_0, \beta^o)
	\]
	as before.
	
	We now state the requisite genericity theorem we need for this class of perturbations.
	
	\begin{proposition}
		Let $\pi_i$, $i = 0, 1$ be perturbations such that the conclusions of \Cref{prop:irred} and \Cref{prop:morse} hold, assuming the hypothesis of them. Let $(W, S, \pp)$ be given equipped with a cylindrical metric with a orbifold singularities, $\check{g}$, as above. Then there are secondary perturbations $\pi_i'$ for $i = 0, 1$ and $\pi^o_k$ for every orbifold point such that for all $\beta_1, \beta_0, \beta^o$ and $z$, the moduli spaces $M_z (W, S, \pp; \beta_1, \beta_0, \beta^o)$ are regular.
	\end{proposition}
	
	\begin{proof}
		This is almost verbatim \cite[Proposition 3.8]{kmknot} except that in our set-up we have orbifold points. However, thanks to \Cref{thm:rm-sing}, we see that we can appeal to the the cited proposition once we verify that we can set the primary perturbation to be zero for the lens space boundary components. And this we can due to \Cref{non-deg}. 
	\end{proof}
	
	\begin{rmk}
		Strictly speaking, we need that the moduli spaces of trajectories on lens spaces are regular too with no primary perturbation for the gluing theory to apply as stated in \cite{kmknot}. This is true by the Weitzenbock formulae stated in \cite[pp. 94]{fu}.
	\end{rmk}

	\subsection{Energy Filtrations and Perturbations} \label{sec:energy}
	
	We will discuss the energy filtration which we shall need in the proof of the triangle. The main ideas appear in the \cite{mono} and we adapt the setup for singular instantons with the perturbations described in \cite{kmknot}. For this section, we will \emph{not} be fixing a perturbation that satisfies \Cref{prop:morse-smale}.
	
	Let $(W, S)$ be a cobordism from $(Y_1, K_1)$ to $(Y_0, K_0)$; we shall use $(Y, K)$ to denote either of them when it is not necessary to specify a particular end of the cobordism. Assume that $S$ has no closed components. Then, the cobordism map is defined after attaching cylindrical ends to the two boundary components; call this resulting manifold with the orbifold structure $(\check{W}, \check{S})$. We now consider a particular form of perturbations which are different from the one used earlier. To define these, we first think of $W$ as obtained by attaching handles to the three manifold $Y_1$. Let the union of the attaching regions of these handles be denoted by $H$. Note that then $Y_1 \setminus (H \cup \nu(K_1)) \times \R \subset \check{W} \setminus \check{S}$ as the handles for the 4-manifold and the 2-manifold are attached along $H \cup \nu(K_1)$. Thus, any map $q:S^1 \times D^2 \to Y_1 \setminus (H \cup \nu(K_1))$ can be extended to the map $\tilde{q} : S^1 \times D^2 \times \R \to \check{W} \setminus \check{S}$ which is `constant' along the $\R$ factor. 
	
	Hence, if $q: S^1 \times D^2 \to Y_1 \setminus (H \cup \nu(K_1))$ is an embedding, $\mu$ is a compactly supported $2$-form on the interior of $Y_1 \setminus (H \cup \nu(K_1))$ such that $q^* (\mu)$ is a compactly supported volume form on $D^2$, $h:SU(2) \to \R$ is a conjugation invariant function and $\mathrm{Hol}_x$ denotes holonomy along $q(S^1 \times \{ x \})$, we can form the usual 3-dimensional holonomy perturbation
	\[
	f_q (A) = \int_{Y_1 \setminus (H \cup \nu(K_1))} h(\mathrm{Hol}_x (A)) \mu
	\]
	and extend it to 4-dimensions to form the perturbed ASD equation:
	\[
	F_A^+ - P_+(h'(T_x(B_t)) \mu) = 0
	\]
	where $A = B + cdt$ on $Y_1 \setminus (H \cup \nu (K_1)) \times \R$ and the metric is a product on it. We abbreviate $\mathrm{Hol}_x (A)$ to $T_x(A)$ in the following. Note that $P_+(\mu) = P_+(*_3 \mu \wedge dt)$ so that $P_+(\grad f_q (B_t) \wedge dt) = P_+(-h'(T_x(B_t)) \mu )$.
	
	We will also use  \[\mathrm{Hol}_q(A)\] to denote $-h'(T_x(B_t)) \mu$.

	Given this, we first consider the following perturbed Chern-Weil integral:
	\[
	\tilde{\kappa} (A) = \int_{\check{W} \setminus \check{S}} \tr[ (F_A - h' (T_x (A)) \mu )^2 ]
	\]
	Here, $\mu$ is the $2$-form pulled-back from $Y_1 \setminus (H \cup \nu(K_1))$ to a compactly supported form on the interior of $Y_1 \setminus (H \cup \nu(K_1)) \times \R \subset \check{W} \setminus \check{S}$. In particular, this implies that $\mu \wedge \mu = 0$. Now, 
	\[
	\tr (F_A - h'(T_x (A) ) \mu )^2 = \tr (F_A^2) - 2 \tr (F_A h' (T_x(A))) \mu.
	\]
	As $\mu$ is supported in $Y_1 \setminus (H \cup \nu(K_1)) \times \R$, we need only integrate on this subset. Denote the $\R$ coordinate by $t$. Further, assume that $A$ is in temporal gauge on this subset. Then, we have the following computation:
	\begin{align*}
		\int_{t \in \R} \int_{Y_1 \setminus (H \cup \nu(K_1))} \tr (- h'(T_x(A_t))  F_{A} \mu) 
		&= \int_{t \in \R} \int_{Y_1 \setminus (H \cup \nu(K_1))} \tr(h'(T_x(A_t)) \partial_t A_t \wedge \mu) \wedge dt \\
		&= \int_{t \in \R} -\bigg[  \partial_t \bigg( \int_{Y_1 \setminus (H \cup \nu(K_1))} h(T_x(A_t)) \mu \bigg) \bigg] dt \\
		&= \int_{-\infty}^\infty - \partial_t f_q(A_t) dt \\
		&= f_q(A_{-\infty}) - f_q(A_{\infty})
	\end{align*}
	where we used that $F_A = -\partial_t A_t \wedge dt + F_{A_t}$ and $F_{A_t} \wedge \mu = 0$ in the first line. 
	
	Putting this together, we have the following conclusion:\\
	\begin{align*}
		\tilde{\kappa} (A) &=  \int_{\check{W} \setminus \check{S}} \tr[ (F_A - h' (T_x (A)) \mu )^2 ] \\
		&= 2(f_q(A_{-\infty}) - f_q(A_{\infty})) + \int_{\check{W} \setminus \check{S}} \tr (F_A^2)  \\
		&= 2(f_q(A_{-\infty}) - f_q(A_{\infty})) + \kappa(A) \\
		&= CS_q(A_{-\infty}) - CS_q(A_{\infty})
	\end{align*}
	where $CS_q(B) = CS(B) + 2f_q(B)$; here we are \emph{not} working with gauge representatives but with explicit connections. We shall refer to such perturbations as \emph{time independent} perturbations.
	
	Hence, we see that if $A$ satisfies 
	\[
	F_A^+ + \mathrm{Hol}^+_{{q}} (A) = 0
	\]
	then,
	\begin{align*}
		\tilde{\kappa} (A) &= \int_{\check{W} \setminus \check{S}} \tr [ (F_A + \mathrm{Hol}_q(A))^2 ] \\
		&= \lVert F_A^- + \mathrm{Hol}_q^-(A) \rVert^2 - \lVert F_A^+ + \mathrm{Hol}_q^+(A) \rVert^2 \\
		&= \lVert F_A^- + \mathrm{Hol}_q^-(A) \rVert^2 \\
		&\geq 0.
	\end{align*}
	In particular, $\tilde{\kappa}(A) = 0$ implies 
	\begin{equation} \label{eqn:deformed}
	F_A + \mathrm{Hol}_q(A) = 0. 
	\end{equation}
	We shall refer to this as the \emph{deformed flatness} equation.
	
	\begin{rmk}
		While all the above constructions used a single framed knot to make the holonomy perturbations along, one can just as easily adapt this framework in the construction of the Banach space of perturbations and ensure that we have a Banach (subspace) of time independent perturbations. We will use the notation $\mathrm{Hol}_\pi$ and such to refer to such perturbations.
		
		As the only constraint is that the perturbations are along framed knots $q$ are disjoint from a knot $K_1$ in $Y_1$ and from the attaching regions $H$ (which is just a neighbourhood of a union of positive codimension sub-manifolds), we can ensure that transversality is attained for the three-dimensional equations and for cylinders by this subspace of perturbations as they are dense in the space of all perturbations. Hence forth, we shall assume the perturbations are made from this subspace so that they satisfy the conclusions of \Cref{prop:morse-smale}. 
	\end{rmk}
	
	A special situation where the above constructions are useful is recorded below as a lemma.
	
	\begin{lemma} \label{dflat}
		Suppose that $\{q_i\}$ is a collection of embeddings with image away from the attaching regions and the singular locus and the resulting perturbation be $\mathrm{Hol}_\pi$. Then, if $A$ is a solution to the time independent perturbed equations on $(W, S)$ with $[A_{-\infty}] = [A_\infty]$ and $\kappa(A) = \int_{\check{W} \setminus \check{S}} \tr (F_A^2) = 0$, then $A$ satisfies the deformed flatness equation~\eqref{eqn:deformed}.
	\end{lemma}
	
	\begin{proof}
		The only additional observation needed is that $f_\pi$ is gauge invariant which allows one to conclude that $\tilde{\kappa} (A) = 0$.
	\end{proof}
	
	These time independent perturbations are not sufficient to ensure transversality for the four dimensional equations. However, we can still use metric perturbations on $W \setminus (S \cup \im (\pi))$ without affecting the results above. Here, $\im(\pi)$ is the union of the image of all the $\{\tilde{q}_i\}$ used to form the perturbation $\pi$.
	
	For the rest of this section we will primarily follow \cite[Chapter 3]{fu}. To set things up, we will deal with connections directly rather than connections modulo gauge. Fix connections $B_1$ and $B_0$ on the ends $(Y_1, K_1)$ and $(Y_0, K_0)$ respectively and a bundle $\pp$ over $(W, S)$ such that all $\pp|_{(Y_i, K_i)}$ is non-integral. Then, we can choose a connection $A_o$ that interpolates between these; fix one such choice that is constant along the the cylindrical ends outside a compact set. Then, we can form the parametrised equations as follows:
	\begin{align*}
		\Pa : \lp[k, A_o]{\Lambda^1}{W} &\times \G_r \to \lp[k-1]{\Lambda^+}{W} \\
		\< A_o + a, \varphi > & \mapsto P_+  \big( (\varphi^{-1})^* ( F_{A_o + a}  + \mathrm{Hol}_\pi (A_o + a)) \big)
	\end{align*}
	Here, the bundles respect the orbifold structure and the sections are orbifold sections as usual. The perturbation $\pi$ is a fixed time independent perturbation which makes the Chern-Simons functional on $(Y_i, K_i)$ Morse, and $\G_r = \{ \varphi \in C^r(\check{W}, GL(T\check{W})) \, : \, \varphi = \Id \text{ on } S \text{ and on } \im (\pi) \}$ parametrises the space of metrics under pull-back. We fix some $r \gg k$ from now on and note that $\G_r$ is a Banach space. We can also define an analogous map where we look at gauge equivalence classes of connections:
	\begin{align*}
		\bar{\Pa} : \mathcal{B}_{[A_o]}(W, S, \pp; [B_1], [B_0]) &\times \G_r \to \lp[k-1]{\Lambda^+}{W} \\
		\< [A_o + a], \varphi > & \mapsto P_+  \big( (\varphi^{-1})^* ( F_{A_o + a}  + \mathrm{Hol}_\pi (A_o + a)) \big)
	\end{align*}
	Denote $\bar{\Pa}_g = \bar{\Pa} (-, \varphi)$ when $g = \varphi^* g_0$.
	
	\begin{proposition} \label{prop:mtransverse}
		Suppose that there is an open set $U \subset \check{W} \setminus (\check{S} \cup \im (\pi))$ such that any connection $A$ for $(A, \varphi) \in \Pa^{-1} (0)$ restricted to $U$ is irreducible, then $\bar{\Pa}$ is transverse to $0$. 
	\end{proposition}

	\begin{proof}[Sketch of Proof]
		We need the linearised operator to be Fredholm to even begin the argument however this is provided by \cite{kmknot}. After that, the proof in \cite{fu} applies verbatim on the open set $\check{W} \setminus (\check{S} \cup \im (\pi))$ provided the necessary irreducibility hypothesis is met. The irreducibility hypothesis is used in a `local' way in \cite{fu}; thus, the irreducibility of restriction to $U \subset \check{W} \setminus (\check{S} \cup \im (\pi))$ suffices to carry out the proof.
	\end{proof}

	For full details, we direct the reader to \Cref{appendix:perturbations}. A standard application of Sard-Smale theorem \cite{sard} gives the following corollary.
	
	\begin{corollary} \label{cor:mpert}
		Assuming the hypothesis of \Cref{prop:mtransverse}, we have that for a generic metric $g$, $\bar{\Pa}_g$ is transverse to $0$.
	\end{corollary}
	
	The proposition also leads to the following corollary by applying Sard-Smale \cite{sard} to a parametrised setting.
	
	\begin{corollary} \label{cor:mparampert}
		If $Q$ is a compact manifold with boundary parametrising a space of metrics, such that the equations are transverse on $\partial Q$, then we can choose a metric perturbation that is supported in the interior of $Q$ which makes the parametrised moduli space transverse.
	\end{corollary}
	
	\subsection{Index Formulae}
	
	We state and prove the necessary index formula we shall need in this paper. Let $Z$ be a closed connected oriented 4-manifold with $\Sigma \subset Z$ a closed surface. We define the energy of a connection $A$ on a singular bundle $\pp$ on $(Z, \Sigma)$ by the following:
	
	\[ \kappa(A) = \frac{1}{8 \pi^2 } \int_{Z \setminus \Sigma} \tr(F_A \wedge F_A)  \]
	
	The relevant fredholm operator in this setting is the following:
	\[ \D_{A, Z} = - d_A^* \oplus d_A^+ : \lp{\Lambda^1}{Z} \to \lp[k-1]{(\Lambda^0 \oplus \Lambda^+)}{Z} \]
	
	Then, we have the following:
	
	\begin{theorem}[{\cite[Lemma 2.11]{kmknot}}]\label{thm:ind0}
		\[ \mathrm{Ind} \D_{A, Z} = 8\kappa(A) - \frac{3}{2} (\chi(Z) + \sigma(Z) ) + \chi (\Sigma) + \frac{1}{2} \Sigma \cdotp \Sigma \]
		\qed
	\end{theorem}

	We now state and prove an analogous result when we allow isolated cone singularities as in \Cref{defn:metrics}.
	
	\begin{theorem} \label{thm:ind}
		Let $W$ be a closed oriented 4-manifold with a closed surface $\Sigma$ and a single isolated cone singular point $p \in W \setminus \Sigma$. Suppose further that a conical neighbourhood of $p$ is isomorphic to a cone on $L(2,1)$ and that $g$ is a compatible orbifold metric in the sense of \Cref{defn:metrics}. Then, given a singular bundle $\pp$ on $(W, \Sigma, \{p\})$ such that the restriction of $\pp$ to a conical neighbourhood of $p$ is trivial, we can form the singular ASD operator
		\[
		\D_{A, W} = - d_A^* \oplus d_A^+ : \lp{\Lambda^1}{W} \to \lp[k-1]{(\Lambda^0 \oplus \Lambda^+)}{W}.
		\]
		This operator is Fredholm and has index
		\[
		\mathrm{Ind} \D_{A, W} = 8\kappa(A) - \frac{3}{2} (\chi(W) + \sigma(W) ) + \chi (\Sigma) + \frac{1}{2} \Sigma \cdotp \Sigma.
		\]
	\end{theorem}

	\begin{proof}
		Both parts will be accomplished by a (linear) excision argument of the ASD operators on the pairs of 4-manifolds broken as $(W \setminus cL(2,1), cL(2,1))$ and $(TS^2, T^*S^2)$. Note that the latter pair can be glued along the boundary to obtain $S^2 \times S^2$. Depending on the limiting flat connection near $p$, we choose $w_2$ of the bundle over $TS^2$ and $T^*S^2$ to be represented either by the zero section or not. With this, we have the corresponding flat connections on $TS^2, T^*S^2$ and $cL(2,1)$ and the connection of interest $A$ on $W \setminus cL(2,1)$.
		
		The standard excision argument, as in \cite[Appendix B]{excision}, shows that $\D_{A, W}$ is Fredholm and gives the following:
		\begin{align*}
		\ind \D_{A, W} + \ind \D_{A_0, S^2 \times S^2} &= \ind \D_{A', (W \setminus cL(2,1)) \cup T^*S^2} + \ind \D_{A_1, Th(TS^2)} \\
		&= 8\kappa(A) - 3(1-b_1(W)+b^+(W)) + \chi (\Sigma) + \frac{1}{2} \Sigma \cdotp \Sigma - 6
		\end{align*}
		where we used \Cref{thm:ind0} and that for $A_1$ we can compute the index directly by working equivariantly on the double branched cover with the Fubini-Study metric on it, to see that $H^1_{A_1} = 0$ and $\dim H^0_{A_1} = \dim H^+_{A_1} = 3$. Finally, noting that $\ind \D_{A_0, S^2 \times S^2} = -6$ by using \Cref{thm:ind0} with empty singular locus, the proof is complete.
	\end{proof}

	We also note the gluing formula for the index when dealing with weighted spaces. We only quote the result as stated in \cite[\S 4.3]{sca} referring to \cite{dym} for details.
	
	\begin{proposition}[{\cite[\S 4.3]{sca}}] \label{prop:gluing}
		Let $(X_i, S_i)$ be pairs ($i = 1, 2$) with cylindrical ends and potentially isolated orbifold points disjoint from $S_i$. Suppose that these have a common boundary component $(Y, K)$. Now, if $A_i$ are connections on $(X_i, S_i)$ which are asymptotically flat on the cylindrical ends such that the flat connections on the component $(Y, K)$ are equal, to say $b$. Then, we can form $A_1 \cup A_2$ and we have
		\[
		\ind (A_1 \cup A_2) = \ind (A_1) + \ind (A_2) + h^0(b) + h^1(b).
		\]
		Here, $h^j(b) = \dim H^j(Y, K; \ad_b)$, the dimension of the cohomology groups of $\ad_b$. \qed 
	\end{proposition}

	\subsection{Maps from Cobordisms} \label{sec:cob-maps}
	We use an alternate description of the moduli spaces involved in the setting with orbifold points so that we can import the result regarding their orientability from \cite{kmknot}. The needed result amounts to an analogue of Uhlenbeck's removable singularities theorem but in the case of orbifold points. \Cref{thm:rm-sing} serves this purpose. Due to this result and the arguments of \cite{kmknot}, we see that all our moduli spaces $M_z(W, S, \pp; \beta_1, \beta_0)$ are orientable. Unless stated otherwise, 4-manifold $W$ has no orbifold points for this sub-section.
	
	We further examine the choices made in determining an orientation in the case when $W$ has two boundary components. Recall that we fixed a choice of classes $\theta_i \in \CB(Y_i, K_i, \pp_i)$ for $i =0, 1$ while specifying $\bold{a}_i$. Remember that $\Lambda (\beta_i)$ is short for $\Lambda (\theta_i, \beta_i)$. Pick representatives $\Theta_i$ for the classes $\theta_i$, and let $A_\Theta$ be a connection on $(W, S, \pp)$ which agrees with $\Theta_i$ at the ends. Define 
	\[
	\Lambda(W, S, \pp)
	\]
	to be the orientations of $\det (\D_{A_\Theta})$.
	
	\begin{defx}
		Let $(Y_i, K_i, \pp_i)$ (for $i = 0, 1$) be manifolds with singular bundle data, and let $(W, S, \pp)$ be a cobordism from the $i = 1$ to $i = 0$. Let the auxiliary bundle data $\boldsymbol{a}_i$ be given on the two ends, as above. We define an $I$-\emph{orientation} of $(W, S, \pp)$ as the choice of an element of $\Lambda(W, S, \pp)$.
	\end{defx}
	
	The above definition is made so that we can naturally identify orientations of the determinant line bundle over $\CB_z(W, S, \pp; \beta_1, \beta_0)$ with
	\[
	\Lambda (\beta_1) \Lambda (\beta_0) \Lambda(W, S, \pp).
	\]
	Having picked an $I$-orientation, every $[A] \in M(W,S, \pp; \beta_1, \beta_0)_0$ defines
	\[
	\varepsilon ([A]) : \Z \Lambda (\beta_1) \to \Z \Lambda (\beta_0).
	\]
	We define a homomorphism in this situation by
	\[
	\<m_W(\beta_1), \beta_0> = \mathlarger{\sum}_{[A] \in M(W,S, \pp; \beta_1, \beta_0)_0} \varepsilon ([A]),
	\]
	giving rise to a chain map 
	\[
	m_W : C_*(Y_1, K_1, \pp_1) \to C_*(Y_0, K_0, \pp_0).
	\]
	Upto chain homotopy, this depends only on $(W, S, \pp)$, $\bold{a}_i$, and the $I$-orientation.
	
	The above discussion goes over verbatim when $W$ has more ends but the boundary critical points on these other ends are fixed. In particular, assuming all these other boundary components are distinguished and $S$ is admissible in the sense of \Cref{defn:metrics}, we can define 
	\[
	m_{\beta^o} : C_*(Y_1, K_1, \pp_1) \to C_*(Y_0, K_0, \pp_0)
	\]
	by 
	\[
	\<m_{\beta^o} (\beta_1), \beta_0> = \mathlarger{\sum}_{[A] \in M(W',S, \pp; \beta_1, \beta_0, \beta^o)_0} \varepsilon ([A])
	\]
	where $W'$ is the orbifold completion of $W$ and $\beta^o$ is the data of critical points on the orbifold points of $W'$. The reason the moduli spaces are oriented is due to above discussion above in conjunction with \Cref{thm:rm-sing} as alluded to in the beginning of this subsection.
	
	Finally, summing over all the possible $\beta^o$, we get a map 
	\[
	m_{W'} : C_*(Y_1, K_1, \pp_1) \to C_*(Y_0, K_0, \pp_0)
	\]
	defined for any orbifold cobordism. There is \emph{a priori} no reason that this should define a \emph{chain map}. The issue is that, using the alternate description of moduli spaces as in \Cref{thm:rm-sing}, there could be breaking of trajectories along the distinguished boundary components that could lead to extraneous terms contributing to the standard compactness-gluing argument. The following lemma will allow us to overcome this difficulty when the distinguished boundary component is a $L(2,1) \cong \RP^3$ and the bundle over it is trivial.
	
	\begin{lemma} \label{lem:no-breaking}
		Let $\pp$ denote the trivial $\so$ bundle over $L(2,1)$ let $A$ be an ASD connection on $\R \times L(2,1)$ with bundle $\pp$ pulled back, in the appropriate weighted moduli space. Then, if $\kappa(A) > 0$, $\ind (A) \neq 1$ and $A$ is a transversely cut out point in the moduli space.
	\end{lemma}

	\begin{proof}
		The transversality observation is due to the Weitzenbock formula (cf. \cite[pp. 94]{fu}). We now proceed to show that $\ind (A)  \neq 1$. Since $A$ is an element of the appropriate weighted moduli space, we know it limits to flat connections on the two ends. Let $\theta_\pm$ denote the two flat connections on the trivial bundle on $L(2,1)$ up to gauge equivalence where $\theta_+$ is the trivial connection and the representation defined by $\theta_-$ sends the generator to the non-trivial element of the centre of $SU(2)$.
		
		First, consider the case when the limiting flat connections are equal, say $\alpha$. Then, gluing the two ends together, we get a connection $A'$ with index computed by the gluing formula (\Cref{prop:gluing}) as,
		\[
		\ind (A') = \ind (A) + h^0(\alpha) + h^1(\alpha) = \ind (A) + 3.
		\]
		
		Further, since $A'$ is defined on the closed manifold $S^1 \times L(2,1)$, we also have 
		\[
		\ind (A') = 8\kappa(A') - 3(1-1+0) = 8\kappa(A).
		\]
		We conclude that $\ind (A) = 8\kappa(A) - 3$ where 
		\[
		\kappa(A) = -\frac{1}{4} p_1(\pp') \geq 0
		\]
		where $\pp'$ is equivalent to $\pp$ pulled back. We are writing $p_1$ even though we really mean its pairing with the compact fundamental class and we shall continue to do so. But, by the results of \cite{dw}, we see that $p_1(\pp') \equiv 0 \pmod 4$. Hence, we see that $\ind (A) \geq 5$ in this case.
		
		If the limiting connections are different, we can cap off each end by $D(T^* S^2)$ or $D(TS^2)$ depending on the orientation of the end to form a closed manifold diffeomorphic to $S^2 \times S^2$. Further, by specifying by the Steifel-Whitney class to be the zero section on the manifold attached to the end with limit $\theta_-$, we get an ASD connection $A'$ on $S^2 \times S^2$ with bundle $\pp'$ and $w_2(\pp') = P.D. ( S^2 \times \{*\} \cup \{*\} \times S^2 )$; here we use the canonical flat connections on the disk bundles compatible with the bundle, call these $A_\pm$. By this construction, we have 
		\begin{align*}
		\ind (A') &= \ind (A) + \ind (A_-) + \ind (A_+) + h^0(\theta_-) + h^0(\theta_+) + h^1(\theta_-) + h^1(\theta_+) \\
		&= \ind (A) + (-3-3) + (-3) + (3) + (3) + (0) + (0) \\
		&= \ind (A) - 3.
		\end{align*}
		
		Using the index formula in the closed case, we have 
		\[
		\ind (A') = 8 \kappa (A') - 3(1-0+1) = -2 p_1(\pp') - 6.
		\]
		
		Here, we have $p_1(\pp') \leq 0$ as before and \cite{dw} tells us that in this case $p_1(\pp') \equiv 2 \pmod 4$.Putting this all together, we have
		\begin{align*}
			\ind (A) &= \ind (A') - 3 \\
			&= -2p_1(\pp') - 9 \\
			&\equiv 3 \pmod 8.
		\end{align*}
		
		We conclude that $\ind (A)  \neq 1$ in this case as well.
	\end{proof}
	
	This lemma tells us that there is no breaking along the $\RP^3$ components and so the map $m_{W'}$ in this case is a genuine chain map. Recall that the gluing can potentially involve a gluing parameter and so the case of interest to rule out is the case $\ind(A) = 1$. We suspect that the analogous result holds for all $L(p, q)$ but the argument given above does not generalise.
	
	\subsection{Local Coefficients}
	\label{sec:local-coefficients}
	
	We describe a local coefficient system that we will use following \cite{yaft}. For a more detailed and general discussion, we refer the reader to \cite[\S 3.9]{yaft}. For this section, we will assume that our singular loci are all oriented.
	
	Suppose $K \subset Y$ is a link with components $K_i$ where $i = 1, \cdots, n$. Fix a framing for $K$. Recall from \cite[\S 3.1]{kmknot} that we have a reduction of structure group to $U(1)$ in a tubular neighbourhood of $K$. Using the framing of $K$, we can  talk about $\mathrm{Hol}_{\lambda_i} (B)$ for any singular connection $B$ where $\lambda_i$ are the longitudes of $K_i$ specified by the framing of $K_i$. 
	
	Formally, 
	\[
	\mathrm{Hol}_{\lambda_i} (B) = \lim_{\epsilon \to 0} \mathrm{hol}_{K_{i, \epsilon}} (B) \in U(1)
	\]
	where $K_{i, \epsilon} = \{ \epsilon \} \times K_i \subset D^2 \times K_i \cong \nu(K_i)$ where the latter identification is the chosen framing of $K_i$.
	
	These give us a function
	\[
	\mathrm{Hol}_{\lambda} : \CB = \CB_k (Y, K, \pp) \to U(1)^n.
	\]
	
	We can now pull-back the local coefficient system $U(1)$ described in \cite[\S 3.9]{yaft} to get one on $\CB$. Before we describe it, we note that the framing does not affect this definition in any substantial manner (cf \cite[\S 3.9]{yaft}). Explicitly, this local system has fibre
	\[
	\Gamma_B = \Z [t_1^{\pm 1}, \cdots, t_n^{\pm 1}] \cdotp \left( \prod_{i=1}^{n} t_i^{\frac{1}{2 i \pi}\log (\mathrm{Hol}_{\lambda_i}(B))}\right)
	\]
	at $B \in \CB$. 
	
	Any cobordism of pairs $(W, S)$ also naturally gives a map between the local systems as defined above. For simplicity, assume that $\partial S$ has two components, each on one end of the cobordism. Now, we can pick a framing of $S$ and a reduction of the singular bundle to $U(1)$ along it (here we need $S$ to be oriented) and suppose $A$ is an ASD connection on $(W, S)$ with limits $B_\pm$ on the ends.
	
	Define $\Delta_A (p) = p \cdotp t^{\nu (A)} : \Gamma_{B_-} \to \Gamma_{B_+}$ where $\nu(A)$ is the monopole number of $A$. The monopole number in this context is defined by the following equation:
	\[
	\nu (A) = \frac{i}{2\pi} \lim_{\epsilon \to 0} \int_{S_\epsilon} F_A
	\]
	where $S_\epsilon = \{ \epsilon \} \times S \subset \nu(S) \cong D^2 \times S$ where the latter identification is provided by the framing of $S$. Again, this definition does not depend on the framing of $S$ in any essential way. 
	
	So far the discussion is mostly a restatement of the material of \cite[\S 3.9]{yaft}. Now, we want to extend this definition to the case when $S = S' \cup S_{model}$ where $S'$ is a product (we allow $S' = \emptyset$) and $S_{model}$ is either a disk, a sphere or a disjoint union of two disks with the boundary on different ends of the cobordism.
	
	In each of the three cases (assuming $S' = \emptyset$), we attach an extra variable $T$ on the boundary components with empty link and the default variable $t_0$ (we use $t_0'$ to distinguish between potentially different knots) is associated to the knots in the boundary. 
	
	We \emph{define} the action on local coefficients as
	\begin{align*}
		\Delta_A (p(t_0)) &= p(T) \cdotp T^{\nu (A)} \quad \text{if } S_{model} \cong D^2 \text{ and } \partial_{+} S_{model} = \emptyset \\
		\Delta_A (p(T)) &= p(t_0) \cdotp t_0^{\nu (A)} \quad \text{if } S_{model} \cong D^2 \text{ and } \partial_{-} S_{model} = \emptyset \\
		\Delta_A (p(T)) &= p(T) \cdotp T^{\nu (A)} \quad \text{if } S_{model} \cong S^2 \\
		\Delta_A (p(t_0')) &= p(t_0) \cdotp {t_0}^{\nu (A)} \quad \text{if } S_{model} \cong D^2 \cup D^2 \text{ and } \partial_{\pm} S_{model} \neq \emptyset.
	\end{align*}
	These have a clear extension to the case $S' \neq \emptyset$ which is only difficult to write down due to notational complexity. We end this section with two observations about the above construction:
	\begin{enumerate}[label=(\roman*)]
		\item In all but the case of $S_{model}$ being closed, the framing of $S_{model}$ (and consequently $K$) matter in the above definition as, unlike the earlier discussion, there are no canonical isomorphisms on the local system on the 3-manifold chain groups intertwining with the canonical isomorphisms with the local system on the knot chain groups induced by the change of framing.
		
		\item The usual composition laws as stated in \Cref{sec:metrics-compactness-gluing} continue to hold.
		
	\end{enumerate}
	
	\subsection{Families of Metrics}
	\label{sec:metrics-compactness-gluing}
	
	In this section, we focus on the aspect of perturbations when using a family of \emph{cut metrics} (in the sense of \cite[\S 3.9]{kmknot}) to construct homotopies. For fuller details on the construction of these metrics, we direct the reader to \cite[\S 3.9]{kmknot}. We will also freely borrow terminology from the aforementioned section without further comment.
	
	Suppose that $(W, S, \pp)$ is a cobordism from $(Y_1, K_1, \pp_1)$ to $(Y_0, K_0, \pp_0)$ and we have a family of metrics $G \cong [0, +\infty]$ where $+\infty$ corresponds to having a \emph{cut} hypersurface pair $(Y_c, K_c)$. For this discussion, we further assume $Y_c$ is connected, closed and that $Y_c \subset \text{int}(W)$.
	
	To explicitly describe the family of (orbifold) metrics, let $r$ denote the normal coordinate to $Y_c$. Then, the family of metrics takes the form
	\[ \check{g}_s = f_s(r)^2 dr^2 + \check{g}_{Y_c} \]
	where $f_s(r) = 1$ for every $s \in [0, +\infty]$ when $|r| > 1$ and when $r \in [-1, 1]$, $f_s(r) \to \frac{1}{r}$ as $s \to +\infty$. Recall that $r = 0$ corresponds to the hypersurface $Y_c$.
	
	Thus, when $s \gg 0$, we have an isometric copy of $[-T_s, T_s] \times Y_c$ embedded in $W$. In these coordinates, the perturbations take the form
	\[
	\hat{V}_{\pi_c} + \psi_-(t) \hat{V}_{(\pi_{c, -})'} + \psi_+(t) \hat{V}_{(\pi_{c, +})'}
	\]
	where the functions $\psi_\pm$ are bump functions supported near $t= \pm T_s$. While the primary perturbation $\pi_c$ is required to be independent of $s$, the secondary terms $(\pi_{c, \pm})'$ are allowed to depend on $s$. The appropriate transversality result is recorded in \cite[Proposition 3.24]{kmknot}.
	
	To bring the analogue of the material in \Cref{sec:energy} to bear in this context, we allow perturbing the metric $g_s$ in the region $t \in [-T_s-1, -T_s] \cup [T_s, T_s + 1]$. Adapting the arguments of \cite[\S 3.9]{kmknot} as we did in the previous situation in \Cref{appendix:perturbations}, we have the analogous results of \Cref{sec:energy} in this situation. Specifically, \Cref{cor:mparampert} continues to hold when the space of metrics parametrises cut metrics as above.
	
	The above discussion naturally generalises to the case when $Y_c$ has $n$ connected components (in which case $G \cong [0, +\infty]^n$); we refer the reader to \cite[\S 3.9]{kmknot} for specifics and orientation conventions in this situation.
	
	Finally using the above set-up and orientation conventions, we record the an equation which we shall use often. To set the stage for it, let $G$ be a compact family of cut metrics such that $\pd[] G = \sqcup_j G_j' \times G_j''$ with the cut hypersurface $Y_c^j$. Suppose further that the metrics $G_j'$ and $G_j''$ are families of metrics on the two components of $(W \setminus Y_c^j, S \setminus S \cap Y_c^j)$. Denote $m_j^\circ = m_{G_j^\circ}$ where $\circ \in \{', ''\}$. Then,
	
	\begin{equation} \label{eqn:parametrised-metrics}
		\sum_j (-1)^{\dim G''_j \cdotp \dim G'_j} m''_j \circ m'_j = \partial \circ m_G - (-1)^{\dim G} m_G \circ \partial.
	\end{equation}
	
	\begin{rmk}
		This formula continues to hold when there are isolated orbifold points of order $2$ due to \Cref{lem:no-breaking}.
	\end{rmk}
	
	\subsection{Cobordism with Non-degenerate Ends}
	\label{sec:cobordism-non-degen}
	
	We will have need to deal with a situation where our 4-manifold with have non-degenerate ends. Since we need only a particular topological set-up, rather than introduce new definitions and prove the requisite statement in any generality, we settle for the specific cases we will need.
	
	\begin{proposition} \label{prop:non-degenrate-ends}
		Suppose $(W, S)$ is a pair with $\partial (W, S) = -(Y, K) \sqcup (Y, K) \sqcup (S^2 \times S^1, \emptyset)$ or $\partial (W, S) = -(Y, K) \sqcup (Y, K) \sqcup (S^2 \times S^1, \{p_n, p_s\} \times S^1)$; denote the last boundary pair by $(X, C)$. Let $\pp_W$ be the bundle over it such that $\pp_W|X$ is trivial and $\pp_W$ satisfies the non-integrality condition. 
		
		Let $Q$ be a space of metrics and let $\U \subset \chi_{SU(2)} (X, C)$. Assume that if $\dim Q = 0$, $\U$ is  open interval in the interior of $\chi_{SU(2)} (X, C) \setminus \chi_{SU(2)} (X, C)^{\text{central}}$ and if $\dim Q > 0$, assume $\U$ to be a finite collection of points in the interior of $\chi_{SU(2)} (X, C) \setminus \chi_{SU(2)} (X, C)^{\text{central}}$.
		
		Then, we can form a linear map 
		\[
		m(W, S, \pp_W, \U) :  C(Y, K) \to C(Y, K)
		\]
		by the equation 
		\[
		\< m (W, S, \pp_W, \U) (\mathfrak{a}), \mathfrak{b}> = \# M(W, S, \pp_W; \mathfrak{a}, \mathfrak{b}, \U)_0.
		\]
		
		Here, if $\U$ is not a finite collection of points, we have chosen appropriate perturbation such that 
		\begin{equation} \label{eqn:transverality}
			M(W, S, \pp_W; \mathfrak{a}, \mathfrak{b}, \partial \U) = \emptyset
		\end{equation}
		when expected dimension of $M(W, S, \pp_W; \mathfrak{a}, \mathfrak{b}, \U)$ is $0$.
		
		Then, no breaking of trajectories along $(X, C)$ occurs and so the map is a chain map as usual.
	\end{proposition}
	
	\begin{proof}
		First, we compute a lower bound on the energy needed to break off a ASD connection on the $S^2 \times S^1$ end. Suppose $A$ is an ASD connection on $S^2 \times S^1 \times \R$ with boundary limiting connections in $\U$. By capping off the ends by $S^1 \times D^3$ equipped with flat connections $A_\pm$, we then have 
		\begin{align*}
			8 \kappa (A) &= \ind (A_- \cup A \cup A_+) \\
			&= -1 + 2 + \ind (A) + 2 + -1 \\
			&= \ind (A) + 2.
		\end{align*}
		Here, we have used that any point in the interior of $\chi_{SU(2)}(X, C) \setminus  \chi_{SU(2)} (X, C)^{\text{central}}$ has $h^0 = h^1 = 1$ and \Cref{prop:gluing}. Additionally, we use $\ind(A_\pm) = -1$ which can be seen by direct computation. Hence, if $A$ is not flat, $\ind (A) \geq 6$ proving the first assertion in the case of $C = \emptyset$. The other case is exactly the same.
		
		Secondly, assuming $\U$ is not a finite collection of points, we prove the existence of the perturbation asserted when $C = \emptyset$. Fix $\mathfrak{a}$ and $\mathfrak{b}$ such that the expected dimension of $M(W, S, \pp_W; \mathfrak{a}, \mathfrak{b}, \U)$ is 0. Since $\U$ is an open set, the expected dimension of any connection in associated configuration space is same as that of any connection in $\CB(W, S, \pp_W; \mathfrak{a}, \mathfrak{b}, \chi_{SU(2)} (S^2 \times S^1) \setminus \chi_{SU(2)} (S^2 \times S^1)^{central})$. A linear gluing argument as above shows that any such connection $A$ must have $\ind (A) = -1$. Since $A$ is irreducible due to the non-integrality condition, we see that we can choose perturbations of the form described in \Cref{sec:energy} (see \Cref{cor:mpert}) to ensure such $A$ do not appear with the limiting connection on $S^2 \times S^1$ being a fixed $\rho \in \chi_{SU(2)} (S^2 \times S^1)$. As $\U$ is an open interval, it has two end points. If both end points are not in $\chi_{SU(2)}(S^2 \times S^1)^{central}$, we're done. Else, we need only note that the linear gluing argument shows that $A$ must have had $\ind (A) = -3$ in the case its limiting connection $\rho$ on $S^2 \times S^1$ is central.
	\end{proof}
	
	\begin{rmk} \label{rmk:non-degen-transversality}
		The above proof in addition shows that along a 1-parameter family of perturbations, could be connections on $(W, S)$ whose limiting flat connection on $(X, C)$ is in $\partial \U$. Hence, $m$ is not independent (up to chain homotopy) of the choice of perturbations. However, since the condition \eqref{eqn:transverality} is an open condition, we can still vary the perturbation inside an open set to satisfy it. Suppose two maps $m_1$ and $m_2$ are related by changing the perturbation in such an open set, we will denote it by 
		\[
		m_1 \simeq_{\U} m_2.
		\]
		
		In particular, note that we can take the metric perturbations as described in \Cref{sec:energy}. In general, there is wall crossing behaviour which contributes terms coming from moduli spaces of the form $M(W, S, \pp_W; \mathfrak{a}, \mathfrak{b}, \partial \U)$; this phenomenon appears in \Cref{claim:proj}.
	\end{rmk}
	
	\section{Energy Ordered Morphisms} \label{sec:energy-order}
	
	In this section we shall give a way to order the critical points of the (perturbed) Chern-Simons functional so that some morphisms can be expressed as an upper triangular matrix with respect to this basis. To start, we use the following definition.
	
	\begin{defx} \label{defn:energy-basis}
		Let $\{\mathfrak{a}_i\}_{i = 1}^n$ be called an energy ordered basis for the chain group $C(Y, K)$ if the collection is the set of all critical points of a perturbed Chern-Simons functional and whenever $\mathrm{gr}(\mathfrak{a}_i) \equiv \mathrm{gr} (\mathfrak{a}_j) \pmod 4$ with $i < j$, then there exists a connection $A$ on $\R \times (Y, K)$ such that $[A|_{- \infty}] = \mathfrak{a}_j$ and $[A|_{+ \infty}] = \mathfrak{a}_i$ with $\ind (A) = 0$ and $\kappa (A) \geq 0$. Note that $A$ need not be ASD.
	\end{defx}
	
	\begin{rmk}
		This gives a total ordering on the standard basis elements in every fixed canonical $\Z / 4$ grading.
	\end{rmk}
	
	\begin{defx} \label{defn:energy-ordered-morphism}
		Suppose $(W, S)$ is a cobordism from $(Y, K)$ to itself, $\pp_W$ a bundle over it and $Q$ is some parameter space of metrics on it. Further, we require $\pp_W$ to restrict to isomorphic bundles, say $\pp_Y$ on the boundary components and be non-integral. Further, assume that $w_2(\pp_W) = 0 \in H^2(W, \partial W; \Z/2)$. Let $(W', S')$ be the pair obtained by gluing the two boundary components to each other. Then, if
		\[
		- \dim Q = -3(1 - b_1(W') + b^+(W')) + \chi(S') + \frac{1}{2} S' \cdotp S'
		\]
		we call the linear map 
		\[
		m_Q(W, S, \pp_W) : C(Y, K, \pp_Y) \to C(Y, K, \pp_Y)
		\]
		\emph{energy ordered} and also refer to the data $(W, S, \pp_W, Q)$ as \emph{energy ordered}.
	\end{defx}

	\begin{rmk} \label{rmk:grading-shift}
		We note that in the above context, we have $i(W, S, \pp_W) \equiv 0 \pmod 4$ by \cite[Proposition 4.4]{kmknot}.
	\end{rmk}
	
	\begin{proposition} \label{prop:energy-ordered-morphism}
		Let $(W, S)$ be a cobordism from $(Y, K)$ to itself, $\pp_W$ a bundle over it, and $Q$ a parameter space of metrics on it such that $L = m_Q(W, S, \pp_W)$ is energy ordered. Suppose $L_k$ denotes the restriction of $L$ to the chain subgroup of grading $k$. Then, $L_k$ maps into the chain subgroup of grading $k$.
		And, the matrix for $L_k$, as a matrix with entries in $\Z [T, T^{-1}]$, is upper triangular with respect to any energy ordered basis for every $k \in \Z / 4$. Further, if $\dim Q > 0$, then the diagonal entries of $L_k$ are zero.
	\end{proposition}
	
	\begin{proof}
		By \Cref{rmk:grading-shift}, the first assertion follows. Now, suppose $L_k = (L_{i, j})$ are the matrix entries with respect to an energy ordered basis $\{\mathfrak{a}_i\}_{i = 1}^n$ for the chain subgroup of grading $k$. Then, suppose that $i \leq j$ and there is an ASD connection $A$ on $(W, S)$ and a metric $q \in Q$ with 
		\[
		(q, [A]) \in M_Q(W, S, \pp_W; \mathfrak{a}_i, \mathfrak{a}_j)_0
		\]
		a transverse point.
		
		Then, $\ind (A) = - \dim Q$. Now if $A_0$ is a connection as in \Cref{defn:energy-basis}, then we can form a connection $A \cup A_0$ glued along the end with critical point $\mathfrak{a}_j$; we take $A_0$ to be flat if $i = j$. Now, as the limiting connections are gauge equivalent, we can glue them together to finally get a connection $B$ on $(W', S')$. By linear gluing and the index formula in the closed case, we have that 
		\begin{align*}
			-\dim Q &= \ind(A) + \ind(A_0) \\
			&= \ind (B) \\
			&= 8\kappa (B) - 3(1 - b_1(W') + b^+(W')) + \chi (S') + \frac{1}{2} S' \cdotp S'.
		\end{align*}
		
		Hence, as $(W, S, Q)$ is energy ordered, $0 = \kappa (B) = \kappa (A) + \kappa (A_0)$. Now $\kappa (A), \kappa(A_0) \geq 0$ as $A$ is ASD and as $A_0$ is as in \Cref{defn:energy-basis}. We conclude that in fact $\kappa (A) = \kappa (A_0) = 0$. The latter equation implies that $i = j$ and so we conclude that $L_{i, j} = 0$ if $i < j$.
		
		Next, when $i = j$, \Cref{dflat} applies to show that $A$ satisfies the deformed flatness equation. As this equation is independent of the metric, we see that the linearised ASD equations as a map 
		\begin{equation} \label{eqn:linearised-ASD}
			T_q Q \oplus H^1_A \to H^{+_q}_A
		\end{equation}
		has $T_q Q$ contained in its kernel. Now, $\ind (A) = - \dim Q$, along with irreducibility of $A$, implies that the map in \eqref{eqn:linearised-ASD} is surjective if and only if it is injective. Thus, a deformed flat connection can never be transverse which contradicts that $(q, [A])$ is transverse. We conclude that $L_{i, i} = 0$ if $\dim Q > 0$.
	\end{proof}
	
	For our applications, we need an extension of the above result to one more setting. We shall state and prove the appropriate result only in the specific case we need even though it is likely true in greater generality.
	
	\begin{proposition} \label{prop:energy-ordered-non-degenerate}
		Suppose $(W, S)$ is a pair with $\partial (W, S) = -(Y, K) \sqcup (Y, K) \sqcup (S^2 \times S^1, \emptyset)$ or $\partial (W, S) = -(Y, K) \sqcup (Y, K) \sqcup (S^2 \times S^1, \{p_n, p_s\} \times S^1)$; denote the last boundary pair by $(X, C)$. Let $\pp_W$ be the bundle over it such that $\pp_W|X$ is trivial. Then, we can form $\overline{W} = W \cup S^1 \times D^3$ ($\overline{S} = S$), in the first case and $(\overline{W}, \overline{S}) = (W \cup S^1 \times D^3, S \cup S^1 \times [-1, 1])$ in the second. 
		
		Let $\U \subset \chi_{SU(2)} (X, C)$ be an open interval in the interior of $\chi_{SU(2)} (X, C) \setminus \chi_{SU(2)} (X, C)^{\text{central}}$ if $\dim Q = 0$. If $\dim Q > 0$, assume $\U$ is a finite collection of points in the interior of $\chi_{SU(2)} (X, C) \setminus \chi_{SU(2)} (X, C)^{\text{central}}$. Then, by \Cref{prop:non-degenrate-ends}, we can form a linear map 
		\[
		m_Q(W, S, \pp_W, \U) :  C(Y, K) \to C(Y, K)
		\]
		by the equation 
		\[
		\< m_{Q} (W, S, \pp_W, \U) (\mathfrak{a}), \mathfrak{b}> = \# M_Q(W, S, \pp_W; \mathfrak{a}, \mathfrak{b}, \U)_0.
		\]
		Then, if $(\overline{W}, \overline{S}, Q)$ is monotone, then all the conclusions of \Cref{prop:energy-ordered-morphism} hold.
	\end{proposition}
	
	\begin{proof}
		Note that in the proof of \Cref{prop:energy-ordered-morphism} we show not just an energy ordered morphism is upper triangular but that the moduli spaces corresponding to the lower triangular entries is empty. This is also the situation for the last assertion in \Cref{prop:energy-ordered-morphism}. As such these proofs carry through in the present context once we glue in a flat connection on $S^1 \times D^3$ extending the one on $S^2 \times S^1$ in the case when $C = \emptyset$; the linear gluing argument as in the proof of \Cref{prop:non-degenrate-ends} allows the above proof to go through verbatim. The other case is analogous.
	\end{proof}
	
	\section{Preliminaries for the Proof}
	\label{sec:preliminaries}
	
	We state the main theorem with all its technical qualifiers first. Let $Y$ be a closed three-manifold and $L \subset Y$ be a link with a distinguished component $K$ which has a framing. Further, let $\pp$ be the data of a singular bundle on $(Y, L)$ such that it is represented by $\omega \subset Y \setminus L$ with the restriction that $\partial \omega \cap K = \emptyset$.
	
	Let $\GK$ denote the local system on $(Y_n(K), L \setminus K)$ for any $n \in \Z$ and $\Gamma_L$ the one on $(Y, L)$. Then $\Gamma_L$ is a laurent polynomial ring over $\GK$ with one variable, say $T_K$, i.e., $\Gamma_L = \GK[T_K^{\pm 1}]$. In this way, we will interpret $\Gamma_L$ as a module over $\GK$.
	
	\begin{theorem}\label{thm:main}
		The following is an exact triangle in the derived category of chain complexes of $\Gamma_L$ modules:
		\begin{equation*}
			\begin{tikzcd}[column sep=small]
				C_*(Y_0(K), L \setminus K; \GK) \otimes_{\GK} \Gamma_L \arrow[rr, "\psi"] & & C_*(Y_{2}(K), L \setminus K; \GK) \otimes_{\GK} \Gamma_L \arrow[dl, "f_1"] \\
				& C_*(Y, L; \Gamma_L) \arrow[ul, "f_2"]&
			\end{tikzcd}
		\end{equation*}
		The maps $f_1$ and $f_2$ are described explicitly by cobordism maps and interpreted as in \Cref{sec:local-coefficients}. When $T_K = 1$, the above is an exact triangle in the chain homotopy category of chain complexes and $\psi$ is also a linear combination of cobordism maps. The explicit descriptions are in \Cref{sec:complexes-maps}.
	\end{theorem}

	For the rest of the paper, for notational ease, we will describe the proof in the case when $L = K$ so that $\GK = \Z$. We will also abbreviate $T_K$ to $T$. The argument applies verbatim in the general context as above. 
	
	\subsection{Trace of Rational Surgery}
	\label{sec:rational-surgery}
	
	Let $Y$ be a three manifold with a knot $K$ with meridian $\mu$ specified by a framing. Denote by $\lambda$ the class in $H_1(\partial (Y \setminus N(K)) )$ with $\lambda \cdot \mu = 1$ where we orient the torus as the boundary of $Y \setminus N(K)$. Recall that $p/q$ surgery is defined as follows: 
	\[ Y_{p/q} \coloneqq Y_{p/q} (K) = (Y \setminus N(K)) \cup_{f_{p/q}} H\]
	where $H = D^2 \times S^1$ and $f_{p/q}: \partial H \to \partial (\yk)$ with $f_{p/q} (\partial D^2 \times \{ pt \})$ is homologous to $\mu_{p/q} = p \mu + q \lambda$. We define $K_{p/q} \subset Y_{p/q}$ as $\{ 0 \} \times S^1 \subset H$ inside $Y_{p/q}$ in the above identification, i.e., the dual knot. 
	
	Unless $q = \pm 1$, we have do not have a cobordism from $Y$ to $Y_{p/q}$ (trace of the surgery) a priori. However, we can instead construct a cobordism with a single orbifold singularity. Define $Z$ as follows:
	
	\[ Z = ([-2, -1] \times H_1) \cup_{\Id \times f_{1/0}} ([-2, 2] \times \yk) \cup_{\Id \times f_{p/q}} ([1,2] \times H_2)\]
	where $H_i$ are just $H$, labelled for future reference. 
	
	Note that $Z$ has three boundary components: $Y$, $Y_{p/q}$ and $L(q, p)$. To see how the last boundary component is formed, note that by definition it is obtained by gluing a copy of $H$ to each boundary of $[-1, 1] \times \partial(\yk)$ along $f_{1/0}$ and $f_{p/q}$. This is precisely the standard Heegaard splitting of $L(q,p)$. We can now cone the last component to get a $4$-manifold $W$ with a single orbifold point (the cone point) of order $q$. Note that if $q = \pm 1$, this recovers the usual trace of a surgery. We refer to this construction as attaching a orbifold handle.
	
	Continuing the analogy with trace of an integral surgery, we can talk about core and cocore of the orbifold handle. Formally, we define the core (resp. cocore) as the union of $[-2, -1] \times \{0\} \times S^1 \subset [-2, -1] \times H_1$ and the cone of $\{-2\} \times \{0\} \times S^1 \subset \{-2\} \times H_1 \; (\text{resp. }H_2)$. Note that these are orbifold disks in general with a single orbifold point of order $q$.
	
	We can generalise the above discussion to the cases when we need a cobordism from $Y_{p/q}$ to $Y_{r/s}$.
	
	\begin{example}
		Let $(Y, K) = (S^3, U)$, then we get an orbifold cobordism with one orbifold handle from $Y_0 = S^2 \times S^1$ to $Y_2 = \RP^3$. Note that this can't be achieved with a single standard handle however.
	\end{example}

	\subsection{Cuts for neck stretching} \label{sec:heg}
	
	\begin{figure}[!h]
		\tikzset{every picture/.style={line width=0.75pt}} %
		\begin{tikzpicture}[x=0.75pt,y=0.75pt,yscale=-1,xscale=1]
			\clip (25,25) rectangle (575, 250);
			\draw  [line width=1.5]  (50,50) -- (500,50) -- (500,200) -- (50,200) -- cycle ;
			\draw [color={rgb, 255:red, 245; green, 166; blue, 35 }  ,draw opacity=1 ][line width=1.5]    (200,50) -- (200,200) ;
			\draw [color={rgb, 255:red, 139; green, 87; blue, 42 }  ,draw opacity=1 ][line width=1.5]    (350,50) -- (350,200) ;
			\draw  [color={rgb, 255:red, 126; green, 211; blue, 33 }  ,draw opacity=1 ][line width=1.5]  (97.02,199.13) .. controls (97.02,150.68) and (143.32,111.41) .. (200.45,111.41) .. controls (257.57,111.41) and (303.88,150.68) .. (303.88,199.12) ;  
			\draw  [color={rgb, 255:red, 80; green, 227; blue, 194 }  ,draw opacity=1 ][line width=1.5]  (246.02,200.13) .. controls (246.02,151.68) and (292.32,112.41) .. (349.45,112.41) .. controls (406.57,112.41) and (452.88,151.68) .. (452.88,200.12) ;  
			\draw  [color={rgb, 255:red, 144; green, 19; blue, 254 }  ,draw opacity=1 ][line width=1.5]  (72.45,199.06) .. controls (72.45,199.06) and (72.45,199.06) .. (72.45,199.06) .. controls (72.45,141.66) and (163.33,95.13) .. (275.45,95.13) .. controls (387.56,95.13) and (478.45,141.66) .. (478.45,199.06) ;  
			\draw  [color={rgb, 255:red, 208; green, 2; blue, 27 }  ,draw opacity=1 ][line width=1.5]  (284.45,161.88) .. controls (284.45,145.72) and (314,132.63) .. (350.45,132.63) .. controls (386.9,132.63) and (416.45,145.72) .. (416.45,161.88) .. controls (416.45,178.03) and (386.9,191.13) .. (350.45,191.13) .. controls (314,191.13) and (284.45,178.03) .. (284.45,161.88) -- cycle ;
			\draw  [color={rgb, 255:red, 208; green, 2; blue, 27 }  ,draw opacity=1 ][fill={rgb, 255:red, 208; green, 2; blue, 27 }  ,fill opacity=1 ] (131.32,156.56) .. controls (131.32,154.59) and (132.92,153) .. (134.89,153) .. controls (136.85,153) and (138.45,154.59) .. (138.45,156.56) .. controls (138.45,158.53) and (136.85,160.13) .. (134.89,160.13) .. controls (132.92,160.13) and (131.32,158.53) .. (131.32,156.56) -- cycle ;
			\draw    (133.45,133.13) -- (168.45,228.13) ;
			\draw    (439.45,159.13) -- (415,230) ;
			\draw    (457,152) -- (538.45,123.13) ;
			
			\draw (45,30) node [anchor=north west][inner sep=0.75pt]    {$Y_{0}$};
			\draw (100,70) node [anchor=north west][inner sep=0.75pt]    {$W_{0}$};
			\draw (265,70) node [anchor=north west][inner sep=0.75pt]    {$W_{1}$};
			\draw (420,70) node [anchor=north west][inner sep=0.75pt]    {$W_{2}$};
			\draw (195,30) node [anchor=north west][inner sep=0.75pt]  [color={rgb, 255:red, 245; green, 166; blue, 35 }  ,opacity=1 ]  {$Y_{2}$};
			\draw (327.5,30) node [anchor=north west][inner sep=0.75pt]  [color={rgb, 255:red, 139; green, 87; blue, 42 }  ,opacity=1 ]  {$( Y,K)$};
			\draw (495,30) node [anchor=north west][inner sep=0.75pt]    {$Y_{0}$};
			\draw (163,234) node [anchor=north west][inner sep=0.75pt]  [color={rgb, 255:red, 126; green, 211; blue, 33 }  ,opacity=1 ]  {$\partial Z_{1}$};
			\draw (400,234) node [anchor=north west][inner sep=0.75pt]  [color={rgb, 255:red, 80; green, 227; blue, 194 }  ,opacity=1 ]  {$\partial Z_{2}$};
			\draw (542,115.4) node [anchor=north west][inner sep=0.75pt]  [color={rgb, 255:red, 144; green, 19; blue, 254 }  ,opacity=1 ]  {$\partial Z$};
		\end{tikzpicture}
		\caption{Schematic picture of the hypersurfaces and the 4-manifolds they bound. The point and arc in red indicate singular loci.}
		\label{fig:hypersurfaces}
	\end{figure}
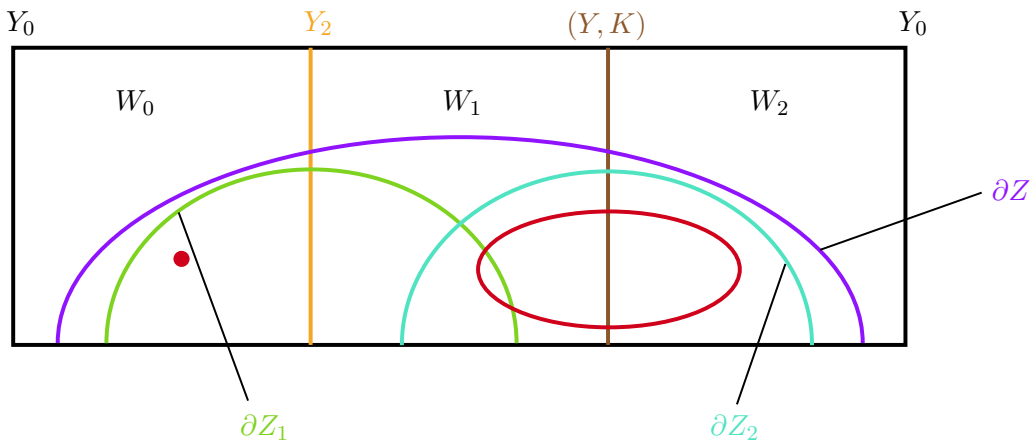
	
	Note that a (orbifold) tubular neighbourhood of the sphere formed by the cocore of $W_0$ and the core of $W_1$ has boundary $S^3$. This is seen by noticing that the boundary can be realised by a Heegaard diagram of genus one with the $\alpha$ and $\beta$ curves given by the meridians along with the solid tori are attached for the $0/1$ and $1/0$ surgeries. 
	
	We further claim that the component of the complement of this sphere that contains the cocore of $W_0$ and the core of $W_1$ is in fact the Thom space of a bundle over $S^2$ with euler class $-2$ with a small ball inside the disk bundle removed (this leads to the boundary $S^3$ and the orbifold point); let $Z_1$ denote this space. Then, $Z_1$ has $H_2(Z_1; \Z) \cong \Z$ with intersection form $-2$. Both of these facts can be deduced from the long exact sequence for manifold and its boundary after capping off the $S^3$ and removing a conical neighbourhood of the orbifold point, coupled with the fact that there is a Kirby diagram description with a single two handle; the sign is determined by looking at the Heegaard diagram description of the boundary.
	
	We now focus on the analogous situation for the case of $W_1$ and $W_2$: We take the core of $W_2$ and the cocore of $W_1$ and look at its tubular neighbourhood. This neighbourhood has boundary $%
	-\RP^3$ again by noticing that the decomposition as above gives a genus one Heegaard diagram with $\alpha$ and $\beta$ curves given by the meridians for the $2/1$ and $0/1$ surgery. Let $Z_2$ denote this neighbourhood. Then, it is clear that $Z_2$ is the disk bundle of a bundle over $S^2$ with euler class $%
	-2$ (as before this is by looking at the long exact sequence and noting that we only have one two handle in the Kirby diagram; the sign is by looking at the Heegaard diagram for the boundary). Our connections on $Z_2$ however, are singular along the zero section in this identification. 
	
	Finally, we look at the neighbourhood of the union of cores of $W_1$ and $W_2$ and cocores of $W_0$ and $W_1$. The boundary has a Heegaard diagram description of genus 1 with $\alpha$ and $\beta$ curves being the meridians for the $0/1$ and $0/1$ surgery respectively and so it is $S^2 \times S^1$; call the neighbourhood $Z$. Then, $Z$ is the complement of the neighbourhood of an unknotted circle in the Thom space of a bundle over $S^2$ with euler class $-2$ and disjoint from the cone point.
	
	See \Cref{fig:hypersurfaces} for a schematic representation of the situation where the red ellipse and dot are singular loci.
	
	\begin{rmk}
		By a Thom space of a (real) vector bundle $E$, $Th(E)$, we mean the one point compactification of $E$. Note that if the base is $S^2$ and $E$ is a real rank 2 oriented vector bundle, then $Th(E)$ is an oriented 4-manifold with a single cone point singularity, the point at infinity. Further, this cone point singularity is also an orbifold singularity.
	\end{rmk}

	\subsection{Cobordisms and geometric cycles for the exact triangle}
	\label{sec:geo-cycle}
	
	In this section we describe the relevant cobordisms with the geometric cycles that describe the bundles over them necessary for the surgery exact triangle. First, fix a geometric cycle $\gamma_P$ on $\yk$ such that $\gamma_P$ is a closed $1$-submanifold in $\yk$; this implies that the (singular) bundle data will always be trivial near $K$.

	We define $\Omega_{P_i} = \gamma_P \times [0, 1]$, i.e., we take the bundle to be just the product away from the surgery region and trivial on the surgery region.
	
	Finally we record that we can endow the cobordisms with almost complex structures such that the singular loci are almost complex submanifolds. We introduce some notation for this purpose. Let $p$ denote the orbifold point in $W_0$, $\Sigma_1 \subset W_1$ the co-core of the 2-handle and $\Sigma_2 \subset W_2$ the core of the attached 2-handle. Further, let $\Sigma$ denote the sphere formed by the co-core of 2-handle in $W_1$ and core of 2-handle in $W_2$, i.e., $\Sigma = \Sigma_1 \cup \Sigma_2$.
	
	\begin{lemma}\label{lem:cgeo-orientation}
		There is an almost complex structure $J$ on $W_0 \cup W_1 \cup W_2$ such that it restricts to isomorphic structures on the collar neighbourhoods of the boundary components. Further, we can arrange that the $\Sigma$ is an almost complex submanifold with respect to $J$.
	\end{lemma}

	\begin{proof}
		We try to find an almost complex structure $J'$ on $W = W_0 \cup W_1 \cup W_2 / {\sim}$ where $\sim$ identifies the two boundary components with a orientation reversing diffeomorphism. That one can find a $J'$ can be checked by a characteristic class computation (note that $W$ is closed!) but this gives us no control over $J'$ near $\Sigma$. Instead, we cut along one of the necks described in \Cref{sec:heg} and equip each piece with an almost complex structure and ensure the structure on the collars of the boundary agree.
		
		To do this, notice that $\partial (W \setminus Z) = S^1 \times S^2$, $(W \setminus Z) \cup (S^1 \times D^3) = Y_0 \times S^1$ and that $(S^1 \times D^3) \cup Z = Th(T^*S^2)$. Further, $S^1 \times \{0\} \subset Th(T^*S^2)$ links with the zero section; call the image of $S^1$ in $Th(T^*S^2)$ as $\gamma_Z$ so that $Z = Th(T^*S^2) \setminus \nu(\gamma_Z)$ where the latter set is the tubular neighbourhood. Similarly, note that $S^1 \times \{0\} \subset Y_0 \times S^1$ has image (isotopic to) $K_0 \times \{*\}$ where $K_0$ denotes the dual knot. We will construct almost complex structures on $Th(T^*S^2)$ and $Y_0 \times S^1$ with the desired properties such that they agree on tubular neighbourhoods of $\gamma_Z$ and $K_0 \times \{*\}$.
		
		First, note that $\gamma_Z$ can be considered to be the standard circle in a fibre of the associated vector bundle $T^*S^2$. The standard complex structure on $Th(T^*S^2)$ is $S^1$ invariant, where $S^1$ acts on $Th(T^*S^2)$ by complex multiplication on the fibres fixing the base and point at infinity. Under this $S^1$ action, $\gamma_Z$ is the orbit of any point on which the action is free. Thus, we can describe the (almost) complex structure on $\nu (\gamma_Z)$ as $A \times D^2$ where $A$ denotes an annulus. 
		
		Finally, we need to find an almost complex structure on $Y_0 \times S^1$ such that restricted to $\nu(K_0 \times \{*\})$, such that it agrees with $\nu(\gamma_Z)$. Here's one way to construct it: choose a (co-oriented) contact structure $\xi$ on $Y_0$. Then, we can define the almost complex structure $\tilde{J}$ on $Y_0 \times S^1$ as taking $\partial_\theta$ (where $\theta$ is the coordinate on the $S^1$ factor) to the complement of $\xi$ and the one defined by the co-orientation on $\xi$. Choosing a Legendrian representative of $K_0$, we see that $\tilde{J}$ restricts to the tubular neighbourhood of $K_0 \times \{*\}$ to give an almost complex structure isomorphic to $\nu(\gamma_Z)$.
	\end{proof}
	
	\begin{rmk}
		We fix the choice of $J$ from above for the rest of this paper. Also, note that this choice ensures that the map defined $(W_0 \cup W_1 \cup W_2 \setminus Z) \cup (S^1 \times D^3) = Y_0 \times [0, 1]$ is the identity. See \Cref{sec:complex orientations} and \cite[\S 5.1]{kmknot} for details.
	\end{rmk}
	
	\subsection{Family of Metrics for Neck Stretching} We describe the parameter space of metrics in our context for the case of the triple composite from $Y_0$ to itself, i.e., the situation as presented in \Cref{fig:hypersurfaces}. To set notation, call the singular locus in $Z_1$ as $\Sigma_{Z_1}$, in $Z_2$ as $\Sigma_{Z_2}$ and denote the singular point by $p$. Note that $\Sigma_{Z_1} \cong D^2$ and $\Sigma_{Z_2} \cong S^2$.
	
	\begin{figure}[!h]
		\tikzset{every picture/.style={line width=0.75pt}} %
		
		\begin{tikzpicture}[x=0.75pt,y=0.75pt,yscale=-1,xscale=1]
			\clip (100,25) rectangle (400, 275);
			\draw   (367.31,120.81) -- (324.24,250.29) -- (187.79,249.34) -- (146.53,119.27) -- (257.48,39.84) -- cycle ;
			\draw    (202.01,79.56) -- (256.67,155.91) ;
			\draw    (312.4,80.33) -- (256.67,155.91) ;
			\draw    (345.78,185.55) -- (256.67,155.91) ;
			\draw    (256.67,155.91) -- (256.02,249.81) ;
			\draw    (256.67,155.91) -- (167.16,184.3) ;
			
			\draw (241,255) node [anchor=north west][inner sep=0.75pt]    {$Q( Z)$};
			\draw (115,180.4) node [anchor=north west][inner sep=0.75pt]    {$Q( Z_{1})$};
			\draw (352.78,179.95) node [anchor=north west][inner sep=0.75pt]    {$Q( Z_{2})$};
			\draw (167,53.4) node [anchor=north west][inner sep=0.75pt]    {$Q( Y)$};
			\draw (320,53.4) node [anchor=north west][inner sep=0.75pt]    {$Q( Y_{2})$};
		\end{tikzpicture}
		\caption{Moduli of metrics used in \Cref{sec:triple-composite}. The edges of the pentagon are used in \Cref{homotopies} }
		\label{fig:metrics}
	\end{figure}
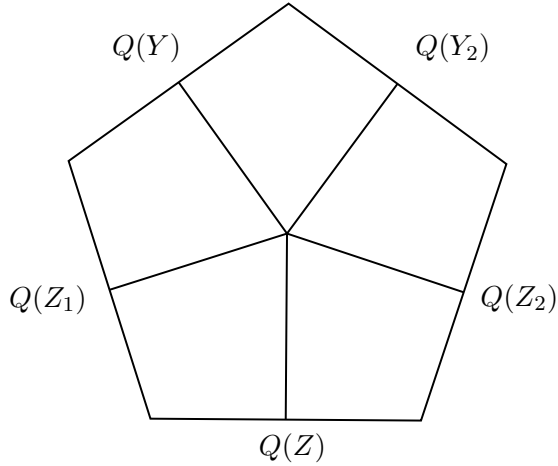
	
	Note that out of $\{ Y_0, Y_2, Y, \partial Z_1, \partial Z_2, \partial Z \}$ at most two hyper-surfaces are non-intersecting. For any two non-intersecting hyper-surfaces, we can take their union and define a 2-parameter family of metrics as described in \Cref{sec:metrics-compactness-gluing}. These parameter spaces are homeomorphic to quadrilaterals. Taking all such pairs of non-intersecting hyper-surfaces, we have five quadrilaterals in total which we can put together to form the pentagon as in \Cref{fig:metrics}. Note for instance that the quadrilateral sharing edges labelled $Q(Y)$ and $Q(Y_2)$ corresponds to the space of metrics defined by $Y \cup Y_2$. The edges of these metrics are used in \Cref{homotopies} and the full parameter space is used in \Cref{sec:triple-composite}.
	
	\section{A Moduli Space Computation} \label{sec:moduli-computation}
	
		Let $Z_1 = \mathbb{D}(T^*S^2)$ and $\Sigma_1 \subset Z_1$ denote the zero section. The main result of this section is the following proposition.
	
	\begin{proposition}\label{prop:key-moduli computation}
		$M^{red}(Z_1, \Sigma_1, \pp_1)$ where $\pp_1$ is the trivial bundle consists of exactly two isolated connections with energy $\kappa = 1/8$ and these are oriented in the same way by using the orientation conventions of \Cref{sec:complex orientations}. These connections are distinguished by the limiting flat connection on the boundary. Further, when the connection on the boundary is trivial, we have $l = 0$ (and $l = -1$ otherwise) where $l$ denotes the monopole charge.
	\end{proposition}
	
	We will first convert the problem to one about moduli on closed pairs of manifolds and then prove it. We state and prove the required moduli computation for the closed pair first.
	
	Let $Z = Z_1 \cup_{\RP^3} \mathbb{D}(T^*S^2) = Z_1 \cup Z_2$, $\Sigma = \Sigma_1$ and $\Sigma_2$ the zero section of the new disk bundle glued to $Z_1$. Further, suppose $\pp$ is the trivial singular bundle data over $(Z, \Sigma)$ and $\pp'$ is bundle data with $w_2(\pp') = P.D.[\Sigma_2]$. Note that $Z \cong \overline{\CP^2} \# \overline{\CP^2}$. Following \cite[\S 4.4.2]{kmbar}, denote the exceptional spheres in this presentation by $E_1$ and $E_2$. Up to some orientation conventions, 
	\begin{equation} \label{eqn:cohom}
		\begin{aligned}
			[\Sigma_1] &= - [E_1] - [E_2] \\
			[\Sigma_2] &= - [E_1] + [E_2] 
		\end{aligned}
	\end{equation}
	
	\begin{lemma} \label{lem:model-computation}
		$M^{red}(Z, \Sigma, \pp)$ and $M^{red}(Z, \Sigma, \pp')$ each consist of a single isolated connection with energy $\kappa = 1/8$. Further, both solutions have are isomorphic up to $\so$ gauge equivalence. These unique solutions carry monopole charge of $l = 0$ (for $\pp$) and $l = -1$ (for $\pp'$).
	\end{lemma}

	\begin{proof}
		We first deal with the case of bundle $\pp$. Since we are dealing with globally reducible solutions on a negative definite simply connected manifold, we just need to determine the cohomology class of $F_A$ to determine the equivalence class of $A$ in the moduli of reducible ASD solutions. Let $e_i$ denote the Poincare dual of $E_i$. Note that 
		\begin{equation} \label{eqn:intersection}
			\int_Z e_i \wedge e_j =
				\begin{cases}
					-1 \quad \text{if } i = j \\
					\phantom{-} 0 \quad \text{otherwise}.
				\end{cases}
		\end{equation}
		
		Suppose the bundle $E$ over $Z \setminus \Sigma$ splits as $L \oplus L^{\vee}$ and the connection splits as $B \oplus B^\vee$. We will analyse the cohomology class of $F_B$ as it determines $F_A$.
		
		Now, due to singularity of $B$ near $\Sigma = \Sigma_1$, we must have 
		\begin{equation} \label{eqn:chern}
		\frac{i}{2\pi} [F_B] = \frac{1}{4} P.D.[\Sigma_1] + a_1 e_1 + a_2 e_2
		\end{equation}
		for some $a_i \in \Z$. By \eqref{eqn:cohom} this is equivalent to 
		\[
		\frac{i}{2 \pi} [F_B] = \left (a_1 - \frac{1}{4} \right ) e_1 + \left (a_2 - \frac{1}{4} \right ) e_2.
		\]
		
		The constraint $\kappa = 1/8$ gives the following equation:
		\begin{align*}
			\frac{1}{8} &= \kappa (A) = \frac{1}{8 \pi^2} \int_Z \tr (F_A \wedge F_A) \\
			&= -\frac{1}{2} \int_Z \tr( \Id ) \bigg[  \left (a_1 - \frac{1}{4} \right ) e_1 + \left (a_2 - \frac{1}{4} \right ) e_2 \bigg] \wedge \bigg[  \left (a_1 - \frac{1}{4} \right ) e_1 + \left (a_2 - \frac{1}{4} \right ) e_2 \bigg] \\
			&= \left( a_1 - \frac{1}{4} \right)^2 + \left( a_2 - \frac{1}{4} \right)^2.
		\end{align*}
		The last line follows by \eqref{eqn:intersection}. As $a_i \in \Z$, this forces $a_i = 0$ to be the only solution. We note that $(k, l) = (0,0)$ for this solution.
		
		For the case of $\pp'$, we can use the same strategy as above, \eqref{eqn:chern}, becomes the following:
		
		\begin{equation} \label{eqn:chern1}
		\frac{i}{2\pi} [F_{B'}] = \frac{1}{4} P.D.[\Sigma_1] + a'_1 e_1 + a'_2 e_2 \cpm \frac{1}{2} P.D. [\Sigma_2].
		\end{equation}
		Note that now $B'$ is defined not on $Z \setminus \Sigma$ but on $Z \setminus (\Sigma \cup \Sigma_2)$ such that the linking holonomy along $\Sigma$ is traceless and along $\Sigma_2$ is $-1$. The ambiguity in sign in the last term is due to the fact that any orientation of the linking circle along $\Sigma_2$ gives the same equivalence class in the moduli space.
		
		The equation $\kappa = 1/8$ now becomes
		\[
		\frac{1}{8} = \left( a'_1 - \frac{1}{4} \cmp \frac{1}{2} \right)^2 + \left( a'_2 - \frac{1}{4} \cpm \frac{1}{2} \right)^2.
		\]
		
		The only solution is $(a'_1, a'_2) = (1, 0)$ (if the $+$ sign is used in \eqref{eqn:chern1}) or $(a'_1, a'_2) = (0,1)$ (if the $-$ sign is used in \eqref{eqn:chern1}). This is expected as the signs are due to the ambiguity in the choice of an integer lift of $P.D.[\Sigma_2]$ viewed as an element of $H^2(Z; \BF_2)$; the above signs are two possible lifts among an infinite family. In any case, we can make sense of $k$ and $l$ for this unique solution and see that $(k, l) = (1/2, -1)$. The reason $k$ is non-integral while $l$ is due to the fact that 
		\[
		2l \equiv \< w_2 (\pp'), \Sigma > \equiv 0 \pmod 2
		\]
		while 
		\[
		4k \equiv w_2(\pp')^2 \equiv 2 \pmod 4.
		\]
		
		Further, note that 
		\[
		\frac{i}{2 \pi } [F_{B \otimes B}] = \frac{1}{2} P.D. [\Sigma] = \frac{i}{2 \pi} [F_{B' \otimes B'}]
		\]
		at the level of induced connections on the complex line sub-bundle of the adjoint bundles. Now $A_{\ad} = d_\epsilon \oplus [B \otimes B]$ and $A'_{\ad} = d_\epsilon \oplus [B' \otimes B']$ are the connections on bundles $\ad (E) = \underline{\R} \oplus F$ and $\ad (E') = \underline{\R} \oplus F'$ over $Z \setminus \Sigma$ where $d_\epsilon$ denotes the trivial connection. Clearly these are $\so$ gauge equivalent.
		
		Finally, as the adjoint bundles with induced connections are $\so$ gauge equivalent, the solutions are $\so$ gauge equivalent.
	\end{proof}

	We now tie back this lemma to \Cref{prop:key-moduli computation}.
	
	\begin{lemma} \label{lem:closed-pair}
		\Cref{prop:key-moduli computation} is implied by \Cref{lem:model-computation}.
	\end{lemma}
	
	\begin{proof}
		We first note that the moduli of reducible connections is unaffected (up to homeomorphism) by changes in metric for $(Z, \Sigma)$. More precisely, for the reducible case, the equation for $(Z, \Sigma)$ is just equivalent to the curvature form (a $i\R$ valued 2-form) being harmonic as $Z$ is negative definite. That the moduli of these don't depend on the metric now follows from abelian Hodge theory.
		
		With that remark out of the way, we will choose a metric such that the the boundary $\RP^3$ we glue along to form $Z$ is a long neck. The advantage is that since $\kappa = 1/8$ is small enough, the moduli is a fibre-product of the moduli on the two four-manifold pieces: Any non-flat ASD connection on $\R \times \RP^3$ has energy $\kappa \geq 1/2$ by a standard Chern-Simons invariant computation. Further, any non-flat ASD connection on $Z_2$ will have energy $\kappa \geq 1/4$. The description of $\pp$ and $\pp'$ show that in each case $Z_2$ admits a unique flat ASD connection on it with the boundary $SU(2)$ representation being trivial (in the case of $\pp$) or non-trivial (for $\pp'$). Further, in the fibre-product description there is no gluing parameter involved as the stabiliser of the connection on $Z_2$ is equal to the stabiliser of the representation on $\RP^3$ in both cases. 
		
		We need to comment on orientations as, while $Z$ does not admit a complex structure, $Z_1$ does such that $\Sigma_1$ is a complex submanifold. As we use the orientation conventions as in \Cref{sec:complex orientations}, we only need to check if the induced connections on the adjoint bundle are $\so$ gauge equivalent. This is clear from the above neck-stretching procedure and \Cref{lem:model-computation}.
	\end{proof}

	\begin{rmk} \label{rmk:xi}
		We make a note of the fact that the two solutions are related by tensoring by the unique non-trivial flat real line bundle on $Z_1 \setminus \Sigma_1$. Denoting the bundle (and the associated flat connection) by $\xi$, this is same as saying the connections $A$ and $A'$ in the proof of \Cref{lem:model-computation} are related by $A' = A \otimes \xi$ (equivalently $A = A' \otimes \xi$).
	\end{rmk}
	
	\section{Proof of the Triangle}\label{sec:tech}
	
	We first state the main homological algebra result we need to prove our main result and then we set up the relevant steps in our context.
	
	\subsection{The Lin Triangle Detection Lemma}
	
	We state and prove a triangle detection lemma which is more general than \cite[Lemma 4.2]{oz-sz}. This result, with $\BF_2$ coefficients, is due to \cite{lin} and we reproduce a proof here that works over $\Z$.
	
	\begin{proposition}[{\cite[Proposition 4]{lin}}] \label{prop:triangle}
		Let $(C_i, \pd)$ be chain complexes, $i \in \Z / 3$. With the maps $f_i : C_i \to C_{i+1}$, $g_i : C_i \to C_{i+1}$, $H_i : C_i \to C_{i+2}$, $F_i : C_i \to C_i$ and, $G_i : C_i \to C_i$ satisfying the following equations for all $i \in \Z / 3$:
		\begin{align*}
			\pd[i+1] f_{i} - f_i \pd &= 0 \\
			\pd[i+2] H_i + H_i \pd &= f_{i+1} f_i - g_i \\
			\pd G_i - G_i \pd &= F_i - (f_{i+2} H_i - H_{i+1} f_i)
		\end{align*}
		
		Additionally, assume that $g_1 = 0$ and that $F_i$ are quasi-isomorphisms. Then, $C = \oplus_{i = 0}^2 C_i$ is a chain complex with differential $\partial : C \to C$ defined by 
		
		\[
		\partial = 
		\begin{pmatrix}
			\pd[0] & f_2 & - H_1 \\
			0 & - \pd[2] & f_1 \\
			0 & 0 & \pd[1]
		\end{pmatrix}	
		\]
		
		and is acyclic.
	\end{proposition}
	
	Assume the proposition for now. Note that the mapping cone of $f_2$, $C(f_2)$, is a subcomplex of $C$ with quotient complex $C_1$, hence, $C_1$ is quasi-isomorphic to $C(f_2)$. The map $\delta : C_1 \to C(f_2)$ realising this is explicitly given by $x \mapsto (-H_1(x), f_1(x))$.
	
	The mapping cone fits in a distinguished triangle as follows:
	\begin{equation*}
		\begin{tikzcd}[column sep=small]
			C_0 \arrow[rr, "i"] & & C(f_2) \arrow[dl, "\pi"] \\
			& C_2 \arrow[ul, "f_2"]&
		\end{tikzcd}
	\end{equation*}
	
	Replacing $C(f_2)$ by $C_1$, we have a distinguished triangle in the derived category of chain complexes between all the $C_i$. We record this as a corollary below.
	
	\begin{corollary} \label{cor:triangle}
		In the set up of \Cref{prop:triangle}, we have the following distinguished triangle in the derived category of chain complexes
		\begin{equation*}
			\begin{tikzcd}[column sep=small]
				C_0 \arrow[rr, "\frac{1}{\delta} \cdotp i"] & & C_1 \arrow[dl, "f_1"] \\
				& C_2 \arrow[ul, "f_2"]&
			\end{tikzcd}
		\end{equation*}
		\qed
	\end{corollary}
	
	The proof of the main theorem will be constructing appropriate chain complexes and maps so that \Cref{prop:triangle} applies whence, we conclude by \Cref{cor:triangle}.
	
	\begin{rmk} \label{rmk:vchain-maps}
		While the map $C_0 \to C_1$ is not a morphism of chain complexes but rather a `virtual' chain map defined by a zigzag, the explicit description above is potentially useful. In fact, if the maps $g_2$ and  $g_0$ are zero like in the context of \cite[Lemma 4.2]{oz-sz}, then we recover their result: the point is that $\delta^{-1}$ admits a lift to a morphism of chain complexes and one explicitly checks that the horizontal map is $f_1$ up to compositions by quasi-isomorphisms. We will not give a proof of this assertion and instead just resort to the result in \cite{oz-sz} as the only difference between the hypothesis of \Cref{prop:triangle} and \cite[Lemma 4.2]{oz-sz} is requiring $g_2 = g_0 = 0$.
	\end{rmk}
	
	\begin{proof}[Proof of \Cref{prop:triangle}]
		We first observe that $g_1 = 0$ implies $\pd[3] H_1 + H_1 \pd[1] = f_{2} f_1$. Hence $\pd[]^2 = 0$. Now that $(C, \pd[])$ is a complex, we proceed to show it is acyclic.
		
		Following the proof in \cite[\S 3]{lin}, we define a map $G : C \to C$ as follows:
		
		\[ G = 
		\begin{pmatrix}
			G_0 & 0 &0 \\
			H_0 & -G_2 & 0 \\
			f_0 & H_2 & G_1
		\end{pmatrix}
		\]
		
		Now, define $\varphi$ by the following:
		\begin{align*}
			\varphi &= \pd[] G - G \pd[] \\
			&= 
			\begin{pmatrix}
				F_0 & - f_2 G_2 - H_1 H_2 - G_0 f_2 & -H_1 G_1 + G_0 H_1 \\
				g_0 & F_2 & f_1 G_1 + H_0 H_1 + G_2 f_1 \\
				0 & -g_2 & F_1
			\end{pmatrix}		
		\end{align*} 
		While $\varphi$ is not a chain map, it is an \emph{anti-chain} map. Hence, we can define the mapping cone $M_\varphi$ with complex $\hat{C} \oplus C$ (we use a hat to distinguish the two copies of $C$) and differential 
		
		\begin{align*}
			\pd[\varphi] &= 
			\begin{pmatrix}
				\hat{\pd[]} & \varphi \\ 
				0 & \pd[]
			\end{pmatrix} \\
			& = 
			\begin{pmatrix}
				\hat{\pd[0]} & \hat{f_2} & -\hat{H_1} & F_0 & * & * \\
				& - \hat{\pd[2]} & \hat{f_1} & g_0 & F_2 & *\\
				& & \hat{\pd[1]} & 0 & -g_2 & F_1 \\
				& & & \pd[0] & f_2 & -H_1 \\
				& & & & - \pd[2] & f_1 \\
				& & & &  & \pd[1]
			\end{pmatrix}
		\end{align*}
		where we drop the 0 entries below the diagonal in the last expression for readability. We show $(M_\varphi, \pd[\varphi])$ is acyclic.
		
		As the differential above is upper triangular, we can use the filtration provided by this decomposition to compute its homology via a spectral sequence. Since there are no maps between $\hat{C}$ and $C$ in the first off diagonal entries of $\pd[\varphi]$, there are no differentials between their sub-quotients in the $E^1$ page of this spectral sequence. The $E^2$ page is as follows:
		
		\begin{center}
			\begin{tikzpicture}
				\matrix (m) [matrix of math nodes,row sep=1.5em,column sep=2.5em,minimum width=2em]
				{
					\mathrm{coker}\hat{f}_2 & \mathrm{ker}\hat{f}_2/\mathrm{Im}\hat{f}_1 & \mathrm{ker}\hat{f}_1\\
					\mathrm{coker}f_2 & \mathrm{ker}f_2/\mathrm{Im}f_1 & \mathrm{ker}f_1\\};
				
				\path[-stealth]
				(m-2-1) edge node [above]{$g_0$} (m-1-2)
				(m-2-2) edge node [above]{$-g_2$} (m-1-3)
				;
				
				\draw(m-1-3) edge[out=160, in=20, ->] node [above]{$-\hat{H}_1$}  (m-1-1);
				\draw(m-2-3) edge[out=200, in=340, ->] node [below]{$-H_1$}  (m-2-1);
			\end{tikzpicture}
		\end{center}
		where by abuse of notation we denote the maps induced on the appropriate groups by the maps indicated above. Now, as $g_0$ is chain-homotopic to $f_1 f_0$ (by $H_0$), its image is contained in the image of $\hat{f_1}$. Hence, the map induced by $g_0$ above is $0$. Similarly, $g_2$ is chain homotopic to $f_0 f_2$, by $H_2$, and hence vanishes on any quotient of $\mathrm{ker} f_2$. Hence, the sub-quotients of $H_*(\hat{C})$ and $H_*(C)$ don't interact at this page either. The $E^3$ page is:
		\begin{center}
			\begin{tikzpicture}
				\matrix (m) [matrix of math nodes,row sep=2em,column sep=2em,minimum width=2em]
				{\mathrm{coker}\hat{f}_2/\mathrm{Im}\hat{H}_1 & \mathrm{ker}\hat{f}_2/\mathrm{Im}\hat{f}_1
					& \mathrm{ker}\hat{f}_1\cap \mathrm{ker}\hat{H}_1\\
					\mathrm{coker}{f}_2/\mathrm{Im}{H}_1 & \mathrm{ker}{f}_2/\mathrm{Im}{f}_1
					& \mathrm{ker}{f}_1\cap \mathrm{ker}{H}_1\\};
				
				\path[-stealth]
				(m-2-1) edge node [left]{$ F_0$}(m-1-1)
				(m-2-2) edge node [left]{$ F_2$}(m-1-2)
				(m-2-3) edge node [left]{$ F_1$}(m-1-3);
				;
			\end{tikzpicture}
		\end{center}
		
		We now use that $F_i$ are quasi-isomorphisms: as all the groups on this page are sub-quotients of $H_*(C)$, the vertical maps are all isomorphisms. Hence, spectral sequence collapses on $E^4$ with all groups being trivial. 
		
		This shows that $(M_\varphi, \pd[\varphi])$ is acyclic, i.e., $\varphi$ induces an isomorphism on $H_*(C)$. However, by construction, we see that $\varphi$ maps $\mathrm{ker} \, \pd[]$ into $\mathrm{Im} \, \pd[]$, i.e., $\varphi$ induces the $0$ map on $H_*(C)$. As the $0$ map is only an isomorphism on the trivial group, this shows that $(C, \pd[])$ is acyclic.
	\end{proof}
	
	\subsection{The chain complexes and chain maps} \label{sec:complexes-maps}
	
	We describe the chain complexes $(C_i, \pd )$ and the chain maps $f_i : C_i \to C_{i+1}$ in this subsection.
	
	The complexes for $i = 0, 1, 2$ are the usual instanton chain complexes for $(Y_0, \pp_0)$, $(Y_2, \pp_2)$ and $(Y, K, \pp)$ respectively, with $\Z[T, T^{-1}]$ coefficients; the latter inherits this from the local coefficient system while the first two are modules over $\Z$ by default and we tensor by $\Z[T, T^{-1}]$.
	
	The maps $f_i$ can be described by the cobordisms $W_i$:
	
	\begin{align*}
		f_0 &= CI(W_0, p, \pp_0; \mathfrak{o}_+) - CI(W_0, p, \pp_0; \mathfrak{o}_-) \\
		f_1 &= CI(W_1, \Sigma_1, \pp_1) \\
		f_2 &= CI(W_2, \Sigma_2, \pp_2).
	\end{align*}
	
	For future reference, we shall use the following notation for the two components of $f_0$:
	\[
	f_0^\pm = CI(W_0, p, \pp_0; \mathfrak{o}_\pm).
	\]
	
	\subsection{The $H_i$ maps}
	\label{homotopies}
	
	We define the maps $H_i$ and $g_i$ in this subsection. 
	
	Recall that $\partial Z_1$ and $Y_2$ are two hyper-surfaces for which we can define a moduli space of cut metrics on $W_0 \cup W_1$; call this space of metrics $Q_t$, represented by one of the edges in \Cref{fig:metrics}. This will be parametrised over an interval and we can define maps between the chain groups of the boundary manifold after fixing a bundle (i..e, geometric cycle) over the 4-manifold and choosing a flat connection $\mathfrak{o} \in \{ \fo_+, \fo_-\}$. Denote the associated maps by $h_0^\pm$. Let $X = W_0 \cup W_1 \setminus Z_1$ and the restriction of the bundle on $W_0 \cup W_1$ defined by the geometric representatives be denoted by $\pp_X$ and the restriction to $Z_1$ be $\pp_{Z_1}$. Note that this immediately implies the following:
	
	\[
	\partial_2 h_0^\pm + h_0^\pm \partial_0 = f_1 f_0^\pm - m^{\pm}_{\partial Z_1}
	\]
	where $m^{\pm}_{\partial Z_1}$ refers to the map defined by the following formula: 
	\[
	\< m^{\pm}_{\partial Z_1} (\mathfrak{a}), \mathfrak{b}> = \# M(X, C, \pp_X; \mathfrak{a}, \mathfrak{b}) \times_{r_{\partial Z_1}} M(Z_1, \Sigma_{Z_1} \cup \{ p \}, \pp_{Z_1}; \fo_\pm)  
	\]
	whenever the expected dimension of the fibre-product is zero and 0 otherwise. Here $C$ is an annulus, $\Sigma_{Z_1}$ a disk and $p$ the cone point in $W_0$; note that $C \cup \Sigma_{Z_1}$ is the co-core disk of $W_1$.
	
	Note that $M(Z_1, \Sigma_{Z_1} \cup \{ p \}, \pp_{Z_1})$ could \textit{a priori} contain reducible connections in which case a gluing parameter is needed to be taken into account in the fibre product; we shall continue to use notation that does not indicate this even though this is how we shall interpret the `glued' moduli spaces. Here, $r_{\partial Z_1}$ is the map that sends an instanton to its restriction on $(\partial Z_1, U)$ where $\pd[] Z_1 \cong S^3$ and $U$ is the unknot. 
	
	Implicitly in the above, we have assumed the chosen perturbation is supported away from $\partial Z_1$. We can do this as the flat connections on $\partial Z_1$ are non-degenerate with the standard metric.
	
	\begin{lemma} \label{prenil}
		There is a $\varepsilon \in \{ \pm 1 \}$ which only depends on $(Y, K)$ such that 
		\begin{align*}
			\< m^{+}_{\partial Z_1} (\mathfrak{a}), \mathfrak{b}> &= \varepsilon \# M(X, C, \pp_X; \mathfrak{a}, \mathfrak{b}) \\
			\< m^{-}_{\partial Z_1} (\mathfrak{a}), \mathfrak{b}> &= \varepsilon T^{-1} \# M(X, C, \pp_X; \mathfrak{a}, \mathfrak{b}).
		\end{align*}
	\end{lemma}
	
	\begin{proof}
		Note that in the above claim, we count only finite energy solutions; in particular, since the boundary $(S^3, U)$ has only one solution, we ask the connection to be asymptotic to this unique singular flat connection on $(S^3, U)$.
		
		Let $a, b, c$ be connections corresponding to the gauge equivalence class of critical points on $Y_0, (Y, K), (\partial Z_1, U)$ respectively. Note that as $(\partial Z_1, U) = (S^3, U)$, $[c]$ is a reducible critical point while $[a]$ and $[b]$ are irreducible, due to the non-integrality conditions on $Y_0$ and $(Y, K)$.
		
		We want to show there there is exactly one $[B]$ on $Z_1$ for any $[A]$ on $X$ such that the pair of connections $[A]$ and $[B]$, which restrict on the boundary components to $[a], [b]$ and $[c]$, restricts to a fixed $\fo \in \{\fo_+, \fo_-\}$ and have total index $0$; for this is the expected dimension of the fibre-product. By total index, we mean, thanks to the gluing formula,
		\begin{equation} 
			0 =  \ind (A) + \ind (B) + h^0(c) + h^1(c)
		\end{equation}
		Note that we are using weighted spaces as defined in \Cref{sec:weighted-spaces}. 
		
		For the computation of $h^i(c)$, note that $h^1(c) \leq h^1(\tilde{c})$ where $\tilde{c}$ is the critical point on the branched double cover. Now, the branched double cover of $(S^3, U)$ is $S^3$ and so $h^1(\tilde{c}) = 0$ by the Weitzenb\"ock formulae. Further, $h^0(c) = 1$ for the connection is flat on $S^3 \setminus U$ with non-trivial monodromy, implying it is stabilised by $U(1)$. 
		
		This means we have 
		\begin{equation} \label{ind}
			0 = 1 + \ind (A) + \ind (B)
		\end{equation}
		
		Now, $\ind (A) + h^0(A) \geq 0$ and since the restriction of $[A]$ to the boundary components $Y_0$ is irreducible, $h^0(A) = 0$. Similarly, $\ind (B) + h^0(B) \geq 0$ and $h^0(B) \leq 1$; for the last inequality, note that the stabiliser commutes with the monodromy along the singularity. Combined with \eqref{ind}, we see that we must have $\ind (A) = \ind (B) + 1 = 0$. Hence we need to compute the signed count of reducible connections $[B]$ with $\ind (B) = -1$.
		
		We first cap off the boundary component $(S^3, U)$ of $Z_1$ with the flat connection on $(D^4, D^2)$, which has index $-1$ by an explicit computation. The gluing formula, with the computation of $h^i(c)$ tells us this doesn't affect the resulting index on the space $(Z_1 \cup D^4, \Sigma_{Z_1} \cup D^2 \cup \{ p \})$. Let $Z_1 ' = Z_1 \cup D^4$, $\Sigma ' = \Sigma_{Z_1} \cup D^2$. Call the resulting connection on $Z_1'$ as $B'$; this connection is also reducible as the stabiliser consists of those elements that commute with the monodromy near the singular locus $\Sigma$ and these also stabilise the flat connection on $(D^4, D^2)$. Then note that $\ind (B') = -1$ combined with the index formula for the closed case shows $\kappa (B') = 1/8$. Now, as $(Z_1' \setminus \{p\}, \Sigma') \cong (\mathbb{D}(T^*S^2), S^2)$, we see that due to \Cref{thm:rm-sing}, \Cref{prop:key-moduli computation} applies to show that there exactly two such solutions distinguished by the flat connection $\fo$, both oriented in the same way. Further, when $\fo = \fo_-$, $l = -1$ and $l = 0$ otherwise.
		
		Finally, note that the there is no gluing parameter here as $A$ is irreducible and $\Stab (B) \hookrightarrow \Stab (c)$ is an isomorphism.
	\end{proof}
	
	We define $H_0 = h^{+}_0 - h^{-}_0$ and $g_0 =  m^+_{\partial Z_1} - m^{-}_{\partial Z_1}$. Note that we have $f_1 f_0 - g_0 = \pd[2] H_0 + H_0 \pd[0]$.
	
	All of the above arguments can be repeated to define $H_2$ and $g_2$ satisfying $f_0 f_2 - g_2 = \pd[1] H_2 + H_2 \pd[2]$.
	
	\begin{rmk}
		We see that $g_0, g_2 \neq 0$ necessitates the use of the framework in \Cref{sec:tech}. However, if we take $T = 1$, $g_0 = g_2 = 0$ and, as we shall see, the remaining hypothesis of \cite[Lemma 4.2]{oz-sz} holds. This is what allows us to identify all the maps in the triangle with cobordism maps when $T= 1$. 
	\end{rmk}
	
	Finally, we define $H_1$ below.
	
	To define $H_1$, we cut along the hyper-surfaces $\partial Z_2$ and $(Y, K)$ on the cobordism defined by $W_1 \cup W_2$; call the collection of metrics $\tilde{Q}_t$, again represented by an edge in \Cref{fig:metrics}. This will be parametrised over an interval and the map associated to it with the bundles over them defined by the geometrics cycles described before is defined to be $H_1$. Let $X' = W_1 \cup W_2 \setminus Z_2$ and the restriction of the bundle on $W_1 \cup W_2$ defined by the geometric representatives be denoted by $\pp_{X'}$ and the restriction to $Z_2$ be $\pp_{Z_2}$. Note that this immediately implies the following:
	
	\[
	\partial_2 H_1 + H_1 \partial_0 = f_2 f_1 - m_{\partial Z_2}
	\]
	where $m_{\partial Z_2}$ refers to the map defined by the following formula: 
	\[
	\< m_{\partial Z_2} (\mathfrak{a}), \mathfrak{b}> = \# M(X', \emptyset, \pp_X; \mathfrak{a}, \mathfrak{b}) \times_{r_{\partial Z_2}} M(Z_2, \Sigma_{Z_2}, \pp_{Z_2})  
	\]
	whenever the expected dimension of the fibre-product is zero and 0 otherwise. Here $\Sigma_{Z_2}$ is the union of the core of $W_2$ and the co-core of $W_1$. 
	
	\begin{lemma}
		$m_{\partial Z_2} = 0$.
	\end{lemma}
	
	\begin{proof}
		
		Fix gauge representatives for the critical points on the relevant ends of the moduli space: let these be $a,b$ and $c$ corresponding to connections on $Y_2, Y_0$ and $\partial Z_2 \cong \RP^3$. Due to the Weitzenb\"ock formula, $h^1(c) = 0$ and since the only flat connections on $\RP^3$ with the trivial $SU(2)$ bundle have holonomy contained in the centre of $SU(2)$, we have that $c$ is stabilised by $SU(2)$, i.e., $h^0(c) = 3$. Now, let $A$ be a gauge representative of an ASD connection on $X'$ and $B$ be one on $Z_2$. The index formula gives:
		\[
		0 = \ind (A) + \ind (B) + h^0(c) + h^1(c).
		\]
		
		Now, $A$ is irreducible as the restriction to $Y_0$ is. Hence, $\ind (A) \geq 0$ and so we'd need $\ind (B) = -3$ to satisfy the above equation. However, as $\ind (B) + h^0(B) \geq 0$, this implies that $h^0(B) = 3$ in this case, i..e, $B$ is flat and defines the trivial representation on $Z_2 \setminus \Sigma_{Z_2}$. This is impossible as the holonomy along a small circle linking $\Sigma_{Z_2}$ is non-trivial. Hence, $m_{\partial Z_2} = 0$ as needed.
	\end{proof}
	
	Note that we have $g_1 = 0$ as needed.
	
	\subsection{The $G_i$ maps} 
	\label{sec:triple-composite}
	
	In the union of the cobordisms, joined end to end, $W_0 \cup W_1 \cup W_2$, we have 5 distinct hyper-surfaces from \Cref{sec:heg}: $\partial Z_1, \partial Z_2, \partial Z, Y_2, Y$. Given how they intersect, they give rise to a space of cut metrics parametrised by a pentagon, say $Q$ (see \Cref{fig:metrics}). We can define $G_0^{\pm} : C_0 \to C_0$ defined by this family along with a choice of flat connection on the distinguished boundary component, $\fo_\pm$. The pentagon has faces, denoted by say, $Q(Z_1), Q(Z_2), Q(Z), Q(Y_2)$ and, $Q(Y)$.  We now have
	\[
	\partial_0 G_0^{\pm} - G_0^{\pm} \partial_0 = m^{\pm}_{Q(Z_1)} + m^{\pm}_{Q(Z_2)} + m^{\pm}_{Q(Z)} + m^{\pm}_{Q(Y_2)} + m^{\pm}_{Q(Y)}.
	\]
	
	By the arguments of the \Cref{homotopies}, $m^{\pm}_{Q(Z_2)} = 0$ and by definition, $m^{\pm}_{Q(Y_2)} = H_1 f_0^\pm$ and $m^{\pm}_{Q(Y)} = - f_2 h_0^{\pm}$.
	
	So, we are left with 
	\[
	\partial_0 G^{\pm}_0 - G_0^{\pm} \partial_0 = m^{\pm}_{Q(Z)} + m^{\pm}_{Q(Z_1)} - ( f_2 h_0^{\pm} - H_1 f_0^\pm ).
	\]
	
	Combining both of these, we have the following:
	\[
	\partial_0 G_0 - G_0 \partial_0 = m_{Q(Z)} + m_{Q(Z_1)} - (f_2 H_0 - H_1 f_0)
	\]
	A more explicit description of the maps in this equation appear below. We first deal with $m_{Q(Z_1)}$.
	
	Let $\tilde{X} = W_0 \cup W_1 \cup W_2 \setminus Z_1$. Then, denoting the appropriate bundle restrictions by subscripts, we have the following formula:
	
	\begin{gather*}
		\< m_{Q(Z_1)} (\mathfrak{a}), \mathfrak{b}> = \# ( M_{Q( Z_1)}(\tilde{X}, C, \pp_{\tilde{X}}; \mathfrak{a}, \mathfrak{b}) \times_{r_{\partial Z_1}} M(Z_1, \Sigma \cup \{ p \}, \pp_{Z_1}; \fo_+)) - \\  \#(M_{Q( Z_1)}(\tilde{X}, C, \pp_{\tilde{X}}; \mathfrak{a}, \mathfrak{b}) \times_{r_{\partial Z_1}} M(Z_1, \Sigma \cup \{ p \}, \pp_{Z_1}; \fo_-) )
	\end{gather*}
	
	Using the same proof as we did for \Cref{prenil}, we have the proof of the following:
	
	\begin{lemma}
		$\< m_{Q(Z_1)} (\mathfrak{a}), \mathfrak{b}> = \pm (1 - T^{-1}) \# M_{Q( Z_1)}(\tilde{X}, C, \pp_{\tilde{X}}; \mathfrak{a}, \mathfrak{b}).$ \qed
	\end{lemma}
	
	We further make the following claim which follows immediately from \Cref{prop:energy-ordered-morphism} once we notice that the all finite energy connections in the moduli space above extend over the $(S^3, U)$ end to $(D^4, D^2)$ by removable singularities as in \cite[Proposition 7.7]{kmsurf1}.
	
	\begin{lemma} \label{lem:strictly-upper-triangular}
		$m_{Q(Z_1)}$ is strictly upper triangular as a matrix with entries in $\Z [T, T^{-1}]$ with respect to any energy ordered basis of critical points of the CS functional. \qed
	\end{lemma}
	
	Next, we show that $m_{Q(Z)}$ is chain-homotopic to an upper triangular matrix with respect to an energy ordered basis and a quasi-isomorphism. Then, a little more work will be needed in showing that $F_3 = m_{Q(Z)} + m_{Q(Z_1)}$ is a quasi-isomorphism as needed. We go about the first claim for now, which proves the main theorem when $T= 1$.
	
	Let $M = W_0 \cup W_1 \cup W_2 \setminus Z$. Then, $\partial M = \partial Z = S^2 \times S^1$ and $Z$ is diffeomorphic to the complement of a tubular neighbourhood of an unknotted $S^1$ (away from the cone point) inside $Th(E)$ where $Th$ denotes the Thom space and $E$ is a real rank 2 vector bundle over $S^2$ with euler class $-2$. The metric family $Q(Z)$ is constant on $M$ while is a one parameter family over $Z$; call the latter family $Q_Z$. We can again give a fibre-product description but the situation is complicated by the presence of reducibles. To discuss this, we first understand the space of flat connections on $\partial Z$.
	
	Recall that $\pp_{Z}|_{\partial Z} = \pp_{\partial Z} \cong \underline{\R^3}$. In particular it lifts to the trivial $SU(2)$ bundle; fix one such lift. With this in hand, we note that the critical points correspond to an element of $[-1, 1]$ as follows: any flat connection of $\pp_{\partial Z}$ pulls-back to the $SU(2)$ lift and the holonomy of any flat connection on the $SU(2)$ bundle corresponds to a homomorphism $\pi_1 (\partial Z) \cong \Z \to SU(2)$. Modulo the even gauge transformations, this corresponds to the conjugacy class of an element of $SU(2)$ which is classified by $h \mapsto  \tr (h) / 2 \in [-1, 1]$.
	
	Note that any configuration of connection that contributes to the count in $m_{Q(Z)}$ consists of $[A_M]$ an ASD connection on $M$, $[c]$ a gauge equivalence class of a flat connection on $\partial M$ (we are tacitly assuming that the perturbation is supported away from $\partial M$), either an ASD connection $[A_Z]$ over $Z$ or a pair of ASD connections $[B_1]$ and $[B_2]$ over $Z \setminus Z_i$ and $Z_i$ for some $i \in {1,2}$. Denote a gauge representative of the flat connection over $\partial Z_i$ as $d$ in the latter case.
	
	In the following, we shall treat the data of $B_1$, $B_2$ and $d$ as the data of $A_Z$ if the argument is similar in both cases. We shall also use $h(c) = (h^0(c) + h^1(c))/2$ and similarly for $d$.
	
	The gluing formula gives the following constraints:
	\begin{subequations}
		\begin{align}
			\ind(A_M) + \ind (A_Z) + 2h(c) &= - \dim(Q(Z)) = -1 \label{ind1} \\
			\ind(A_M) + \ind (B_1) + \ind (B_2) + 2h(c) +2h(d) &= -\dim (Q(Z)) = -1 \label{ind2}
		\end{align}
	\end{subequations}
	
	Assuming that the perturbations were chosen to be away from the finite set of hyper-surfaces above (which we can as the flat connections on them are non-degenerate with the standard metric on these hyper-surfaces), we claim the following:
	
	\begin{lemma} \label{lem:fibre-product1}
		\begin{gather*}
			\< m_{Q(Z)} (\mathfrak{a}), \mathfrak{b}> = \\ \# M(M, \emptyset, \pp_M; \mathfrak{a}, \mathfrak{b}, \fri^0)_0 \times_{r_{\partial Z}} M_{Q^0_Z}^{{red}}(Z, \Sigma \cup \{p\}, \pp_{Z}; \fri^0, \fo_+)_1  \\
			- T^{-1} \# M(M, \emptyset, \pp_M; \mathfrak{a}, \mathfrak{b}, \fri^0)_0 \times_{r_{\partial Z}}  M_{Q^0_Z}^{{red}}(Z, \Sigma \cup \{p\}, \pp_{Z}; \fri^0, \fo_-)_1 
		\end{gather*}
		where $Q^0_Z$ refers to the interior of the 1-parameter family of metrics $Q_Z$, $\fri^0$ is the interior of $\fri$, $r_{\partial Z}$ is the restriction map.
		
		This is equivalent to the following assertions:
		\begin{enumerate}[label=(\roman*)]
			\item (breaking along $Z_1$) There is no collection of ASD connections $[A_M]$, $[B_1]$ and $[B_2]$ on $M$, $Z \setminus Z_1$ and $Z_1$ respectively such that \eqref{ind2} is satisfied and contribute non-trivially to the count in $m_{Q(Z)}$.
			\item (breaking along $Z_2$) There is no collection of ASD connections $[A_M]$, $[B_1]$ and $[B_2]$ on $M$, $Z \setminus Z_2$ and $Z_2$ respectively such that \eqref{ind2} is satisfied and contribute non-trivially to the count in $m_{Q(Z)}$.
			\item ($c$ is not central) There is no pair of ASD connections $[A_M]$ and $[A_Z]$ on $M$ and $Z$ respectively with boundary limiting flat connection $[c]$ such that \eqref{ind1} is satisfied with $\dim \Stab (c) = 3$ and contribute non-trivially to the count in $m_{Q(Z)}$.
			\item ($A_Z$ is reducible) There is no pair of ASD connections $[A_M]$ and $[A_Z]$ on $M$ and $Z$ respectively such that \eqref{ind1} is satisfied with $A_Z$ irreducible and contribute non-trivially to the count in $m_{Q(Z)}$.
		\end{enumerate}
	\end{lemma}
	
	We break the proof into steps first starting with ruling out breaking along $Z_1$.
	
	\begin{proof}[Proof of \lowerromannumeral{1}] \label{proof:breaking along Z1}
		
		In this case, we first observe that $Z \setminus Z_1 \cong D^2 \times S^2 \setminus D^4$ and the singular locus is $S^2 \setminus D^2$; this implies that $[c] = 0 \in \fri^0$, i.e., the flat connection with traceless $SU(2)$ holonomy. 
		
		To see this, first note that we can glue in the unique solution on $(D^4, D^2)$ to the above to obtain a finite energy ASD connection on $(D^2 \times S^2, \{0 \} \times S^2)$. The double branched cover of this is $D^2 \times S^2$; we shall analyse finite energy ASD connections on this manifold. As it has no compactly supported self-dual forms, we see that the connection is forced to be flat. As $D^2 \times S^2$ is simply connected we see it must be the trivial flat connection. This implies that $[c] = 0 \in \fri^0$. 
		
		We can glue together the connections $B_i$ and further glue a copy of $S^1 \times D^3$ with the trivial bundle and the unique flat connection that extends $c$ on $S^1 \times S^2$; for future reference, we denote the latter connection by $A_{S^1 \times D^3}$. Call the resulting connection $A$. Then, by the index formula in \Cref{thm:ind}, we have 
		\begin{equation} \label{closed-conn}
			\ind (A) = 8 \left (k_0 + \frac{l_0}{2} + \frac{1}{8} \right) - 3 + 2 - 1 = 8 k_0 + 4 l_0 - 1 \geq -1. 
		\end{equation}
		Here, $k_0 \in \frac{1}{4} \Z, \,  l_0 \in \frac{1}{2} \Z$ (see \cite[\S 2 (iv)]{kmsurf1}), records the difference in instanton and monopole number from the original configuration and the last equality is due to the fact that $ 0 \leq \kappa (A) = k_0 + l_0 / 2 + 1/8$ combined with the integrality properties of $k_0$ and $l_0$.
		
		With the gluing formula, this implies 
		\begin{equation}\label{inter1}
			\ind (A) = \ind (A_{S^1 \times D^3}) + 2h(c) + \ind (B_1) + \ind (B_2) + 2h(d).
		\end{equation}
		Now notice that we can use the index formula for $S^1 \times S^3$ along with the gluing formula to deduce $\ind (A_{S^1 \times D^3}) = -h(c)$:
		\[
		0 = \ind (A_{S^1 \times S^3}) = \ind (A_{S^1 \times D^3}) + \ind (A_{S^1 \times D^3}) + 2h^0(c)
		\]
		where $A_{S^1 \times S^3}$ is a flat connection on the trivial bundle over $S^1 \times S^3$ and its index is computed by the index formula for the closed case.
		
		Using \eqref{inter1}, \eqref{ind2} says $\ind (A) + h(c) + \ind (A_M) = -1$. Now, $h(c) = 1$ ($c$ is not gauge equivalent to a trivial $SO(3)$ connection on $S^1 \times S^3$) and $\ind (A) \geq -1$ by \eqref{closed-conn}. Further, we have $\ind (A_M) \geq -\dim \fri^0 = -1$ for $A_M$ is irreducible. This implies that if we have the case implied by \eqref{ind2}, we must have a connection $A_M$ which restricts to a flat connection on $S^1 \times S^2$ with traceless holonomy (recall that we noted above that $[c]$ corresponds to the traceless holonomy flat connection). Since $A_M$ is irreducible with $\ind(A_M) = -1$, we may perturb the equation such that the perturbation is supported in the interior of $M$ to ensure such a connection does not exist. From now on, we assume we have chosen such a (small) perturbation in addition to the perturbations needed to ensure the necessary transversality. 
	\end{proof}
	
	Next, we treat the case of breaking along $Z_2$. 
	
	\begin{proof}[Proof of \lowerromannumeral{2}]
		In this case, we first observe that $Z \setminus Z_2$ has a single orbifold point $\{ p \} $ of order $2$ and that $\partial ( Z \setminus Z_2) \cong \RP^3 \sqcup \partial Z$. Further, note that the bundle restricted to $\partial ( Z \setminus Z_2 )$ is trivial and so is the action of the isotropy group over the fibre at the orbifold point $p$. This implies that $[c] \in \partial \fri$. 
		
		To see this, first note that by removable singularities \Cref{thm:rm-sing}, we may assume we have a finite energy ASD connection on $Z \setminus (Z_2 \cup \{ p \})$. As this manifold has no compactly supported self-dual forms, we see that the connection is forced to be flat. By the topological remarks above, the connection on the two $\RP^3$ ends has holonomy contained in the centre of $SU(2)$ and so are trivial as $SO(3)$ connections. Now, the fundamental group of $Z \setminus (Z_2 \cup \{ p \})$ is trivial and hence, we conclude $[c] \in \partial \fri$.
		
		We can glue together the connections $B_i$ and further glue a copy of $S^1 \times D^3$ with the trivial bundle and the unique flat connection that extends $c$ on $S^1 \times S^2$; for future reference, we denote the latter connection by $A_{S^1 \times D^3}$. Call the resulting connection $A$. Then, by the index formula, we have 
		\[
		\ind (A) = 8 \left (k_0 + \frac{l_0}{2} + \frac{1}{8} \right) - 3 + 2 - 1 = 8 k_0 + 4 l_0 - 1 \geq -1. 
		\]
		Here, $k_0 \in \frac{1}{4} \Z, \,  l_0 \in \frac{1}{2} \Z$ (see \cite[\S 2 (iv)]{kmsurf1}), records the difference in instanton and monopole number from the original configuration and the last equality is due to the fact that $ 0 \leq \kappa (A) = k_0 + l_0 / 2 + 1/8$ combined with the integrality properties of $k_0$ and $l_0$.
		
		With the gluing formula, this implies 
		\begin{equation}\label{inter}
			\ind (A) = \ind (A_{S^1 \times D^3}) + 2h(c) + \ind (B_1) + \ind (B_2) + 2h(d).
		\end{equation}
		Now notice that we can use the index formula for $S^1 \times S^3$ along with the gluing formula to deduce $\ind (A_{S^1 \times D^3}) = -h(c)$:
		\[
		0 = \ind (A_{S^1 \times S^3}) = \ind (A_{S^1 \times D^3}) + \ind (A_{S^1 \times D^3}) + 2h^0(c)
		\]
		where $A_{S^1 \times S^3}$ is the trivial connection on the trivial bundle over $S^1 \times S^3$ and its index is computed by the index formula for the closed case.
		
		Using \eqref{inter}, \eqref{ind2} says $\ind (A) + h(c) + \ind (A_M) = -1$. Now, $h(c) = 3$ ($c$ is gauge equivalent to the trivial $SO(3)$ connection on the $S^1 \times S^3$) and $\ind (A) \geq -1$. Further, we have $\ind (A_M) \geq -\dim \fri^0 = -1$ for $A_M$ is irreducible, which achieves the desired contradiction.
	\end{proof}
	
	We next rule out $[c] \in \partial \fri$. 
	
	\begin{proof}[Proof of  \lowerromannumeral{3}]
		In this case, $h(c) = 3$ and, analogous to before $\ind (A_M) \geq - \dim \fri = -1$. The index computation \eqref{ind1} forces $\ind (A) + 3 + \ind (A_M) = -1$ where $A$ is the connection constructed by gluing a $S^1 \times D^3$ to $A_Z$ analogous to above. Now, as $\ind (A) \geq -\dim U(1) - \dim Q(Z) = -2$ as it is stabilised at worst by $U(1)$, we have $\ind (A) + \ind (A_M) \geq -3$ which is a contradiction.
	\end{proof}
	
	Finally, we rule out that $A_Z$ is irreducible. 
	
	\begin{proof}[Proof of \lowerromannumeral{4}]
		If this were so, then the gluing parameter would be $\Stab (c) \cong U(1)$ (recall we have shown $[c] \in \fri^0$). As $A_M$ is irreducible, each pair of connections would lead to a $U(1)$ family of glued connections in the moduli space. This would contradict that only isolated points are present in the (parametrised) moduli space.
	\end{proof}
	
	To sum up, we are looking for the following configuration of connections:
	
	\[
	[c] \in \fri^0 \setminus \{ 0 \}, \; q \in Q_Z^0, \; \ind (A_M) = -1, \; \ind (A_Z) = -2 
	\]
	
	The restriction of $[c] \neq 0$, achieved by an appropriate perturbation on $M$, is to ensure there are no contributions from $q \in \partial Q_Z$ (see the proof of \lowerromannumeral{1} above).
	
	\begin{lemma}
		The projection maps $M^{red}_{Q_Z^0} (Z, \Sigma \cup \{p\}, \pp_Z; \fri^0, \fo_\pm)_1 \to Q_Z^0$ are smooth homeomorphisms.
	\end{lemma}
	
	\begin{proof}
		We note that $\otimes \xi$, where $\xi$ is as in \Cref{rmk:xi}, provides a homeomorphism of the two moduli spaces and commutes with the projection map. Hence, we only need to prove the first half of the lemma.
		
		The projection map is open and continuous as the set is topologised as a subset of $\mathscr{B} \times Q_Z^0$ and will be smooth in the transverse case, which we can arrange by a suitable choice of perturbation supported away from finitely many hyper-surfaces. We just need to show the map is bijective. We note that all base manifolds involved below have trivial fundamental group so that all gauge transformations are even.
		
		Let $([A_Z], q)$ be such that $\ind (A_Z) = -2$ and $h^0(A_Z) = 1$ (since $[A_Z]$ is reducible). Then, as we can glue $A_Z$ to the flat connection on $S^1 \times D^3$ that restricts to $[c]$ on $S^1 \times S^2$ to get a reducible connection $A$ on $Th(E)$ with $\ind (A) = -1$ as dictated by the gluing formula. The index formula now implies that $\kappa (A) = 1/8$. Now, the results of \Cref{sec:moduli-computation} apply to finish the proof by showing there is a unique such $A$ determined by the flat connection near the orbifold point (we are implicitly invoking \Cref{thm:rm-sing} here).
	\end{proof}
	
	\begin{lemma} \label{lem:deg}
		Let $0 \in \fri$ denote the critical point on $S^1 \times S^2$ corresponding to holonomy conjugate to a traceless element of $SU(2)$ and $\pm 1 \in \fri$ corresponding to the holonomy being central elements of $SU(2)$. Then, the map $r_{\partial Z} : M_{Q^0_Z}^{{red}}(Z, \Sigma \cup \{p\}, \pp_{Z}; \fri^0, \fo_+)_1 \cup -M_{Q^0_Z}^{{red}}(Z, \Sigma \cup \{p\}, \pp_{Z}; \fri^0, \fo_-)_1 \to \fri \setminus \{ 0 , \pm 1\} \cong (-1, 0) \cup (0, 1) $ is a proper map with degree $\pm 1$ when restricted to each interval in the domain onto its image. Further, the degree for both of the mentioned maps are equal, i.e., either both maps are degree $+1$ or both are $-1$.
	\end{lemma}
	
	\begin{proof}
		As we know the moduli spaces are open intervals by the preceding lemma, we only need to determine where the end points go. Note that the end points of the moduli space correspond to connections broken along $\pd[] Z_i$; the relevant index formula checks are analogous to the proofs of the assertions in \Cref{lem:fibre-product1}. Recall that $Z_2$ contains the singular sphere and that $Z \setminus {Z_2 \cup \{p\}}$ is diffeomorphic to $\RP^3 \times I \setminus S^1 \times D^3$ with the $S^1 \subset \RP^3 \times I$ isotopic to $\gamma \times \{0\}$ where $\gamma$ is a simple loop generating $\pi_1 (\RP^3)$. Thus, the only ASD connections on $Z \setminus {Z_2 \cup \{p\}}$ are flat and whether they exist or not is determined by whether the ASD connection on $Z_2$ has limiting boundary flat connection agreeing with the one at $p$ or not. The flat connection near $p$ also determines the holonomy along the $S^1 \subset S^1 \times S^2 \cong \partial Z$.
		
		Next, when breaking along $Z_1$, we note that $Z \setminus Z_1$ has two boundary components $S^1 \times S^2$ and $(S^3, U)$ after capping off the latter end (cf \cite[Proposition 7.7]{kmsurf1}), we see that the space is diffeomorphic to $(S^2 \times D^2, S^2 \times \{0\})$. Only ASD connections on this are flat and so the limiting holonomy is $0 \in \fri$. And this connection remains unchanged by $\otimes \xi$ up to gauge equivalence.
		
		We conclude that both intervals on the domain map to the ones in co-domain with degree $\pm 1$. We now need to check the signs agree for the two intervals in the domain with their orientations decided by \Cref{sec:complex orientations} (see also \Cref{rmk:complex-orientations}).
		
		We first recall that the parametrised moduli space is oriented via the fibre-first convention for the following map:
		
		\begin{equation*}
			\begin{tikzcd}
				M_{q}^{red} (Z, \Sigma \cup \{p\}, \pp_{Z}; \fri^0, \fo_\pm) \arrow[r] & M_{Q^0_Z}^{{red}}(Z, \Sigma \cup \{p\}, \pp_{Z}; \fri^0, \fo_\pm)_1 \arrow[d] \\
				& Q^0_Z
			\end{tikzcd}
		\end{equation*}
		By the claim regarding orientations in \Cref{prop:key-moduli computation}, we see that the fibres have the same orientation. As $\otimes \xi$ does not affect the projection map above, $\otimes \xi$ is an orientation preserving diffeomorphism between the two arcs in the domain. However, $\otimes \xi$ acts in a orientation reversing manner on $\fri \setminus \{0\}$. The extra minus sign in the cobordism maps then makes the degree of the restriction map $r_{\partial Z}$ in the statement have degree $\pm 1$ as needed.
	\end{proof}
	
	\begin{rmk}
		We note from the above proof that $M_{Q^0_Z}^{{red}}(Z, \Sigma \cup \{p\}, \pp_{Z}; \fri^0, \fo_+)_1$ maps into $(0, 1) \subset \fri$.
	\end{rmk}
	
	We are ready to show that $m_{Q(Z)}$ is chain homotopic to a quasi-isomorphism. By the preceding lemma and \Cref{lem:fibre-product1}
	\[
	\pm \< m_{Q(Z)} (\mathfrak{a}), \mathfrak{b}> = \# M(M, \emptyset, \pp_M; \mathfrak{a}, \mathfrak{b}, \fri_+)_0 + T^{-1} \# M(M, \emptyset, \pp_M; \mathfrak{a}, \mathfrak{b}, \fri_-)_0 
	\]
	where $\fri_- = (-1, 0) \subset \fri$ and $\fri_+ = (0, 1) \subset \fri$ and the sign on the left side is independent of $\mathfrak{a}$ and $\mathfrak{b}$. We have implicitly varied the perturbation in the interior of $M$ if necessary so that we are in the generic case. We define some new maps as follows assuming we are in the generic case:
	
	\begin{equation} \label{projections}
		\<\Pi_\pm(\mathfrak{a}), \mathfrak{b} > = \# M(M, \emptyset, \pp_M; \mathfrak{a}, \mathfrak{b}, \fri_\pm)_0.
	\end{equation}
	
	Hence, 
	\[
	\pm m_{Q(Z)} = \Pi_+ + T^{-1} \Pi_-.
	\]
	
	Note that $M \cup S^1 \times D^3 \cong Y_0 \times [0, 1]$. The hyper-surface $\partial (S^1 \times D^3) \cong S^1 \times S^2$ induces a parameter of metrics that gives a homotopy from $\Id$ to $\Pi_+ + \Pi_-$; the details are provided in the proof of \cite[Lemma 5.2]{sca}.
	
	We further claim that $\Pi_\pm^2 \simeq_{\fri_\pm} \Pi_\pm + N_\pm$ where the notation is from \Cref{rmk:non-degen-transversality} and $N_\pm$ are strictly upper triangular with respect to any energy ordered basis. Assuming this, we see that 
	\begin{align*}
		\Pi_+ \Pi_- &\simeq \Pi_+ (\Id - \Pi_+) \\
		&= \Pi_+ - \Pi_+^2 \\
		&\simeq_{\fri_+} N_+.
	\end{align*}
	
	Similarly, $\Pi_- \Pi_+ \simeq_{\fri_+} N_+$. Hence, 
	\begin{align*}
		(\Pi_+ + T^{-1}\Pi_-) (\Pi_+ + T\Pi_-) &\simeq_{\fri_+} \Pi_+^2 + \Pi_-^2 + N_+(T+T^{-1})\\
		&\simeq_{\fri_-} \Pi_+^2 + \Pi_- - N_- + N_+(T + T^{-1})\\
		&\simeq_{\fri_+} \Pi_+ + \Pi_- - N_- + N_+(T + T^{-1} - 1)\\
		&\simeq \Id + N.
	\end{align*}
	where $N = -N_- + N_+(T + T^{-1} - 1)$ is strictly upper triangular with respect to any energy ordered basis.
	In conclusion, $m_{Q(Z)}$ is a quasi-isomorphism.
	
	However, we need to show that $m_{Q(Z)} + m_{Q(Z_1)}$ is a quasi-isomorphism. Equivalently, by the above, we need to show that $m_{Q(Z)} (\Pi_+ + T \Pi_-) + m_{Q(Z_1)} (\Pi_+ + T \Pi_-)$ is a quasi-isomorphism. Up to chain homotopies of the kind used above, this map is $\pm \Id \pm N + H_*(m_{Q(Z_1)} (\Pi_+ + T \Pi_-))$. By \Cref{prop:energy-ordered-non-degenerate}, the map $\Pi_+ + T\Pi_-$ is upper triangular with respect to an energy-ordered basis and with respect to this basis, $m_{Q(Z_1)}$ is strictly upper triangular by \Cref{lem:strictly-upper-triangular}. In conclusion, $m_{Q(Z_1)} (\Pi_+ + T \Pi_-)$ and $N$ are simultaneously strictly upper triangular, hence their sum (or difference) is nilpotent. Thus, $m_{Q(Z)} + m_{Q(Z_1)}$ is a quasi-isomorphism. All that remains is to prove the claim above.
	
	\begin{claim} \label{claim:proj}
		$\Pi_\pm^2 \simeq_{\fri_\pm} \Pi_\pm - N_\pm$.
	\end{claim}
	
	\begin{proof}
		We will deal with $\Pi_+$ as the proofs are identical in both cases. As $\Pi_+^2$ is chain homotopic to the map represented by gluing two copies of $M$ along $Y_0$ and $-Y_0$ with the two $S^1 \times S^2$ ends having their flat connections restricted to $\fri_+ \subset \fri$. We can now find a four-manifold $W \subset M \cup_{Y_0} M$ such that $M \cup_{Y_0} M \setminus W \cong M$, $\partial W$ is three copies of $S^1 \times S^2$ with their appropriate orientations and $W \cong S^1 \times (S^3 \setminus (B^3 \cup B^3 \cup B^3))$. Recall that by \Cref{lem:cgeo-orientation}, we have an almost complex structure on $W$ and we use this structure to orient moduli spaces on $W$ using the conventions in \Cref{sec:complex orientations}.
		
		We can now stretch along one of the components of $\partial W$ to break the manifold into two pieces, $M \cup_{Y_0} M \setminus W \cong M$ and $W$.
		
		First, we show that there is a unique ASD connection on $W$ which restricts to a given $[\theta] \in \fri$ on $\partial M$ when the other two boundary components of $W$ are allowed to any flat connections in $\fri$. By the topological description of $W$, we see that $\text{Image}(H^2(W, \partial W; \Z) \to H^2(W; \Z)) = 0$ and hence every finite energy ASD connection on $W$ is flat. Furthermore, if $F \subset \partial W$ is a connected component, $\pi_1 (F) \to \pi_1 (W)$ is an isomorphism. Hence, an ASD connection exists on $W$ only if all the flat connections on the boundary are gauge equivalent and in this case there is a unique gauge equivalence class of ASD connection. Finally, the sign for $\Pi_+$ term is due to the fact that we are orienting our moduli spaces using almost complex structures (as in \Cref{sec:complex orientations}) and the almost complex structures are compatible by construction.
		
		Next, we have additional terms due to the non-compactness of $\fri_+$ (see \Cref{rmk:non-degen-transversality}). These can be described as the ASD connections with the limiting flat connection in one of the boundary components being $0 \in \partial \fri_+$ while the other one is in the interior and the metric varies over the one-dimensional family. That this is strictly upper triangular with respect to any energy ordered basis is due to \Cref{prop:energy-ordered-non-degenerate}.
	\end{proof}
	
	The construction of $G_1$ and $F_1$ are analogous and the proofs are almost identical to the one above.
	
	We now construct $G_2$ and $F_2$.
	
	We put together the manifolds $W_2, W_0$ and $W_1$ in that order and note that we have $5$ distinct hyper-surfaces $\partial Z_1'$, $\partial Z_2'$, $\partial Z'$, $Y_0$ and, $Y_2$. These hyper-surfaces give rise to a pentagon of metrics which we shall call $Q$. We can define $G_2^{\pm} : C_2 \to C_2$ defined by this family. Let the pentagon's faces be denoted by $Q(Z_1'), Q(Z_2'), Q(Z'), Q(Y_0)$ and, $Q(Y_2)$. We then have
	\[
	\pd[2] G_2^{\pm} - G_2^{\pm} = m_{Q(Z')}^{\pm} + m_{Q(Z_1')}^{\pm} + m_{Q(Z_2')}^{\pm} + m_{Q(Y_0)}^{\pm} + m_{Q(Y_2)}^{\pm}.
	\]
	
	By definition, we have $m_{Q(Y_0)}^{\pm} = h_0^{\pm} f_2$ and $m_{Q(Y_2)}^{\pm} = -f_1 h_2^{\pm}$.
	
	So, we are left with 
	\[
	\partial_0 G^{\pm}_2 - G_2^{\pm} \partial_0 = m^{\pm}_{Q(Z')} + m^{\pm}_{Q(Z_1')} + m^{\pm}_{Q(Z_2')} - ( f_1 h_2^{\pm} - h_0^{\pm} f_2 ).
	\]
	
	Combining both of these, we have the following:
	\[
	\partial_0 G_2 - G_2 \partial_0 = m_{Q(Z')} + m_{Q(Z_1')} + m_{Q(Z_2')} - (f_1 H_2 - H_0 f_2)
	\]
	A more explicit description of these appear below. We first deal with $m_{Q(Z_1')}$ and $m_{Q(Z_2')}$.
	
	Let $\tilde{X}_1 = W_2 \cup W_0 \cup W_1 \setminus Z_1'$. Then, denoting the appropriate bundle restrictions by subscripts, we have the following formula:
	
	\begin{gather*}
		\< m_{Q(Z_1')} (\mathfrak{a}), \mathfrak{b}> = \# M_{Q( Z_1)}(\tilde{X}_1, C_1, \pp_{\tilde{X}_1}; \mathfrak{a}, \mathfrak{b}) \times_{r_{\partial Z_1'}} M(Z_1', \Sigma_{Z_1'} \cup \{ p \}, \pp_{Z_1'}; \fo_+) \\- \# M_{Q( Z_1)}(\tilde{X}_1, C_1, \pp_{\tilde{X}_1}; \mathfrak{a}, \mathfrak{b}) \times_{r_{\partial Z_1'}} M(Z_1', \Sigma_{Z_1'} \cup \{ p \}, \pp_{Z_1'}; \fo_-)
	\end{gather*}
	
	Using the same proof as we did for \Cref{prenil}, we have the following.
	
	\begin{lemma}
		$\pm \< m_{Q(Z_1')} (\mathfrak{a}), \mathfrak{b}> = (1- T^{-1}) \#  M_{Q( Z_1)}(\tilde{X}_1, C_1, \pp_{\tilde{X}_1}; \mathfrak{a}, \mathfrak{b})$. \qed 
	\end{lemma}
	
	Analogous claims hold for $m_{Q(Z_2')}$ where we replace $\tilde{X}_1$ by $\tilde{X}_2 = W_2 \cup W_0 \cup W_1 \setminus Z_2'$ and the bundles and other data appropriately.
	
	We further make the following claim which follows immediately from \Cref{prop:energy-ordered-morphism} once we notice that the all finite energy connections in the moduli spaces above extend over the $(S^3, U)$ end to $(D^4, D^2)$ by removable singularities as in \cite[Proposition 7.7]{kmsurf1}. 
	
	\begin{lemma} \label{lem:strictly-upper-triangular-1}
		$m_{Q(Z_1')} + m_{Q(Z_2')}$ is strictly upper triangular as a matrix with entries in $\Z [T, T^{-1}]$ with respect to an energy ordered basis of critical points of the CS functional. \qed
	\end{lemma}
	
	Next, we show that $m_{Q(Z')}$ is chain-homotopic to an upper triangular matrix with respect to an energy ordered basis and a quasi-isomorphism. Then, a little more work will be needed in showing that $F_2 = m_{Q(Z')} + m_{Q(Z'_1)} + m_{Q(Z'_2)}$ is a quasi-isomorphism as needed. We go about the first claim for now, which proves the main theorem when $T= 1$.
	
	Let $M = W_2 \cup W_0 \cup W_1 \setminus Z'$. Then, $\partial M = \partial Z' = S^2 \times S^1$ and $Z'$ is diffeomorphic to the complement of a tubular neighbourhood of an unknotted $S^1$ (away from the cone point) inside $Th(E)$ where $Th$ denotes the Thom space and $E$ is a real rank 2 vector bundle over $S^2$ with euler class $-2$. The metric family $Q(Z')$ is constant on $M$ while is a one parameter family over $Z'$; call the latter family $Q_{Z'}$. We can again give a fibre-product description but the situation is complicated by the presence of reducibles. To discuss this, we first understand the space of flat connections on $\partial Z'$.
	
	Note that $\pp_{Z'} |_{\partial Z'} = \pp_{\partial Z'} \cong \underline{\R^3}$. In particular it lifts to the trivial $SU(2)$ bundle; fix one such lift. With this in hand, we note that the critical points correspond to an element of $U(1)$. 
	
	First, note that $(Z', \Sigma_{Z'}) \cong (S^2 \times S^1, \{p_n, p_s\} \times S^1)$ where $p_n$ and $p_s$ are the north and south poles.
	
	Next, any flat connection of $\pp_{\partial Z'}$ pulls-back to the $SU(2)$ lift and the holonomy of any flat connection on the $SU(2)$ bundle corresponds to a homomorphism $\pi_1 (\partial Z' \setminus \partial B') \cong \Z^2 \to SU(2)$ where  two generators correspond to $\{*\} \times S^1 \subset (S^2 \setminus \{p_n, p_s\}) \times S^1$ and a $\gamma \times \{*\} \subset (S^2 \setminus \{p_n, p_s\}) \times S^1$ where $\gamma$ is a loop around $p_n$. Up to conjugation, we can fix the meridian to be the standard traceless element. Fix a maximal torus of $SU(2)$ that passes through this element, then the other generator is an element of this circle. This shows the space is isomorphic to $U(1)$. The key difference from the situation of $S^2 \times S^1$ is that, if there are two elements of this circle are conjugate in $SU(2)$, then the element that realises this conjugation does not commute with the holonomy along the meridian. We note that the flip symmetry acts on this space and it acts as $z \mapsto \overline{z} = 1/z$. Denote this representation variety by $\fv$; we will identify with $U(1)$ as above given by holonomy along the $\{ * \} \times S^1$ with a fixed orientation (this orientation will not matter for any other purpose).
	
	Note that any configuration of connection that contributes to the count in $m_{Q(Z')}$ consists of $[A_M]$ an ASD connection on $M$, $[c]$ a gauge equivalence class of a flat connection on $(\partial Z', \partial B') \cong (S^2 \times S^1, \{p_n, p_s\} \times S^1)$ (we are tacitly assuming that the perturbation is supported away from $\partial Z'$), either an ASD connection $[A_{Z'}]$ over $Z'$ or a pair of ASD connections $[B_1]$ and $[B_2]$ over $Z' \setminus Z_i'$ and $Z_i'$ for some $i \in \{1,2\}$. Denote a gauge representative of the flat connection over $\partial Z_i'$ as $d$ in the latter case.
	
	In the following, we shall treat the data of $B_1$, $B_2$ and $d$ as the data of $A_{Z'}$ if the argument is similar in both cases. We shall also use $h(c) = (h^0(c) + h^1(c))/2$ and similarly for $d$.
	
	To prove this, we need to rule out all possible configurations of connections that don't appear in the fibre-product above subject to the following constraints:
	\begin{subequations}
		\begin{align}
			\ind(A_M) + \ind (A_{Z'}) + 2h(c) &= - \dim(Q(Z')) = -1 \label{ind11} \\
			\ind(A_M) + \ind (B_1) + \ind (B_2) + 2h(c) +2h(d) &= -\dim (Q(Z')) = -1 \label{ind21}
		\end{align}
	\end{subequations}
	
	\begin{lemma} \label{lem:fibre-product-1}
		\begin{gather*}
			\< m_{Q(Z')} (\mathfrak{a}), \mathfrak{b}> = \# M(M, C, \pp_M; \mathfrak{a}, \mathfrak{b}, \fv)_0 \times_{r_{\partial Z'}} M_{Q^0_{Z'}}^{{red}}(Z', \Sigma_{Z'}, \pp_{Z'}; \fv, \fo_+)_1 \\ - T^{-1} \# M(M, C, \pp_M; \mathfrak{a}, \mathfrak{b}, \fv)_0 \times_{r_{\partial Z'}} M_{Q^0_{Z'}}^{{red}}(Z', \Sigma_{Z'}, \pp_{Z'}; \fv, \fo_-)_1   
		\end{gather*}
		where $Q^0_Z$ refers to the interior of the 1-parameter family of metrics $Q_Z$, $r_{\partial Z'}$ is the restriction map, $C$ is a disjoint union of two properly embedded annuli.
		
		This is equivalent to the following assertions:
		\begin{enumerate}[label=(\roman*)]
			\item (breaking along $Z'_1$) There is no collection of ASD connections $[A_M]$, $[B_1]$ and $[B_2]$ on $M$, $Z' \setminus Z'_1$ and $Z'_1$ respectively such that \eqref{ind21} is satisfied and contribute non-trivially to the count in $m_{Q(Z')}$.
			\item (breaking along $Z'_2$) There is no collection of ASD connections $[A_M]$, $[B_1]$ and $[B_2]$ on $M$, $Z' \setminus Z'_2$ and $Z'_2$ respectively such that \eqref{ind21} is satisfied and contribute non-trivially to the count in $m_{Q(Z')}$.
			\item ($A_{Z'}$ is reducible) There is no pair of ASD connections $[A_M]$ and $[A_{Z'}]$ on $M$ and $Z'$ respectively such that \eqref{ind11} is satisfied with $A_Z$ irreducible and contribute non-trivially to the count in $m_{Q(Z')}$.
		\end{enumerate}
	\end{lemma}
	
	We first rule out \eqref{ind21}, i.e., we show that the metric lies in $Q^0_{Z'}$. We treat the case of breaking along $Z_1'$ first and then for $Z_2'$.
	
	\begin{proof}[Proof of \lowerromannumeral{1}]
		
		In this case, we first observe that $Z' \setminus Z_1' \cong D^2 \times S^2 \setminus D^4$ and the singular locus is a disjoint union of an annulus and a disk; this implies that $[c] = \pm 1$.
		
		To see this, first note that we can glue in the unique solution on $(D^4, D^2)$ to the above to obtain a finite energy ASD connection on $(D^2 \times S^2, D^2 \times \{p_n, p_s\})$. The double branched cover of this is $D^2 \times S^2$; we shall analyse finite energy ASD connections on this manifold. As it has no compactly supported self-dual forms, we see that the connection is forced to be flat; this statement holds true for $(D^2 \times S^2, D^2 \times \{p_n, p_s\})$ too. Now, $\pi_1 (D^2 \times S^2 \setminus (D^2 \times \{p_, p_s\})) \cong \Z$ is generated by a meridian that links a component of the singular locus $D^2 \times \{p_n, p_s\}$. Since this only generator must be traceless, $[c]$ is forced to be trivial up to $\so$ gauge transformation.
		
		We can glue together the connections $B_i$ and further glue a copy of $(S^1 \times D^3, S^1 \times [-1, 1])$ with the trivial bundle and the unique flat connection that extends $c$ on $(S^1 \times S^2, S^1 \times \{p_n, p_s\})$; for future reference, we denote the latter connection by $A_{(S^1 \times D^3, S^1 \times [-1,1])}$. Call the resulting connection $A$. Then, by the index formula, we have 
		\[
		\ind (A) = 8 \left (k_0 + \frac{l_0}{2} + \frac{1}{8} \right) - 3 + 2 - 1 = 8 k_0 + 4 l_0 - 1 \geq -1. 
		\]
		Here, $k_0 \in \frac{1}{4} \Z, \,  l_0 \in \frac{1}{2} \Z$ (see \cite[\S 2 (iv)]{kmsurf1}), records the difference in instanton and monopole number from the original configuration and the last equality is due to the fact that $ 0 \leq \kappa (A) = k_0 + l_0 / 2 + 1/8$ combined with the integrality properties of $k_0$ and $l_0$.
		
		With the gluing formula, this implies 
		\begin{equation}\label{inter11}
			\ind (A) = \ind (A_{(S^1 \times D^3, S^1 \times [-1,1])}) + 2h(c) + \ind (B_1) + \ind (B_2) + 2h(d).
		\end{equation}
		Now notice that we can use the index formula for $(S^1 \times S^3, S^1 \times S^1)$ along with the gluing formula to deduce $\ind (A_{(S^1 \times D^3, S^1 \times [-1,1])}) = -h(c)$:
		\[
		0 = \ind (A_{(S^1 \times S^3, S^1 \times S^1)}) = \ind (A_{(S^1 \times D^3, S^1 \times [-1,1])}) + \ind (A_{(S^1 \times D^3, S^1 \times [-1,1])}) + 2h^0(c)
		\]
		where $A_{(S^1 \times D^3, S^1 \times S^1)}$ is a flat connection on the trivial bundle over $(S^1 \times S^3, S^1 \times S^1)$ and its index is computed by the index formula for the closed case.
		
		Using \eqref{inter11}, \eqref{ind21} says $\ind (A) + h(c) + \ind (A_M) = -1$. Now, $h(c) = 1$ ($c$ is a flat connection on a space with non-empty singular locus) and $\ind (A) \geq -1$. Further, we have $\ind (A_M) \geq -\dim \fri^0 = -1$ for $A_M$ is irreducible. This implies that if we have the case implied by \eqref{ind21}, we must have a connection $A_M$ which restricts to a flat connection on $(S^1 \times S^2, S^1 \times \{p_n, p_s\})$ with holonomy along $S^1 \times \{*\}$ being trivial (recall that $[c] = -1$). Since $A_M$ is irreducible with $\ind(A_M) = -1$, we may perturb the equation such that the perturbation is supported in the interior of $M$ to ensure such a connection does not exist. From now on, we assume we have chosen such a (small) perturbation in addition to the perturbations needed to ensure the necessary transversality. 
	\end{proof}
	
	The above proof applies verbatim to rule out breaking along $Z_2'$ to show that \lowerromannumeral{2} holds. 
	
	Finally, we rule out that $A_{Z'}$ is irreducible.
	
	\begin{proof}[Proof of \lowerromannumeral{3}]
		If this were so, then the gluing parameter would be $\Stab (c) \cong U(1)$ (recall that the base space $c$ is defined on has non-empty singular locus). As $A_M$ is irreducible, each pair of connections would lead to a $U(1)$ family of glued connections in the moduli space. This would contradict that only isolated points are present in the (parametrised) moduli space.
	\end{proof}
	
	To sum up, we are looking for the following configuration of connections:
	
	\[
	[c] \in \fv \setminus \{-1, +1\}, \; q \in Q_{Z'}^0, \; \ind (A_M) = -1, \; \ind (A_{Z'}) = -2 
	\]
	
	The restriction of $[c] \neq -1$, achieved by an appropriate perturbation on $M$.
	
	We now analyse $M^{red}_{Q_{Z'}^0} (Z', \Sigma_{Z'}, \pp_{Z'}; \fv, \fo_\pm)_1$.
	
	\begin{lemma}
		The projection maps $M^{red}_{Q_{Z'}^0} (Z', \Sigma_{Z'}, \pp_{Z'}; \fv, \fo_\pm)_1 \to Q_{Z'}^0$ are smooth homeomorphisms.
	\end{lemma}
	
	\begin{proof}
		We note that $\otimes \xi$, where $\xi$ is as in \Cref{rmk:xi}, provides a homeomorphism of the first two moduli spaces and commutes with the projection map. Hence, we only need to prove the first half of the lemma.
		
		The projection map is open and continuous as the set is topologised as a subset of $\mathscr{B} \times Q_Z^0$ and will be smooth in the transverse case, which we can arrange by a suitable choice of perturbation supported away from finitely many hyper-surfaces. We just need to show the map is bijective. We note that all base manifolds involved below have trivial fundamental group so that all gauge transformations are even.
		
		Let $([A_{Z'}], q)$ be such that $\ind (A_{Z'}) = -2$ and $h^0(A_{Z'}) = 1$ (since $[A_{Z'}]$ is reducible). Then, as we can glue $A_{Z'}$ to the flat connection on $(S^1 \times D^3, S^1 \times [-1, 1])$ that restricts to $[c]$ on $(S^1 \times S^2, S^1 \times \{p_n, p_s\}$ to get a reducible connection $A$ on $Th(E)$ with $\ind (A) = -1$ as dictated by the gluing formula. The index formula now implies that $\kappa (A) = 1/8$. Now, the results of \Cref{sec:moduli-computation} apply to finish the proof by showing there is a unique such $A$ determined by the flat connection near the orbifold point (we are implicitly invoking \Cref{thm:rm-sing} here).
	\end{proof}
	
	\begin{lemma}
		Let $\{\pm 1\} \in \fv$ denote the critical points on $(S^1 \times S^2, S^1 \times \{p_n, p_s\})$ corresponding to holonomy along the $S^1$ factor being $\pm \Id$ in $SU(2)$. Then, the map $r_{\partial Z'} : M^{red}_{Q_{Z'}^0} (Z', \Sigma_{Z'}, \pp_{Z'}; \fv, \fo_+)_1 \cup -M^{red}_{Q_{Z'}^0} (Z', \Sigma_{Z'}, \pp_{Z'}; \fv, \fo_-)_1 \to \fv \setminus \{ \pm 1\} \cong (-1, 1) \sqcup (-1, 1) $ is a proper map with degree $\pm 1$ when restricted to each interval in the domain onto its image. Further, the degree for both of the mentioned maps are equal, i.e., either both maps are degree $+1$ or both are $-1$.
	\end{lemma}
	
	\begin{proof}
		Note that $\otimes \xi$ (as in \Cref{rmk:xi}) acts on both the co-domain and domain spaces. On the co-domain, the action extends to $\fv$ and on this, it acts as $z \mapsto \bar{z}$. We now proceed to analyse the restriction map on $ M^{red}_{Q_{Z'}^0} (Z', \Sigma_{Z'}, \pp_{Z'}; \fv, \fo_+)_1$.
		
		We now compute the holonomy along the $\gamma = S^1 \times \{*\} \subset S^1 \times S^2 \cong \partial Z'$ explicitly as $q_t \in Q_{Z'}$, for $t \in [-1,1]$, varies. Note that there is a unique solution to the ASD equation, say $[A_t]$, such that $A_t$ is abelian and has harmonic curvature form $F_{A^{\ad}_t}$. If $D$ denotes a disk that bounds $\gamma$ and is disjoint from the singular locus, we have
		\[
		\hol_{\gamma} (A^{\ad}_t) = \exp \left ( - \int_{D} F_{A^{\ad}_t} \right ) \in U(1) \subset \so.
		\]
		We now need to pick a lift to $SU(2)$ to get a point in $\fv$. Noting that at $t = -1$, we have the holonomy corresponds to $\pm 1 \in \fv$, we choose the following lift:
		\[
		\hol_{\gamma} (A_t) = \exp \left ( -i k \pi - \frac{1}{2} \int_{D} F_{A^{\ad}_t} \right ).
		\]
		Here $k = 1$ if the starting point is $-1 \in \fv$ and $k = 0$ otherwise.
		
		There are two disks $D_\pm$ such that $[D_\pm]$ are the positive generators of $H_2(Z'_\pm, \partial Z'_\pm)$ and are disjoint from the singular loci. Here $Z'_- = Z' \setminus Z'_1$ and $Z'_+ = Z' \setminus Z'_2$ corresponding to $t = \pm 1$ when the metric breaking happens; the singular loci are analogously described. Note that $Z'_\pm \cup D^4 \cong S^2 \times D^2$. As $Z'_\pm \hookrightarrow Z'$, we have $D_\pm \in H^2(Z', \partial Z')$ and $[D_+] = [D_-] = [D]$. We claim that
		\[
		\int_{D_+} F_{A^{\ad}_t} - \int_{D_+} F_{A^{\ad}_{-t}} \to \pm 2i \pi 
		\]
		as $ t \to 1$. By the description of $D_+$, we see that $\int_{D_+} F_{A^{\ad}_t} \to 0$ as $t \to 1$. Hence, we only need to show 
		\[
		\int_{D_+} \frac{i}{2 \pi} F_{A^{\ad}_{t}} \to \pm 1
		\]
		as $t \to -1$.
		
		We shall tackle this computation over the double branched cover of $Z'$ with branching locus the singular locus; call it 
		\begin{equation} \label{d-cover}
			\widetilde{Z}' \cong S^2 \times D^2 \# \overline{\CP^2}.
		\end{equation} 
		We also pick lifts of the disks $D_\pm$ to $\widetilde{D}_\pm$ and denote the lift of the connection by $\widetilde{A}^{\ad}_t$.
		
		The description of $\widetilde{D}_\pm$ can be made more explicit by choosing the diffeomorphism in \eqref{d-cover} such that $S^2 \times D^2$ factor is $\widetilde{Z}'_- \cup D^4$. Having made this choice, $\widetilde{D}_-$ is just the $\{*\} \times D^2$ and $\widetilde{D}_+$ is the proper transform of $\widetilde{D}_-$ under the blow-up map $\widetilde{Z}' \to \widetilde{Z}'_-$ where the blow-up is at a point on $\widetilde{D}_-$; denote the exceptional divisor by $E$
		
		Now, 
		\begin{align*}
			\int_{D_+} \frac{i}{2 \pi} F_{A^{\ad}_{t}} 
			&= \int_{\widetilde{D}_+} \frac{i}{2 \pi} F_{\widetilde{A}^{\ad}_{t}} \\
			&= \int_{\widetilde{D}_-} \frac{i}{2 \pi} F_{\widetilde{A}^{\ad}_{t}} - \int_{E} \frac{i}{2 \pi} F_{\widetilde{A}^{\ad}_{t}}
		\end{align*}
		and the first term in the last sum approaches $0$ as $t \to -1$ by the description of $\widetilde{D}_-$. For the second term, we first note that as $t \to -1$, the energy concentrates in a neighbourhood of $E$ (for the only ASD connections on $\widetilde{Z}'_-$ are flat) and that $\kappa(\widetilde{A}) = 2 \kappa (A) = 1/4$. This implies that
		\[
		\frac{i}{2\pi} F_{\widetilde{A}^{\ad}_{t}} \to \pm P.D.[E]
		\]
		as $t \to -1$. As 
		\[
		\int_E P.D.[E] = [E] \cdotp [E] = -1,
		\]
		this proves the claim.
		
		This shows that the restriction map on $ M^{red}_{Q_{Z'}^0} (Z', \Sigma_{Z'}, \pp_{Z'}; \fv, \fo_+)_1$ maps onto one of the components of $\fv \setminus \{\pm 1\}$ with degree $\pm 1$. To prove the statement of the lemma, we need to deal with orientations of the two components of the moduli spaces but this is analogous to the case in \Cref{lem:deg}; recall that $\otimes \xi$ acts as $z \mapsto \bar{z}$ on $\fv$ as noted in the very beginning of this proof and there is an extra minus sign as in \Cref{lem:deg}.
	\end{proof}
	
	We are ready to show that $m_{Q(Z')}$ is chain homotopic to a quasi-isomorphism. By the preceding lemma and \Cref{lem:fibre-product-1},
	\[
	\pm \< m_{Q(Z')} (\mathfrak{a}), \mathfrak{b}> = \# M(M, C, \pp_M; \mathfrak{a}, \mathfrak{b}, \fv_+)_0 + T^{-1} \# M(M, C, \pp_M; \mathfrak{a}, \mathfrak{b}, \fv_-)_0
	\]
	where $\fv_\pm \subset \fv \setminus \{ \pm 1\}$ are the two arcs and we label them so that $M^{red}_{Q_{Z'}^0} (Z', B', \pp_{Z'}; \fv, \fo_+)_1$ maps onto $\fv_+$ under the restriction map and the sign on the left side is independent of $\mathfrak{a}$ and $\mathfrak{b}$. We have implicitly varied the perturbation in the interior of $M$ if necessary so that we are in the generic case. We define some new maps as follows assuming we are in the generic case:
	
	\begin{equation} \label{projections-1}
		\<\Pi_\pm(\mathfrak{a}), \mathfrak{b} > = \# M(M, C, \pp_M; \mathfrak{a}, \mathfrak{b}, \fv_\pm)_0.
	\end{equation}
	
	Hence, 
	\[
	\pm m_{Q(Z')} = \Pi_+ + T^{-1} \Pi_-.
	\]
	
	Note that $(M, C) \cup (S^1 \times D^3, S^1 \times [-1, 1]) \cong (Y, K) \times [0, 1]$. The hyper-surface $\partial (S^1 \times D^3, S^1 \times [-1, 1]) \cong (S^1 \times S^2, S^1 \times \{p_n, p_s\})$ induces a parameter of metrics that gives a homotopy from $\Id$ to $\Pi_+ + \Pi_-$; the details are exactly as in the proof of \cite[Lemma 5.2]{sca}.
	
	We further claim that $\Pi_\pm^2 \simeq_{\fv_\pm} \Pi_\pm - N_\pm$ where the notation is from \Cref{rmk:non-degen-transversality} and $N_\pm$ are strictly upper triangular with respect to any energy ordered basis. Assuming this, we see that 
	\begin{align*}
		\Pi_+ \Pi_- &\simeq \Pi_+ (\Id - \Pi_+) \\
		&= \Pi_+ - \Pi_+^2 \\
		&\simeq_{\fv_+} N_+.
	\end{align*}
	
	Similarly, $\Pi_- \Pi_+ \simeq_{\fri_+} N_+$. Hence, 
	\begin{align*}
		(\Pi_+ + T^{-1}\Pi_-) (\Pi_+ + T\Pi_-) &\simeq_{\fv_+} \Pi_+^2 + \Pi_-^2 + N_+(T+T^{-1})\\
		&\simeq_{\fv_-} \Pi_+^2 + \Pi_- - N_- + N_+(T + T^{-1})\\
		&\simeq_{\fv_+} \Pi_+ + \Pi_- - N_- + N_+(T + T^{-1} - 1)\\
		&\simeq \Id + N.
	\end{align*}
	where $N = -N_- + N_+(T + T^{-1} - 1)$ is strictly upper triangular with respect to any energy ordered basis.
	In conclusion, $m_{Q(Z)}$ is a quasi-isomorphism.
	
	However, we need to show that $m_{Q(Z)} + m_{Q(Z_1)}$ is a quasi-isomorphism. Equivalently, by the above, we need to show that $m_{Q(Z)} (\Pi_+ + T \Pi_-) + m_{Q(Z_1)} (\Pi_+ + T \Pi_-)$ is a quasi-isomorphism. Up to chain homotopies of the kind used above, this map is $\pm \Id \pm N + m_{Q(Z_1)} (\Pi_+ + T \Pi_-)$. By \Cref{prop:energy-ordered-non-degenerate}, the map $\Pi_+ + T\Pi_-$ is upper triangular with respect to an energy-ordered basis and with respect to this basis, $m_{Q(Z_1)}$ is strictly upper triangular by \Cref{lem:strictly-upper-triangular}. In conclusion, $m_{Q(Z_1)} (\Pi_+ + T \Pi_-)$ and $N$ are simultaneously strictly upper triangular, hence their sum (or difference) is nilpotent. Thus, $m_{Q(Z)} + m_{Q(Z_1)}$ is a quasi-isomorphism. All that remains is to prove the claim above.
	
	\begin{claim}
		$\Pi_\pm^2 \simeq_{\fv_\pm} \Pi_\pm - N_\pm$.
	\end{claim}
	
	\begin{proof}
		We will deal with $\Pi_+$ as the proofs are identical in both cases. As $\Pi_+^2$ is chain homotopic to the map represented by gluing two copies of $(M, C)$ along $(Y, K)$ and $-(Y, K)$ with the two $(S^2 \times S^1, \{p_n, p_s\} \times S^1)$ ends having their flat connections restricted to $\fv_+ \subset \fv$. We can now find a four-manifold $(W, S) \subset (M, C) \cup_{(Y, K)} (M, C)$ such that $(M, C) \cup_{(Y, K)} (M, C) \setminus (W, S) \cong (M, C)$, $\partial (W, S)$ is three copies of $(S^2 \times S^1, \{p_n, p_s\} \times S^1)$ with their appropriate orientations and $(W, S) \cong (S^1 \times (S^3 \setminus (B^3 \cup B^3 \cup B^3)), S^1 \times (I \cup I \cup I))$ where $I$ denotes an interval and the three arcs connect between two different $B^3$. Recall that by \Cref{lem:cgeo-orientation}, we have an almost complex structure on $(W, S)$ and we use this structure to orient moduli spaces on $(W, S)$ using the conventions in \Cref{sec:complex orientations}.
		
		We can now stretch along one of the components of $\partial W$ to break the manifold into two pieces, $M \cup_{(Y, K)} M \setminus W \cong M$ and $W$.
		
		First, we show that there is a unique ASD connection on $W$ which restricts to a given $[\theta] \in \fri$ on $\partial M$ when the other two boundary components of $W$ are allowed to any flat connections in $\fri$. By the topological description of $W$, we see that $\text{Image}(H^2(W, \partial W; \Z) \to H^2(W; \Z)) = 0$ and hence every finite energy ASD connection on $W$ is flat. Furthermore, if $F \subset \partial W$ is a connected component, $\pi_1 (F) \to \pi_1 (W)$ is an isomorphism. Hence, an ASD connection exists on $W$ only if all the flat connections on the boundary are gauge equivalent and in this case there is a unique gauge equivalence class of ASD connection. Finally the sign for $\Pi_+$ term is due to the fact that we are orienting our moduli spaces using almost complex structures (as in \Cref{sec:complex orientations}) and the almost complex structures are compatible by construction.
		
		Next, we have additional terms due to the non-compactness of $\fri_+$. These can be described as the ASD connections with the limiting flat connection in one of the boundary components being on $\partial \fri_+$ while the other one is in the interior and the metric varies over the one-dimensional family. That this is strictly upper triangular with respect to any energy ordered basis is due to \Cref{prop:energy-ordered-non-degenerate}.
	\end{proof}
	
	\bibliographystyle{alpha}
	\bibliography{polish}
	
	\appendix
	
	\section{Time Independent Perturbations Transversality Proof}
	\label{appendix:perturbations}

	This appendix gives more details on the proof of \Cref{prop:mtransverse} closely following the exposition in \cite{fu}.
	\begin{proof}[Proof of \Cref{prop:mtransverse}]
		Note that $\bar{\Pa}$ is transverse to $0$ if every connection $A$ for $(A, \varphi) \in \Pa^{-1}(0)$ is irreducible and $\Pa$ is transverse to $0$. Hence, we will be done if we show $\Pa$ is transverse to $0$.

		Suppose $(A, \varphi) \in \Pa^{-1} (0)$ and the differential $D_{(A, \varphi)} \Pa : \lp[k, A_o]{\Lambda^1}{W} \times T\G_r \to \lp[k-1]{\Lambda^+}{W}$ is not surjective. Note that $D_{(A, \varphi)} \Pa |_{\lp[k, A_o]{\Lambda^1}{W}}$ has closed image with finite dimensional cokernel as this is just the operator $a \mapsto d_A^+ (a) + P_+(D_A \mathrm{Hol}_\pi (a))$. Hence, as  $\im D_{(A, \varphi)} \Pa \supset D_{(A, \varphi)} \Pa (\lp[k, A_o]{\Lambda^1}{W} \times \{ 0 \})$,  
		$D_{(A, \varphi)} \Pa$ has closed image with finite dimensional cokernel. Now, if the cokernel is non-zero, we can pick a $\Phi$ which is $L^2$-orthogonal to the image, \emph{a priori} $\Phi \in \lp[1-k]{\Lambda^+}{W}$. Let $\tilde{\Phi} = \varphi^*(\Phi)$ and $g = \varphi^*(g_0)$. Then, from $\Phi \perp_{L^2} D_{(A, \varphi)} (\lp[k, A_o]{\Lambda^1}{W})$ we have the following chain of equalities:
		\begin{align*}
			0 &= \int_{\check{W}} \< P_+((\varphi^{-1})^* (d_A a + D_A \mathrm{Hol}_\pi (a))), \Phi >_{g_0} \\
			&= \int_{\check{W}} \< d_A a + D_A \mathrm{Hol}_\pi (a), \varphi^*(\Phi) >_{\varphi^* g_0} \\
			&= \int_{\check{W}} \<a, d_A^* \tilde{\Phi} + (D_A \mathrm{Hol}_\pi)^* \tilde{\Phi}>_g.
		\end{align*}
		As $D_A \mathrm{Hol}_\pi (a) = 0$ if $\supp(a) \cap \supp (\pi) = \emptyset$, we conclude that $d_A^* \tilde{\Phi} = 0$ on $\check{W} \setminus (\check{S} \cup \supp (\pi))$. Now, $d_A^* : \lp[*]{\Lambda^{+_g}}{W} \to \lp[* - 1]{\Lambda^1}{W}$ is elliptic with $L^2_k$ coefficients we conclude that $\tilde{\Phi} \in L^2_{k-1}(U, \g_{P} \otimes \Lambda^{+_g})$ by standard boot strapping where $U$ is an open subset of $\check{W} \setminus (\check{S} \cup \supp (\pi))$ as in the hypothesis.
		
		Next, we exploit that $\Phi \perp_{L^2} D_{(A, \varphi)} \Pa (T\G_r)$:
		
		\begin{align*}
			0 &= \int_{\check{W}} \< P_+ \big((\varphi^{-1})^* (r^* F_A) \big), \Phi>_{g_0} \\
			&= \int_{\check{W}} \<r^* F_A, \tilde{\Phi} >_g.
		\end{align*}
		Here, $r \in T_{\Id} \G_r$. We conclude that $\<r^* F_A, \tilde{\Phi} >_g = 0$ on $\tilde{W} \setminus (\tilde{S} \cup \supp (\pi))$. Note that for any point $p$ in the latter subset, $r(p) \in \mathfrak{gl}(T_p \tilde{W})$ is arbitrary. We now appeal to the following lemma.
		
		\begin{lemma}[{\cite[Lemma 3.7]{fu}}] \label{ortho}
			Suppose $F \in \Lambda^2_-V^*\otimes W$ and $\phi \in \Lambda^2_+V^*\otimes W$ with $\<r^*F, \phi> = 0$ for every $r \in \mathfrak{gl}(V)$, then $\<F, \phi>_W \in \Lambda^2_-V^*\otimes \Lambda^2_+V^*$ vanishes. Viewing $F \in \Hom (\Lambda^2_-V^*, W)$ and $\phi \in \Hom (\Lambda^2_+V^*, W)$, the conclusion is that $\im F$ and $\im \phi$ are orthogonal. 
			\qed
		\end{lemma}
		
		Now, $\tilde{\Phi}$ satisfies the elliptic equation $d_A^* \tilde{\Phi} = 0$ on $\check{W} \setminus (\check{S} \cup \supp (\pi))$ and $\tilde{\Phi}$ is not identically $0$ on it. By unique continuation principle, \cite{uni}, we have that $\tilde{\Phi}$ does not vanish identically on $U$. Similarly, $F_A$ satisfies $d_A^* F_A = \pm * d_A * F_A = \mp * d_A F_A = 0$ by the second Bianchi identity and so $F_A$ also does not vanish identically on $U$. 
		
		Pick any point $p \in U$ and suppose that $\tilde{\Phi} (p) \neq 0$ and $F_A$ has rank $2$ (it can not have rank $3 = \dim (\g_{\check{P}})_p$ due to \Cref{ortho} above). In this case, in a small neighbourhood of $p$, rank of $F_A$ is constantly $2$ and so $\tilde{\Phi} = \alpha \otimes u$ on this set where $\alpha$ and $u$ are functions taking values in $\Lambda^{+_g}$ and $\g_{\check{P}}$ respectively with $|u| = 1$. As $\tilde{\Phi}$ is self dual and satisfies $d_A^* \tilde{\Phi} = 0$, we have $d_A \tilde{\Phi} = 0$. This leads to the following equalities:
		
		\begin{align*}
			0 &= d_A \tilde{\Phi} = d \alpha \otimes u + \alpha \wedge d_A u \\
			\implies d \alpha &= - \alpha \wedge \<d_A u, u>  \\ 
			&= - \frac{1}{2} \alpha \wedge d \lvert u \rvert^2 = 0 \\
			\implies 0 &= \alpha \wedge d_A u  
		\end{align*}
		
		By \cite[Lemma 3.5]{fu}, we can choose a co-frame $\theta^i$ at the point $p$ so that 
		\[
		\alpha = \frac{|\alpha|}{2} (\theta^1 \wedge \theta^2 + \theta^3 \wedge \theta^4).
		\]
		Then, if 
		\[
		d_A u = \sum_i \theta^i \otimes w_i, \quad w_i \in \g_{\check{P}},
		\]
		we have 
		\begin{align*}
			0 &= \alpha \wedge d_A u \\
			&=  \frac{|\alpha|}{2} (\theta^1 \wedge \theta^2 + \theta^3 \wedge \theta^4) \wedge (\sum_i \theta^i \otimes w_i) \\
			&=  \frac{|\alpha|}{2} (\theta^1 \wedge \theta^2 \wedge \theta^3 \otimes w_3 + \theta^1 \wedge \theta^2 \wedge \theta^4 \otimes w_4 + \theta^3 \wedge \theta^4 \wedge \theta^1 \otimes w_1 + \theta^3 \wedge \theta^4 \wedge \theta^2 \otimes w_2).
		\end{align*}
		Thus, $w_i = 0$ and so $(d_A u)(p) = 0$. Additionally, \Cref{ortho} shows that $\<F_A, u>_{\g_{\check{P}}} = 0$. Hence, as $[F_A, u] = d_A(d_A u) = 0$, we have $u(p) = 0$ by \cite[Lemma 3.12]{fu} which contradicts that $\tilde{\Phi} (p) \neq 0$. Thus, the set of points where $F_A$ has rank at least $2$, an open set, is contained in $\{ \tilde{\Phi} = 0 \}$ which has empty interior by \cite{uni}. Hence, we conclude that in fact $F_A$ has rank $1$ on $U$.
		
		We consider the subset $\{p \in U \, | \, F_A (p) \neq 0\}$. On this set, the same reasoning as in the preceding paragraph, shows that $F_A = \sigma \otimes u$ for some functions $\sigma$ and $u$ taking values in $\Lambda^-$ and $\g_{\check{P}}$ respectively with $d_A u = 0$.
		
		Next, we claim that $U \setminus \{F_A = 0\}$ is connected. Then, we can extend $u$ to be defined on all of $U$ due to $d_A u = 0$. Suppose  $U \setminus \{F_A = 0\}$ is disconnected with one of the components being $U_1$, then we can define a anti-self dual 2-form $\Psi$ on $U$ such that $\Psi |_{U_1} = F_A$ and $\Psi|_{U \setminus  U_1} = 0$. An integration by parts shows that $\Psi$ defines a distribution on $U$ with first derivative a $L^2$ function and so, $\Psi \in L^2_1 (U, \Lambda^- \otimes \g_{\check{P}})$. Now, $d_A \Psi \in L^2(U, \Lambda^- \otimes \g_{\check{P}})$ vanishes on $U_1$ and $U \setminus (U_1 \cup \{F_A = 0\})$. As the union of the latter subsets is dense in $U$, we conclude $d_A \Psi = 0$. Ellipticity of $d_A : L^2_1(U, \Lambda^- \otimes \g_{\check{P}}) \to L^2(U, \Lambda^3 \otimes \g_{\check{P}})$ shows that $\Psi$ is $C^2$ (we need $k \gg 0$ here) after which since $\Psi$ vanishes on an open set the unique continuation principle \cite{uni} is violated.
		
		Hence, we finally have a $u \in L^2(U, \g_{\check{P}})$ with $d_A u = 0$ and $|u| = 1$. This contradicts that $A|_U$ is irreducible. 
	\end{proof}

\end{document}